\documentclass[pdflatex,sn-mathphys-num,12pt]{sn-jnl}

\usepackage{xcolor}
\usepackage{array}
\usepackage{graphicx}%
\usepackage{multirow}%
\usepackage{amsmath,amssymb,amsfonts}%
\usepackage{amsthm}%
\usepackage{mathrsfs}%
\usepackage[title]{appendix}%
\usepackage{xcolor}%
\usepackage{textcomp}%
\usepackage{manyfoot}%
\usepackage{booktabs}%
\usepackage{algorithm}%
\usepackage{algorithmicx}%
\usepackage{algpseudocode}%
\usepackage{listings}%

\theoremstyle{thmstyleone}%
\newtheorem{theorem}{Theorem}
\newtheorem{proposition}{Proposition}

\theoremstyle{thmstyletwo}%
\newtheorem{remark}{Remark}%

\theoremstyle{thmstylethree}%
\newtheorem{definition}{Definition}%

\usepackage{geometry}
\usepackage{enumerate}
\usepackage{amssymb,mathrsfs,amsthm,multirow}
\usepackage{extarrows}
\usepackage{mathtools}
\usepackage{xcolor}
\usepackage{graphicx}
\usepackage{subfigure}
\usepackage{chngcntr}
\usepackage[inline]{enumitem}
\usepackage{amsmath,cases,url}
\usepackage{bm}

\newcommand{\dd}{\mathsf {d\kern -0.07em l}} 
\newtheorem{lemma}{Lemma}
\newtheorem{corollary}{Corollary}

\newtheorem{assumption}{Assumption}
\usepackage{color,colortbl,xcolor}
\usepackage{hyperref}
\usepackage{adjustbox}

\begin{document}

\title[Eliciting Von Neumann–Morgenstern utility from discrete choices with response error ]{Eliciting Von Neumann–Morgenstern utility from discrete choices with response error }


\author[1]{\fnm{Bo} \sur{Chen}}\email{chenbo1211@stu.xjtu.edu.cn}

\author*[1]{\fnm{Jia} \sur{Liu}}\email{jialiu@xjtu.edu.cn}


\affil[1]{\orgdiv{School of Mathematics and Statistics}, \orgname{Xi'an Jiaotong University}, 
\city{ Xi'an}, 
\state{Shaanxi}, 
\country{ 
China}}




\abstract{
We develop a preference elicitation method for {\color{black}a} Von Neumann-Morgenstern (VNM)-type decision-maker 
{\color{black}from} pairwise comparison data in the presence of response errors.
We apply the maximum likelihood estimation (MLE) method to {\color{black}jointly} elicit the 
non-parametric 
systematic VNM utility function
and the scale parameter of the response error, assuming a Gumbel distribution. 
We {\color{black}incorporate 
structural} preference information known 
in advance 
about the decision-maker's
risk {\color{black}attitude}
through 
{\color{black}linear constraints}
on the 
utility function, 
{\color{black}including} monotonicity, concavity, {\color{black}and} Lipschitz continuity.
{\color{black}Under discretely distributed lotteries}, the resulting MLE problem can be reformulated as a convex program.
We derive {\color{black}finite-sample}
error
{\color{black}bounds} between the MLE 
and the true parameters, and establish quantitative convergence of the MLE-based VNM utility function to the true utility function in the sense of the Kolmogorov distance under 
{\color{black}
some conditions on the lotteries. These conditions may have potential applications in the design of efficient lotteries for preference elicitation.}
We {\color{black}further show}
that the optimization problem maximizing the expected MLE-based VNM utility is robust against the response error and estimation error in a probabilistic sense. 
{\color{black}
Numerical experiments in a portfolio optimization application illustrate and support the theoretical results.
}


}





\maketitle


\section{Introduction}


Discrete choices arise when a decision-maker 
(DM) selects from a finite set of alternatives and are typically observed through pairwise or $k$-wise comparisons or ratings.   
{\color{black}Among these, pairwise comparisons are the most common and} 
play an important role in understanding and modeling individual preferences in behavioral economics and marketing. 
Such pairwise comparison data are widely used to elicit, for example, a reward function in Large Language {\color{black}Models}~\citep{sun2025rethinking}, an ordinal utility function in recommendation systems \citep{fan2025ranking}, or a cardinal utility function in financial decision-making {\color{black}problems} \citep{chen2026robo}, all of which are well-known quantitative representations of a DM's preference. 
In practice, {\color{black}however,}
information about 
{\color{black}the DM's}
preferences is often incomplete; thus, preference elicitation 
is 
a crucial step toward learning the 
decision-making process. 



A major line of research on preference elicitation assumes that the DM provides consistent pairwise comparison {\color{black}responses} (i.e., 
making the same choice among the alternatives in identical situations) and that there exists a real-valued utility function representing the DM's true preferences {\color{black}and hence observed choices}. 
{\color{black}Under this assumption, consistent pairwise comparison data can be used to construct polyhedral ambiguity sets. 
These sets can then be exploited to generate new queries to shrink the ambiguity set, or to support preference robust optimization (PRO) for decision making based on the worst-case utility function within the ambiguity set.}
{\color{black}From the perspective of questionnaire generation, \citet{toubia2004polyhedral,toubia2003fast} propose a 
class
of polyhedral methods 
that leverage geometric properties to generate informative queries that more effectively elicit a}
multi-attribute linear utility function. 
{\color{black}\citet{zhang2025modified} propose a modified polyhedral method to elicit a DM’s nonlinear univariate utility function, which does not rely on explicit information about the functional shape
or the inflection point of the utility function.}
{\color{black}\citet{armbruster2015decision}, as well as \citet{hu2015robust}, are among the first
to incorporate pairwise comparison data into PRO under a non-parametric Von Neumann-Morgenstern (VNM) utility framework.}
A major contribution {\color{black}of this line of work}
is to incorporate {\color{black}\emph{structural preference information} about the DM's VNM utility function} known in advance, in addition to the {\color{black}consistent} pairwise comparison data, 
in the decision-making problem. 
{\color{black}Such 
structural preference information} 
may include for example, monotonicity consistent with first-order stochastic dominance~\cite{dentcheva2007stability} and concavity consistent with second-order stochastic dominance~\cite{rudolf2008optimization}. 
{\color{black}
Numerous 
studies within the PRO framework incorporate further preference information (e.g., Lipschitz continuity and moment-type conditions) into the DM’s utility function and develop tractable reformulations of the resulting utility-maximization problems, together with theoretical analysis; see, e.g.,~\cite{chen2026robo,hu2017optimization,hu2019data,guo2024utility,hu2025distributional,wu2025bi,liu2025preference,vayanos2020robust,zhang2025modified,haskell2022preference}. 
Meanwhile, studies on robust risk optimization adopt a {\color{black}closely related
paradigm;}
see, e.g.,~\cite{delage2022shortfall,guo2022robust,zhang2021preference}.}
{\color{black}However, this line of research
pays limited attention to random utility functions and random response errors.}
{\color{black}In practice,}
the DM 
sometimes 
exhibits 
{\color{black}inconsistent choices in pairwise comparisons, even under repeated queries.}
Empirical evidence measured in a model-free manner shows that consistency in risky choices is 
{\color{black}about}
75–85\% on average \citep{hey2001does}. 
There are two common sources of {\color{black}such inconsistent choices/behavior in the DM's choices}: state-dependent utility and response error. 
In some cases, the DM's utility function is 
{\color{black}assumed}
to vary with the environment, referred to as the state {\color{black}or the contextual variable}. 
Under this assumption, the modeler may assign a set of deterministic utility functions to each state, together with a probability measure over the random states.~\footnote{
This is known as state-dependent utility functions \citep{schervish1990state}. 
{\color{black}\citet{liu2025preference} study the time consistency and tractable solution methods of multi-stage preference robust optimization problems with state-dependent or state-independent utility functions.
They also study how to elicit a 
state-dependent utility function or an ambiguity set of state-dependent utility functions from historical scoring data.}
In the binary logit framework, state-dependent utility can be 
{\color{black}viewed}
as introducing randomness into the systematic utility, which leads to the common 
mixed logit model~\citep{train2009discrete}. 
The discussion of state-dependent utility functions or mixed logit models is beyond the scope of this paper. }

In some other cases, it is believed that the DM holds a true discrete choice model or a deterministic utility function representing her/his underlying true preference. 
{\color{black}However,}
she/he might provide 
{\color{black}responses that are}
biased 
{\color{black}relative to}
the true preferences when responding to a query, known as response error \citep{mcfadden1974conditional}. 
The most 
{\color{black}widely used}
discrete choice model accounting for response error is the Bradley-Terry-Luce (BTL) model \citep{luce1959individual}, also known as the binary logit model. 
The probability that the DM chooses alternative $i$ over alternative $j$ is modeled as: 
$$
\mathbb{P}[i \text{ is chosen over } j]=\frac{e^{V_i}}{e^{V_i}+e^{V_j}},  
$$
where $V_i$ ($V_j$, resp.) denotes the preference score assigned by the DM to alternative $i$ {\color{black}($j$, resp.)} based on 
{\color{black}the}
observed deterministic attributes/outcomes.~\footnote{In the field of statistics, 
{\color{black}numerous}
studies on the top-$K$ ranking problem employ the BTL model to analyze pairwise or $k$-wise comparison data over 
a finite set of items~\cite{chen2022partial,fan2025ranking}, where statistical inference is conducted to derive asymptotic guarantees for the estimated preference scores.}
The BTL model can be 
{\color{black}derived}
from the well-known \textit{additive random utility model} {\color{black}(ARUM) in the econometric literature~\cite{mcfadden1974conditional,fosgerau2013choice,sorensen2022mcfadden,train2009discrete}.} 
{\color{black}A common practice for ARUM is to model the random utility $U_i$}
as the sum of the \textit{systematic component} $V_i$ and a Gumbel-distributed \textit{random response error}~\cite{ben1985discrete,train2009discrete}:
\begin{equation}
   U_i=V_i+{\rm Gumbel}(0,\sigma), \label{eq-btl} 
\end{equation}
and alternative $i$ is chosen over alternative $j$ if and only if $U_i > U_j$.
{\color{black}
Throughout this paper, we focus on a single-DM setting in which response errors may arise due to limited attention, mood variation, imperfect perception of the queries, or other incidental personal factors, leading to occasional inconsistencies across 
choices.~\footnote{\color{black}
Note that the additive random utility model in \eqref{eq-btl} can also be used to interpret choice behavior in a multi-DM (population) setting, where the systematic component captures the homogeneous consensus shared across the population and the error term captures heterogeneity and idiosyncratic shocks across individuals.}
The Gumbel assumption can be justified as an approximation to the normal distribution~\cite{ben1985discrete}, and it features slightly heavier tails than the normal distribution, allowing for more aberrant behavior~\cite{train2009discrete}.
}
Preference elicitation 
{\color{black}for the}
BTL model,  
{\color{black}as well as for related additive random utility models with alternative systematic components and error distributions,}
have been widely studied. 
{\color{black} Regarding}
the systematic 
{\color{black}component $V_i$}, 
most studies focus on the finite-attribute
setting~\cite{chen2022partial}, the linear 
{\color{black}specification}
$V_i=\beta^\top x_i$~\cite{saure2019ellipsoidal}, 
or a non-linear function 
with pre-specified structure~\cite{blavatskyy2007stochastic}. 
\citet{matzkin1991semiparametric} 
studies the preference elicitation of 
{\color{black}an}
additive random utility 
{\color{black}model}
with a
non-parametric systematic utility (without specifying a parametric form), where 
only structural 
{\color{black}constraints}
such as monotonicity and concavity are 
{\color{black}imposed}. 
\footnote{
\citet{matzkin1991semiparametric} develops a two-step elicitation method: 
first estimating systematic utility values 
{\color{black}via}
constrained MLE 
with an increasing and concave utility function over 
{\color{black}a finite set of}
attributes,  
and then interpolating a non-parametric
function from these values. 
This method has been extended to non-parametric estimation of additive random utility 
{\color{black}models}
with distribution-free 
error terms~\cite{matzkin1992nonparametric,briesch2010nonparametric} and 
non-additive random utility functions \citep{matzkin2003nonparametric,matzkin2019constructive}. 
}
The BTL model 
{\color{black}estimated}
from pairwise or $k$-wise comparison data has also been used to learn an agent's reward function 
{\color{black}within}
the random utility framework in Reinforcement Learning from Human Feedback (RLHF). 
\footnote{
Linear reward functions 
{\color{black}over}
state-action pairs 
with a vector parameter 
are the most widely used in RLHF 
\citep{zhu2023principled,li2023reinforcement}.  
\citet{zhan2023provable} instead 
{\color{black}employ}
parameterized neural networks to approximate 
reward functions. 
}

When the DM is presented with random alternatives (i.e., risky lotteries), the corresponding choice model involves both the 
{\color{black}outcome uncertainty arising}
from the realization of lotteries and 
the response error 
{\color{black}from}
the DM, 
{\color{black}making the model}
more complicated than 
{\color{black}in settings}
with deterministic outcomes. 
For discrete choices over risky lotteries 
{\color{black}in the absence of}
response error, the well-known Von Neumann–Morgenstern (VNM) expected utility theory
provides a flexible modeling framework for capturing the DM’s preferences over random gains and losses through a VNM utility function.~\footnote{
A DM is VNM-rational, i.e., her/his preferences satisfy four axioms (completeness, transitivity, continuity, and independence), if and only if she/he acts to maximize the 
{\color{black}expected VNM utility}
over 
risky lotteries~\cite{von1947theory}. 
}
This raises the question of whether a 
{\color{black}VNM-type utility representation} can be established for discrete choices over risky lotteries in the presence of {\color{black}
response error}.  
\citet{eliashberg1985measurement} 
introduce response error in 
{\color{black}
preference elicitation over risky lotteries}
by positing a 
{\color{black}VNM utility} function $u(x, r)$, where the {\color{black}random risk parameter} $r$ 
{\color{black}is assumed to follow}
a parametric distribution $f(r|\lambda)$,  {\color{black}thereby capturing measurement error.}
\citet{gul2006random} 
{\color{black}provide an}
axiomatic foundation 
{\color{black}for}
applying VNM expected utility theory to 
{\color{black}discrete choices under random utility}, 
showing that 
a random choice rule maximizing 
{\color{black}expected utility for}
some 
VNM utility functions if and only if it is mixture continuous, monotone, extreme, and linear.

Based on the 
{\color{black}VNM expected utility framework with randomness}, some studies have explored preference elicitation from pairwise comparisons 
{\color{black}of}
risky lotteries 
{\color{black}in the presence of}
response error. 
Analogous to the BTL model and 
{\color{black}the}
additive random utility model,  
a common approach 
adds an error term to the expected utility, 
{\color{black}treating expected utility as the}
\textit{systematic component} 
{\color{black}and the error term as the}
\textit{random component}, 
{\color{black}thereby yielding}
the DM's random utility.~\footnote{
{\color{black}In parallel with work on}
additive 
response error, Bayesian approaches 
{\color{black}have also been}
widely 
{\color{black}used}
to elicit or update the DM's random utility functions \citep{viappiani2010optimal,saure2019ellipsoidal}. 
However, most research is limited to 
{\color{black}settings with a finite number of states/outcomes}, 
since identifying a non-parametric utility function poses significant challenges when computing the posterior distribution of the random function under an infinite-dimensional Bayesian framework \citep{gu2018nonparametric}.
In this paper, we limit our discussion to the frequentist analysis of 
VNM utility 
models
with additive response error. 
}
\citet{hey1994investigating} study finite pointwise VNM utility 
{\color{black}elicitation}
for discrete choice over finitely supported lotteries 
{\color{black}by}
incorporating a standard normally distributed error term added to the expected utility. 
\citet{buschena2000generalized} 
further 
assume that the error term is heteroskedastic, with its standard deviation varying across decision problems. 
\citet{blavatskyy2007stochastic} 
develops a model in which 
the systematic utility 
{\color{black}is assumed to follow}
a pre-specified power form and the additive error term 
{\color{black}is assumed to follow}
a truncated normal distribution. 


In this paper, we continue the 
research on the preference elicitation of VNM utility 
{\color{black}functions}
from discrete choices among risky lotteries with additive response error. 
We develop a \emph{non-parametric} 
{\color{black}VNM utility model with additive response error}
$U(\cdot)=\mathbb{E}[u(\cdot)]+\varepsilon$, where the corresponding elicitation problem is formulated 
{\color{black}as a maximum likelihood estimation (MLE) problem}.
Following 
{\color{black}
\citet{armbruster2015decision}, 
\citet{hu2015robust} and \citet{guo2024utility}, we 
incorporate structural preference information}
derived from economic theory: the DM’s utility function is 
{\color{black}assumed}
to be monotonically increasing, concave, and Lipschitz continuous. We incorporate this knowledge as 
{\color{black}linear}
constraints in 
{\color{black}the elicitation}
procedure and reformulate the MLE problem 
as a convex program 
{\color{black}driven by}
pairwise comparisons of finitely supported 
lotteries. 
{\color{black}We analyze the existence and uniqueness of the MLE solution.}
Furthermore, we show that the MLE-based VNM utility 
{\color{black}achieves}
a good estimate by establishing theoretical convergence results 
{\color{black}in}
the $\ell_2$  and $\ell_{\infty}$ statistical errors relative to the DM’s true utility values 
{\color{black}at the breakpoints}. 
{\color{black}We also provide a finite-sample error bound on}
the Kolmogorov distance 
{\color{black}between the MLE estimate and}
the true utility function. 
When the elicited MLE-based VNM utility is applied to portfolio optimization, we find that, in a deterministic expected utility maximization problem, it 
{\color{black}is robust to both}
response error and estimation error.

The contributions can be highlighted as follows:
\begin{itemize}
\item 
To the best of our knowledge, this is the first 
{\color{black}work to study}
MLE within a {\color{black}\emph{non-parametric} VNM utility framework with response error}.
We further show that the MLE problem can be reformulated as a convex program based on pairwise 
comparison data 
{\color{black}over finitely supported lotteries}.
\footnote{Compared with the elicitation methods for 
VNM 
utility models 
{\color{black}with additive response error}
that pre-specify a parametric structure of the VNM 
utility function (see, e.g.,~\cite{hey1994investigating,buschena2000generalized, 
blavatskyy2007stochastic,saure2019ellipsoidal}), we develop an elicitation approach for a non-parametric VNM 
utility function {\color{black}in the systematic component}.} 
    
  \item We incorporate 
  \emph{{\color{black}structural} preference 
  information} about the DM’s VNM utility function as 
  linear constraints in the MLE formulation.
  Numerical results demonstrate its effectiveness
  compared 
  {\color{black}with}
  the purely data-driven approach {\color{black}that uses only
  binary choice data}.
  


\item 
{\color{black}We analyze the existence and uniqueness of the MLE solution and establish convergence guarantees for our elicitation approach under suitable conditions on the queries.}
\footnote{Compared with existing statistical results on preference elicitation for additive random utility functions (see, e.g.,~\cite{matzkin1991semiparametric,matzkin2019constructive,briesch2010nonparametric,zhu2023principled,li2023reinforcement}), 
we 
obtain tighter quantitative 
{\color{black}finite-sample}
error bounds for the MLE 
{\color{black}estimates}
and 
{\color{black}derive corresponding}
estimation error bounds for general non-linear utility functions.
} 
{\color{black}We further interpret the statistical results and demonstrate that the resulting finite-sample error bounds provide principled guidance for designing pairwise comparison lotteries.
}

    \item 
    {\color{black}
The MLE-based VNM utility function, when applied to a portfolio optimization problem, exhibits robustness to both response error and estimation error.
}


\end{itemize}

The 
{\color{black}remainder}
of this paper is organized as follows: 
In Section \ref{sec-VNM}, we construct a 
{\color{black}VNM utility model with response error}
and 
{\color{black}derive}
choice probabilities from pairwise comparisons. 
In Section~\ref{sec-MLE}, we perform statistical inference 
{\color{black}via}
MLE and derive a convex reformulation, followed by the 
{\color{black}existence and uniqueness analysis of the MLE solution, as well as the theoretical convergence analysis of the MLE estimates.}
In Section~\ref{sec-decision-making}, the MLE-based VNM utility function is applied to a
portfolio optimization problem. Section~\ref{sec-num-test} presents 
{\color{black}a set of}
numerical tests with a simulated DM to examine the performance of the proposed preference elicitation approach.  
Section~\ref{sec-conclude} concludes.

\section{\color{black}
VNM-type 
utility modeling with response error}\label{sec-VNM}



In the context of preference-based decision-making under uncertainty, 
we propose  
{\color{black}to use a VNM-type utility model with additive random noise to represent the DM's preferences in the presence of response error when making choices.}
In this section, we first 
{\color{black}present the DM's utility modeling with response error,}
incorporating the VNM expected utility theory, and then derive the corresponding probabilistic choice model for pairwise comparison data.

\subsection{\color{black}
The 
VNM
utility model with response error}

We begin by defining a 
set of plausible utility functions that capture the DM's 
{\color{black}underlying true}
preference, 
{\color{black}with no randomness} in the systematic component. 
Let $\mathcal{L}^p([0,\bar{b}])$ denote the set of all {\color{black}real-valued} functions $u: [0,\bar{b}]\to \mathbb{R}$ that are integrable to the $p$-th order, where $\bar{b}\in\mathbb{R}_+$ is 
a pre-specified constant representing 
the maximal outcome that the DM cares about. 
{\color{black}Let} 
$\mathcal{U}_{c}\subseteq\mathcal{L}^1([0,\bar{b}])$ 
{\color{black}denote the set of}
all monotonically increasing, concave, and Lipschitz continuous functions that are normalized to $[0,1]$, i.e., 
\begin{align*}
\mathcal{U}_{c}:= \{ u\in \mathcal{L}^1([0,\bar{b}])\mid \
&u \text{ is monotonically increasing, concave, and Lipschitz continuous }\\
&\text{with modulus bounded by $L$}, 
u(0)=0,\ u(\bar{b})=1\},
\end{align*}
which can be 
further 
{\color{black}written}
as  
\begin{align}\label{eq:U_c}
\nonumber
\mathcal{U}_{c}:= 
\{ u\in\ &  \mathcal{L}^1([0,\bar{b}])\mid \
 u^{'}_{+}(y)\geq0,\ u^{'}_{-}(y)\geq0,\ u^{'}_{+}(y)\leq u^{'}_{-}(y),\ \forall  y\in[0,\bar{b}],\\
&u^{'}_{-}(y_2)\leq u^{'}_{+}(y_1),\ \forall 0\leq y_1<y_2\leq\bar{b},\ u(0)=0,\ u(\bar{b})=1,\ \text{Lip}(u)\leq L\},
\end{align}                          
where $u^{'}_{+}(y)$ and $u^{'}_{-}(y)$ denote the right and left derivatives of $u$ at $y$, respectively; $\text{Lip}(u)\leq L$ means that $u$ is 
uniformly Lipschitz continuous with modulus ${\rm Lip}(u)$ bounded by a pre-specified constant $L>0$.

\begin{remark}\label{remark:prior info}
The constraints in \eqref{eq:U_c} encode 
some 
{\color{black}structural preference}
information about the DM's 
{\color{black}true VNM utility function}.
\begin{enumerate}[label=(\roman*)]
\item 
Under VNM expected utility theory, for any rational 
DM,  
there exists an increasing function $u$ such that, for any two lotteries $X,Y$, the DM prefers $X$ to $Y$ if and only if $\mathbb{E}[u(X)]\geq\mathbb{E}[u(Y)]$. 
The monotonicity of $u$ 
{\color{black}is consistent with the idea that}
most DMs prefer a lottery with a higher sure outcome to one with a lower sure outcome. 

\item The concavity of $u$ 
{\color{black}corresponds to}
the risk-averse attitude of the DM, 
{\color{black}who has diminishing marginal utility in outcomes.}
Under the VNM utility theory, 
the utility functions of risk-averse DMs 
{\color{black}are typically}
increasing and concave. 
A risk-averse DM prefers a certain outcome to a risky one with the same expected value.  
In economics, introducing concave utility functions for rational DMs solves the St. Petersburg paradox and remains the central pillar of modern economic theory~\cite{YUKALOV2021102537}. 
Concavity and monotonicity are also 
{\color{black}imposed}
in {\color{black}semi-parametric elicitation methods}
for 
additive random utility models, where the systematic utility for deterministic outcomes is modeled non-parametrically and the random error is modeled parametrically~\cite{matzkin1991semiparametric}.

\item 
The Lipschitz continuity of $u$ ensures bounded marginal utility, 
{\color{black}in the sense that the DM's utility}
cannot 
{\color{black}change}
too rapidly near any specific outcome~\cite{hu2015robust}.  
This condition is mild and 
is satisfied by most utility functions used in the literature, 
for example, all law-invariant coherent risk measures are globally Lipschitz continuous over random losses~\cite{inoue2003worst}. 
This condition is also related to 
{\color{black}the}
robustness of risk measures under data perturbations in data-driven problems~\cite{guo2021statistical,wang2021quantitative}. 
The Lipschitz modulus $L$ can be 
{\color{black}interpreted and}
estimated as 
the maximum cost the DM is willing to pay for a marginal unit of outcome when possessing nothing.

\item 
Given the concavity and Lipschitz continuity of $u$, its range over $[0,\bar{b}]$ is bounded by $L\bar{b}$. 
{\color{black}Such bounds are often used for normalization,}
which does not change the preference ranking between any two alternatives. 
The normalization condition also prevents the non-uniqueness caused by linear transformations~\cite{haskell2022preference}. 

\end{enumerate}

\end{remark}

Let $\mathcal{L}^p(\Omega, \mathcal{F}, \mathbb{P};[0,\bar{b}])$ denote the space of random 
{\color{black}variables}
$Y: (\Omega, \mathcal{F}, \mathbb{P}) \rightarrow [0,\bar{b}]$ with finite $p$-th moments. 
Let $V:\mathcal{L}^p(\Omega, \mathcal{F}, \mathbb{P};[0,\bar{b}])\to[0,1]$ 
map each random 
{\color{black}variable}
to a deterministic value, {\color{black}which we} refer to as the \textit{systematic component} that represents the DM’s 
{\color{black}underlying}
true utility 
{\color{black}without randomness.}
We assume that 
there exists a utility function $u^{*}\in\mathcal{U}_c$ consistent with 
VNM expected utility theory such that 
$V(Y)=\mathbb{E}[u^{*}(Y)]$ for all $Y\in\mathcal{L}^p(\Omega, \mathcal{F}, \mathbb{P};[0,\bar{b}])$; 
and for any two lotteries $X,Y\in\mathcal{L}^p(\Omega, \mathcal{F}, \mathbb{P};[0,\bar{b}])$, the DM \textit{prefers} $X$ to $Y$ if and only if $\mathbb{E}[u^{*}(X)]\geq\mathbb{E}[u^{*}(Y)]$. 
{\color{black}If}
the DM is VNM-rational, 
there must exist such a VNM utility function consistent with the DM’s preferences over all such pairwise comparisons. 

Let $\varepsilon\in\mathcal{L}^p(\Omega, \mathcal{F}, \mathbb{P};\mathbb{R})$ be a random variable measurable to the $p$-th order, {\color{black}which we} refer to as the \textit{random component} that captures the DM’s response error. 
{\color{black}Suppose}
that $\varepsilon$ is independent of any lottery $Y\in\mathcal{L}^p(\Omega, \mathcal{F}, \mathbb{P};[0,\bar{b}])$. 
{\color{black}We further assume}
that the DM's final random utility representation for 
{\color{black}lottery}
$Y$ is 
{\color{black}given by}
the sum of the systematic utility $\mathbb{E}[u(Y)]$ and the 
{\color{black}response}
error $\varepsilon$, a formulation known as additive random utility {\color{black}model (ARUM) in the econometric literature.} 

We can {\color{black}thus} formally 
{\color{black}write}
the {\color{black}VNM
utility model with response error} as follows: 
\begin{equation}\label{eq:DM-random-model}
    U(Y)=\mathbb{E}[u(Y)]+\varepsilon, \quad \forall Y\in\mathcal{L}^p(\Omega, \mathcal{F}, \mathbb{P};[0,\bar{b}]), 
\end{equation}
where $u\in\mathcal{U}_c$ is the VNM utility function, and the response error $\varepsilon$ is assumed to follow a Gumbel distribution with location parameter 0 and scale parameter $\sigma\in\mathbb{R}_{++}$, as usual in the literature~\cite{mcfadden1974conditional,ben1985discrete,train2009discrete}, 
i.e., $\varepsilon\sim \text{Gumbel}(0,\sigma)$. 
We further impose the condition $\frac{1}{\sigma}\leq \bar{c}$, 
{\color{black}where} 
$\bar{c}\in\mathbb{R}_{++}$ is a pre-specified constant; 
the rationale 
will be explained 
in Remark~\ref{remark:sigma}(iii). 
We summarize 
{\color{black}the above}
assumptions for the full
{\color{black} VNM utility model with response error}
in Assumption \ref{assump:VNM}.


\begin{assumption}\label{assump:VNM}
The DM's choice 
{\color{black}behavior}
can be represented by 
{\color{black}the VNM 
utility model with response error} 
in \eqref{eq:DM-random-model}, where the 
{\color{black}underlying}
true 
VNM utility function in the systematic component is $u^*\in\mathcal{U}_c$, and the true scale parameter of the 
{\color{black}response error}
satisfies $\sigma^*\geq\frac{1}{\bar{c}}$.  
{\color{black}Moreover, the response error $\varepsilon$ is independent of $u^*$ and of any lottery $Y\in\mathcal{L}^p(\Omega, \mathcal{F}, \mathbb{P};[0,\bar{b}])$. }
\end{assumption}


{\color{black}
\begin{remark}
It may be helpful to provide some comments on model \eqref{eq:DM-random-model} and Assumption \ref{assump:VNM}. 
\begin{itemize}

\item[(i)] Traditional discrete choice models specify a linear multivariate systematic utility component for products with observed attributes.
For example, the utility of a product with a deterministic attribute 
vector $y$ is typically modeled as $\tilde{U}(y) = c^{\top}y + \varepsilon$, where 
$c^{\top}y$ is the linear systematic component and $\varepsilon$ captures random 
errors~\cite{saure2019ellipsoidal}. 
By contrast, in \eqref{eq:DM-random-model} we consider a univariate, nonlinear, and non-parametric utility function $u$ under the VNM expected utility framework $\mathbb{E}[u(\cdot)]$, which is a generalization of the linear specifications for risky lottery comparisons.


\item[(ii)] 
The error term $\varepsilon$ captures the deviation between the observed response and 
the DM’s underlying true preference, which may vary across individuals. 
Accordingly, we estimate the scale parameter $\sigma$ of the response error as a free variable, in contrast to~\cite{walker2002generalized,saure2019ellipsoidal} who fix it at $1$. 
Smaller values of $\sigma$ correspond to more consistent choice behavior and greater 
emotional stability, while larger values indicate greater behavioral variability and 
more frequent impulsive choices.~\footnote{
\color{black}The model \eqref{eq:DM-random-model} can potentially be applied to interpret the discrete choices at the population level (multiple DMs) under specific assumptions. 
Following traditional discrete choice models, suppose that the utility of a product with deterministic attributes $y$ 
can be written as $u(y)+\varepsilon$, where 
$u(y)$ captures the homogeneous consensus shared across the
population and $\varepsilon$ represents  individual-level heterogeneity, typically assumed to be independent of $y$. 
We then extend this framework to account for random lotteries.
Under the expected utility 
criterion, the utility of a random lottery $Y$ for a given individual can be written as
$$\mathbb{E}_{\mathbb{P}_Y}\!\left[ u(Y) + \varepsilon \right]= \mathbb{E}[u(Y)] + \varepsilon,$$
where the expectation is taken with respect to $Y$ only, and the equality follows from 
the independence between $Y$ and $\varepsilon$. 
This represents the expected utility for 
a random individual (with heterogeneity $\varepsilon$) when faced with the lottery $Y$.
} 
As $\sigma\to 0$ (i.e., as the response error vanishes), the probability that the DM makes a choice inconsistent with her/his underlying true preference also tends to zero (see further discussion on this point in Remark~\ref{remark:sigma}(iii)). 



\end{itemize}
\end{remark}}




\subsection{Choice probability 
{\color{black}under}
pairwise lottery comparisons}

Suppose that we have a {\color{black}dataset} of $K$ pairwise comparisons preliminarily collected from the DM, denoted by $D=\{(W_k, Y_k, Z_k)\mid k=1,\ldots,K\}$.  
In this dataset, $(W_k, Y_k)$ denotes a pair of risky alternatives (e.g., lotteries) presented in a query to the DM to choose from, and $Z_k$ indicates the corresponding binary choice; we denote $Z_k = 1$ if the DM \textit{chooses} $W_k$, otherwise $Z_k = -1$. 
{\color{black}We consider that}
both $W_k$ and $Y_k$ are discretely distributed random variables in
$\mathcal{L}^p(\Omega, \mathcal{F}, \mathbb{P};[0,\bar{b}])$ 
{\color{black}with finite support.}
The response errors in the DM's binary choices are assumed to be 
{\color{black}independent and identically distributed (i.i.d. for short).}
We now proceed to formulate a probabilistic choice model to 
{\color{black}describe}
the DM's 
{\color{black}observed choices when comparing these lottery pairs in the presence of response error.} 

When choosing between $W_k$ and $Y_k$, the DM assigns 
{\color{black}random utility values}
$U(W_k)=\mathbb{E}\left[u(W_k)\right]+\varepsilon_1$ ($U(Y_k)=\mathbb{E}\left[u(Y_k)\right]+\varepsilon_2$, resp.) to each given lottery $W_k$ ($Y_k$, resp.), 
where $\varepsilon_1$ and $\varepsilon_2$ are 
{\color{black}i.i.d. with}
$\mathrm{Gumbel}(0, \sigma)$. 
To facilitate the analysis, we 
{\color{black}impose the following choice rule.}

\begin{assumption}\label{assump:BTL-utility}
The DM chooses $W_k$ over $Y_k$ if and only if $U(W_k)\geq U(Y_k)$. 
\end{assumption} 

By Assumption \ref{assump:BTL-utility}
{\color{black}and the VNM utility model with response error in \eqref{eq:DM-random-model}}, 
the probability {\color{black}that} the DM chooses $W_k$ over $Y_k$ is: 
\begin{equation}\label{eq:BTL-prob-origin}
\begin{aligned}
{\rm Prob}\left(U(W_k)\geq U(Y_k)\right)
&={\rm Prob} \left(\mathbb{E}\left[u(W_k)\right]+\varepsilon_{1}\geq\mathbb{E}\left[u(Y_k)\right]+\varepsilon_{2}\right)\\
&={\rm Prob} \left(\varepsilon_2\leq\varepsilon_1+\mathbb{E}\left[u(W_k)\right]-\mathbb{E}\left[u(Y_k)\right] \right).   
\end{aligned}
\end{equation}
The cumulative distribution function (CDF) of ${\rm Gumbel}(0, \sigma)$ is
\begin{equation}\label{eq:gumbel-cdf}
    F_{\varepsilon}(x) = e^{-e^{-\frac{x}{\sigma}}}, \ \forall x \in \mathbb{R},
\end{equation}
and its probability density function (PDF) is
\begin{equation*}
    f_{\varepsilon}(x) = \frac{1}{\sigma} e^{-\frac{x}{\sigma}} e^{-e^{-\frac{x}{\sigma}}}, \ \forall x \in \mathbb{R}. 
\end{equation*} 
{\color{black}Conditioning on $\varepsilon_1=x_1$,}
we have
\begin{equation*}
    {\rm Prob} \left(\varepsilon_2\leq\varepsilon_1+\mathbb{E}\left[u(W_k)\right]-\mathbb{E}\left[u(Y_k)\right]\mid \varepsilon_1=x_1 \right)= e^{-e^{-\frac{x_1+\mathbb{E}\left[u(W_k)\right]-\mathbb{E}\left[u(Y_k)\right]}{\sigma}}}.
\end{equation*}
{\color{black}The probability in \eqref{eq:BTL-prob-origin} can be further calculated} as: 
\begin{align}\label{eq:pairwise-choice-prob}
\nonumber
{\rm Prob}\left(U(W_k)\geq U(Y_k)\right)
&
= \int_{-\infty}^{+\infty} {\rm Prob}\left(\varepsilon_2\leq \varepsilon_1+\mathbb{E}\left[u(W_k)\right]-\mathbb{E}\left[u(Y_k)\right]\mid \varepsilon_1=x_1\right) f_{\varepsilon}(x_1) d x_1
\\
\nonumber
&
= \frac{1}{\sigma}\int_{-\infty}^{+\infty} e^{-e^{-\frac{x_1+\mathbb{E}\left[u(W_k)\right]-\mathbb{E}\left[u(Y_k)\right]}{\sigma}}} e^{-\frac{x_1}{\sigma}} e^{-e^{-\frac{x_1}{\sigma}}}dx_1
\\
\nonumber
&
=\frac{1}{\sigma}\int_{-\infty}^{+\infty} \exp \left\{-\left(e^{-\frac{x_1+\mathbb{E}\left[u(W_k)\right]-\mathbb{E}\left[u(Y_k)\right]}{\sigma}}+e^{-\frac{x_1}{\sigma}}\right)\right\}e^{-\frac{x_1}{\sigma}} dx_1
\\
\nonumber
&
=\frac{1}{\sigma}\int_{-\infty}^{+\infty} \exp\left\{-e^{-\frac{x_1}{\sigma}}\left(e^{-\frac{\mathbb{E}\left[u(W_k)\right]-\mathbb{E}\left[u(Y_k)\right]}{\sigma}}+1\right)\right\} e^{-\frac{x_1}{\sigma}} dx_1
\\
\nonumber
&
\xlongequal{t=e^{-\frac{x_1}{\sigma}}}\int_{+\infty}^{0}- \exp\left\{-t\left(e^{-\frac{\mathbb{E}\left[u(W_k)\right]-\mathbb{E}\left[u(Y_k)\right]}{\sigma}}+1\right)\right\} dt
\\
\nonumber
&
=\int_{0}^{+\infty}\exp\left\{-t\left(e^{-\frac{\mathbb{E}\left[u(W_k)\right]-\mathbb{E}\left[u(Y_k)\right]}{\sigma}}+1\right)\right\} dt
\\
\nonumber
&
=\frac{\exp\left\{-t\left(e^{-\frac{\mathbb{E}\left[u(W_k)\right]-\mathbb{E}\left[u(Y_k)\right]}{\sigma}}+1\right)\right\}}{-\left(e^{-\frac{\mathbb{E}\left[u(W_k)\right]-\mathbb{E}\left[u(Y_k)\right]}{\sigma}}+1\right)} \Bigg|_{0}^{+\infty}
\\
&
=\frac{1}{1+e^{-\frac{\mathbb{E}\left[u(W_k)\right]-\mathbb{E}\left[u(Y_k)\right]}{\sigma}}}=\frac{\exp(\frac{\mathbb{E}\left[u(W_k)\right]}{\sigma})}{\exp(\frac{\mathbb{E}\left[u(W_k)\right]}{\sigma})+\exp(\frac{\mathbb{E}\left[u(Y_k)\right]}{\sigma})}, 
\end{align}
which is 
analogous to the 
Bradley-Terry-Luce (BTL) model and the multinomial logit (MNL) model for pairwise comparisons; 
{\color{black}
the key difference is that the DM's utilities enter through the expected utility terms $\mathbb{E}[u(\cdot)]$.}


\begin{remark}\label{remark:sigma}
It may be helpful to make some comments on the choice probability \eqref{eq:pairwise-choice-prob}.

\begin{enumerate}[label=(\roman*)]


\item 
{\color{black}(Linear utility model with response error)}
When the outcomes of lotteries are 
{\color{black}deterministic (i.e., reducing to deterministic products)} 
and the DM has constant marginal utility for these outcomes, the expected utility $\mathbb{E}[u(\cdot)]$ framework reduces to the well-studied linear {\color{black}systematic} utility case. The linear utility function is 
{\color{black}positively homogeneous}
with respect to the scale parameter $\sigma$ {\color{black}in the sense that $\frac{u(y)}{\sigma}=u(\frac{y}{\sigma})$, so $\sigma$ can be normalized (e.g., fixed to 1) in MLE without loss of generality, without affecting estimation of the coefficient vector up to scale~\cite{train2009discrete}}.
 

\item 
{\color{black}(Adjusted VNM utility model with response error)}
{\color{black}Unlike}
the linear case, 
we consider in this paper that the DM has decreasing marginal utility, 
which {\color{black}also motivates} 
us to 
treat the scale parameter $\sigma$ as 
{\color{black}a free variable}
{\color{black}(in addition to the aforementioned incidental personal factors),} and to estimate it jointly with $u$. 
{\color{black}When $0<\sigma<\infty$, we may 
rewrite the model in an adjusted form}
$\widetilde{U}=\mathbb{E}[\tilde{u}(\cdot)]+\tilde{\varepsilon}$, {\color{black}by defining the} adjusted VNM utility function $\tilde{u}:=\frac{u}{\sigma}$ {\color{black}with a fixed left endpoint $\tilde{u}(0)=0$ and a variable right endpoint $\tilde{u}(\bar{b})=\frac{1}{\sigma}$}, and 
{\color{black}defining}
the adjusted response error $\tilde{\varepsilon}:=\frac{\varepsilon}{\sigma}\sim {\rm Gumbel}(0,1)$. {\color{black}This adjustment yields the same choice probability as in \eqref{eq:pairwise-choice-prob} from a computational perspective.~\footnote{\color{black}
We clarify that this adjustment is computationally equivalent for the MLE problem \eqref{eq:MLE-origin} when $0<\sigma<\infty$.
Moreover, $u$ and $\sigma$ are not separately identifiable from the DM's choice $Z_k$, nor can they be separately distinguished in the statistical convergence analysis in Section~\ref{sec:thm-MLE}.
However, $u$ and $\sigma$ are modeled as separate components for important reasons. First, $u$ and $\varepsilon$ are assumed to be independent, and merging them may hinder future statistical testing (beyond the scope of this paper). 
Second, this separation maintains a general modeling framework that accommodates broader data types and preference models, as the adjustment step $\tilde{u}=u/\sigma$ may not be valid for other data types (e.g., ratings), alternative risk measures (e.g., variance), or more general utility functions (e.g., S-shaped utility functions). 
Finally, from an interpretive perspective, allowing a separate variable $\sigma$ for the response error $\varepsilon$ corresponds to a DM-specific level of inconsistency in making choices. We leave these related issues for future research. 
}}

\item 
{\color{black}In \eqref{eq:pairwise-choice-prob}, as $\sigma\to0$, we have ${\rm Prob}\left(U(W_k)\geq U(Y_k)\right)\to0$ 
for any fixed nonzero utility difference $\mathbb{E}[u(W_k)]-\mathbb{E}[u(Y_k)]<0$, and ${\rm Prob}\left(U(W_k)\geq U(Y_k)\right)\to1$ 
for 
$\mathbb{E}[u(W_k)]-\mathbb{E}[u(Y_k)]>0$. 
That is, as the response error vanishes, the probability that the DM makes a choice inconsistent with her/his 
true preference also vanishes.}
If the choice probability ${\rm Prob}\left(U(W_k)\geq U(Y_k)\right)$ 
{\color{black}equals}
0 or 1, {\color{black}then
there exists a deterministic utility model that rationalizes the DM’s observed choices
by comparison,} 
{\color{black}and}
repeated 
{\color{black}responses}
to the same query by the DM 
{\color{black}would be identical.}
We do not consider this degenerate case in this paper and make a standing
assumption that $0<{\rm Prob}\left(U(W_k)\geq U(Y_k)\right)<1$ and $0<{\rm Prob}\left(U(W_k)\leq U(Y_k)\right)<1$.  
Thus, it is reasonable to require $\sigma$ 
to be bounded away from zero, i.e., {\color{black}$\sigma\geq\frac{1}{\bar{c}}$ (equivalently, $\frac{1}{\sigma}\leq\bar{c}$)}, which can also be 
{\color{black}interpreted}
as a uniform upper bound on 
the adjusted utility $\tilde{u}=\frac{u}{\sigma}$. 
{\color{black}We revisit this issue after introducing the maximum likelihood estimation in Remark \ref{remark:gamma=0}(iv).
}

\item {\color{black} 
In \eqref{eq:pairwise-choice-prob}, as $\sigma\rightarrow \infty$, we have ${\rm Prob}\left(U(W_k)\geq U(Y_k)\right)\to\frac{1}{2}$
for any fixed utility difference
$\mathbb{E}[u(W_k)] - \mathbb{E}[u(Y_k)]$.  
In this case, the DM chooses $W_k$ or $Y_k$ with equal probability, independent of her/his underlying true preferences. 
That is, the observed choice behavior is entirely dominated by the DM's response error and becomes essentially random. 
{\color{black}We revisit this issue and provide some conditions under which this case arises in Proposition \ref{prop:gamma=0} and Remark \ref{remark:gamma=0}(ii).
}

}


\end{enumerate}


\end{remark}

\section{Preference elicitation from pairwise comparison data}\label{sec-MLE}

In this section, we 
{\color{black}consider the problem of}
preference elicitation from pairwise comparison data. 
We assume that the DM's choice behavior can be described by  
{\color{black}the VNM utility model with response error}
$U(Y)=\mathbb{E}[u(Y)]+\varepsilon$, as in Assumption \ref{assump:VNM}. 
Our objective is to infer the VNM utility function $u(\cdot)$ and the distribution of the response error 
{\color{black}(parameterized by $\sigma$)}
from the observed pairwise comparison data.

We first formulate the preference elicitation task as a maximum likelihood estimation (MLE) problem 
{\color{black}posed}
over an 
{\color{black}infinite-dimensional}
function space. 
To 
{\color{black}address}
the computational challenges posed by the non-parametric structure of 
$u$, 
we show in Proposition~\ref{prop:MLE-piecewise} that when the 
{\color{black}lotteries in the pairwise comparisons}
are discretely distributed with finite support, 
the maximum of the likelihood can be attained by a piecewise linear utility function. 
Building on this result, we reformulate the original 
{\color{black}infinite-dimensional}
MLE problem as a finite-dimensional convex program, as stated in Theorem~\ref{thm:MLE-reformulation}. 
{\color{black}We then analyze the existence and uniqueness of the optimal solution to this convex program in Propositions \ref{prop:gamma=0} and \ref{prop:unique}. Furthermore, in Proposition \ref{thm:optimal-set-MLE}, we characterize the full optimal solution set in the original function space by extending the piecewise linear optimizer.}
Finally, we establish quantitative {\color{black}error bounds and} convergence {\color{black}rates for} the estimated utility 
{\color{black}relative to}
the DM’s true utility. 
{\color{black}The analysis explicitly accounts for}
both the random response error and the approximation error {\color{black}induced by} the piecewise linear representation. 

\subsection{Statistical inference by MLE}\label{sec:stat-infer}

Suppose we have collected a dataset of pairwise comparisons from the DM, 
{\color{black}denoted as}
$D=\{(W_k, Y_k, Z_k)\mid k=1,\ldots,K\}$, where $Z_k=1$ if the DM chooses $W_k$ and $Z_k=-1$ if the DM chooses $Y_k$.  
Based on the choice probability in \eqref{eq:pairwise-choice-prob}, the probability of observing 
$(W_k, Y_k, Z_k)$ 
given 
$u(\cdot)$ and $\sigma$ is 
\begin{equation*}
{\rm Prob}\left(
Z_k
\mid u,\sigma\right)
= \frac{\exp\left(\frac{\mathbb{E}[u(W_k)]}{\sigma}\right)\mathbb{I}_{(Z_k=1)}}{\exp\left(\frac{\mathbb{E}[u(W_k)]}{\sigma}\right)+\exp\left(\frac{\mathbb{E}[u(Y_k)]}{\sigma}\right)}+\frac{\exp\left(\frac{\mathbb{E}[u(Y_k)]}{\sigma}\right)\mathbb{I}_{(Z_k=-1)}}{\exp\left(\frac{\mathbb{E}[u(W_k)]}{\sigma}\right)+\exp\left(\frac{\mathbb{E}[u(Y_k)]}{\sigma}\right)}, 
\end{equation*}
for $k=1,\ldots,K$, where $\mathbb{I}_{(A)}$ denotes the indicator function, taking the value $1$ if $A$ occurs and $0$ otherwise.  
Assume that the DM's 
{\color{black}choices across}
the $K$ 
{\color{black}lottery pairs}
are independent, then the probability of observing the entire dataset $D$ {\color{black}given $u$ and $\sigma$} is: 
\begin{equation}\label{eq:likelihood-origin}
    {\rm Prob}(D\mid u,\sigma)=\prod\limits_{k=1}^K\left[\frac{\exp\left(\frac{\mathbb{E}[u(W_k)]}{\sigma}\right)\mathbb{I}_{(Z_k=1)}}{\exp\left(\frac{\mathbb{E}[u(W_k)]}{\sigma}\right)+\exp\left(\frac{\mathbb{E}[u(Y_k)]}{\sigma}\right)}+\frac{\exp\left(\frac{\mathbb{E}[u(Y_k)]}{\sigma}\right)\mathbb{I}_{(Z_k=-1)}}{\exp\left(\frac{\mathbb{E}[u(W_k)]}{\sigma}\right)+\exp\left(\frac{\mathbb{E}[u(Y_k)]}{\sigma}\right)}\right]. 
\end{equation}
Taking the logarithm of \eqref{eq:likelihood-origin} yields the log-likelihood function $\ell(u,\sigma)$, 
defined for VNM utility function $u\in\mathcal{U}_c$ and scale parameter $\sigma\geq\frac{1}{\bar{c}}$: 
\begin{align}\label{eq:BTL-log-liklihood}
\ell(u,\sigma)
&
=\sum\limits_{k=1}^K \log \left[\frac{\exp\left(\frac{\mathbb{E}[u(W_k)]}{\sigma}\right)\mathbb{I}_{(Z_k=1)}}{\exp\left(\frac{\mathbb{E}[u(W_k)]}{\sigma}\right)+\exp\left(\frac{\mathbb{E}[u(Y_k)]}{\sigma}\right)}+\frac{\exp\left(\frac{\mathbb{E}[u(Y_k)]}{\sigma}\right)\mathbb{I}_{(Z_k=-1)}}{\exp\left(\frac{\mathbb{E}[u(W_k)]}{\sigma}\right)+\exp\left(\frac{\mathbb{E}[u(Y_k)]}{\sigma}\right)}\right]
\nonumber\\
&=\sum\limits_{k=1}^K \log \left[\frac{\mathbb{I}_{(Z_k=1)}}{1+\exp\left(\frac{\mathbb{E}[u(Y_k)]-\mathbb{E}[u(W_k)]}{\sigma}\right)}+\frac{\mathbb{I}_{(Z_k=-1)}}{1+\exp\left(\frac{\mathbb{E}[u(W_k)]-\mathbb{E}[u(Y_k)]}{\sigma}\right)}\right]\nonumber\\
&{\color{black}
=\sum\limits_{k=1}^K \log \left[\frac{1}{1+\exp\left(\frac{-Z_k(\mathbb{E}[u(W_k)]-\mathbb{E}[u(Y_k)])}{\sigma}\right)}\right]}.
\end{align}

Assume that all the $K$ pairs of lotteries have finite support. 
Let $\mathcal{S}:= \{0\}\cup\bigcup\limits_{k=1}^{K} \operatorname{supp}(W_k)\cup\bigcup\limits_{k=1}^{K} \operatorname{supp}(Y_k)\cup\{\bar{b}\}$, 
and 
{\color{black}let}
$N:=|\mathcal{S}|$ 
{\color{black}denote}
the cardinality of $\mathcal{S}$. 
Let $\mathbb{Y}:= \{{y}_{j}\}_{j=1,\ldots,N}$ be the ordered sequence of all points in $\mathcal{S}$ with fixed ${y}_{1}=0$ and ${y}_{N}=\bar{b}$, 
{\color{black}which we refer to as the support set or the set of breakpoints}. 
Note that 
$\mathcal{S}$ is 
the union of the payoff supports of all 
observed lotteries. 
Hence, our framework naturally allows different lotteries to have 
different 
support points,
and any new payoff values appearing in the dataset are incorporated 
into $\mathcal{S}$ {\color{black} and $\mathbb{Y}$}. 
A richer support set enhances the identification of the entire utility function. 
This is consistent with the convergence results established later in Theorem~\ref{thm:bound-utility} and Corollary~\ref{coll:bound-utility}, where the error bound 
depends on the mesh size $\mu_N:=\max_{j=2,\ldots,N} (y_{j}-y_{j-1})$ of $\mathbb{Y}$.

{\color{black}
For each pair of lotteries $(W_k,Y_k)$, let ${p}^k:=[{p}^k_1,\ldots,{p}^k_N]^{\top}\in\mathbb{R}^N$ denote the probability-mass difference vector, where $p^k_j:=\mathbb{P}[W_k=y_j]-\mathbb{P}[Y_k=y_j]$ for $j=1,\ldots,N$, $k=1,\ldots,K$. 
Denote $p^k_{2:N}:=[p^k_2,\ldots,p^k_N]\in\mathbb{R}^{N-1}$ as the reduced probability-mass difference vector related to $p^k$, $k=1,\ldots,K$. 
Let ${P}$ be the reduced probability-mass difference matrix whose $k$-th row is $p^k_{2:N}$, i.e., ${P}=[({p}^1_{2:N})^{\top};\ldots;({p}^K_{2:N})^{\top}]^{\top}\in\mathbb{R}^{K\times (N-1)}$. 
Then the difference between expected utilities of $W_k$ and $Y_k$ is $\mathbb{E}[u(W_k)]-\mathbb{E}[u(Y_k)]=\sum_{j=1}^N p^k_j u(y_j)$ for $k=1,\ldots,K$. 
}



{\color{black}Consequently,}
for $u\in\mathcal{U}_c$ and $\sigma\geq\frac{1}{\bar{c}}$, we {\color{black}can} recast {\color{black}the log-likelihood function in} \eqref{eq:BTL-log-liklihood} as: 
\begin{equation}\label{eq:log-likelihood-re}
\ell(u,\sigma)=
-\sum_{k=1}^K 
\log \left[1+\exp{\left(\frac{\sum\nolimits_{j=1}^N 
{\color{black}-Z_kp^k_j}
u(y_j)}{\sigma}\right)}\right]. 
\end{equation}
Given dataset $D = \{(W_k, Y_k, Z_k) \mid k = 1, \ldots, K\}$, we 
{\color{black}compute}
the 
MLE 
of $u$ and $\sigma$ by maximizing the log-likelihood function $\ell(u,\sigma)$ in \eqref{eq:log-likelihood-re}, i.e.,  
\begin{equation}\label{eq:MLE-origin}
    \left(\hat{u}_{{\rm MLE}},\ \hat{\sigma}_{\rm MLE}\right) \in \mathop{\arg\max}\limits_{u\in\mathcal{U}_c,\ \sigma\geq \frac{1}{\bar{c}}} \ \ell(u,\sigma). 
\end{equation}
In the next subsection, we 
derive tractable reformulations of the MLE problem \eqref{eq:MLE-origin} over a non-parametric function space.

\subsection{Solution 
{\color{black}to}
the MLE problem \eqref{eq:MLE-origin}}

{\color{black}Since}
the log-likelihood 
$\ell(u,\sigma)$ 
{\color{black}depends}
only on the 
values of $u$ at 
{\color{black}finitely many
points in the 
{\color{black}set of breakpoints}
$\mathbb{Y}$}, 
we show in Proposition \ref{prop:MLE-piecewise} that the maximization over $\mathcal{U}_c$ can be attained by a piecewise linear function with breakpoints exactly equal to the 
set $\mathbb{Y}$.
{\color{black}Building on this,}
Theorem \ref{thm:MLE-reformulation} reformulates the MLE problem over the space of piecewise linear functions as a convex program.

\begin{definition}[Piecewise-linear Lower Approximation (PLA)]\label{def:PLU}

Given the set of breakpoints $\mathbb{Y}$, define the piecewise-linear lower approximation (PLA) operator $\mathcal{T}:\mathcal{L}^1([0,\bar{b}])\to\mathcal{L}^1([0,\bar{b}])$ such that for each $u\in\mathcal{U}_c$, 
\begin{equation*}
    \mathcal{T}u(y)
    =\left\{ 
    \begin{aligned}
    &0 && \text{\rm for}\quad  y=y_1=0, \\
    &\frac{u(y_{j+1})-u(y_j)}{y_{j+1}-y_{j}}(y-y_{j})+u(y_j) && \text{\rm for}\quad  y_{j} \leq y < y_{j+1},\ j=1,\ldots,N-1,\\
    & 1 &&\text{\rm for}\quad  y=y_{N}=\bar{b}. 
    \end{aligned}
    \right.
\end{equation*}
Let ${\mathcal U}_{N} :=\{\mathcal{T}u: u\in \mathcal{U}_c\}$ denote the set of all 
PLA
utility functions of $u\in\mathcal{U}_c$ defined as such.  
\end{definition}

By definition, 
$\mathcal{T}u$ coincides with $u$ 
{\color{black}over the breakpoints in}
$\mathbb{Y}$, i.e., $\mathcal{T}u(y_j) = u(y_j)$ for $j=1,\ldots,N$.  
We may view 
$\mathcal{U}_N$ as an approximation to $\mathcal{U}_c$.
\citet{guo2024utility} and \citet{liu2025preference} 
quantify the PLA errors by virtue of Hoffman's lemma. 
In our setting, there is no PLA error at the finitely many breakpoints in $\mathbb{Y}$, and there exists a PLA utility function that attains the maximum of MLE problem \eqref{eq:MLE-origin}. However, PLA errors 
{\color{black}enter}
the quantitative error bounds between the entire utility function estimated by MLE and the true utility function. We revisit this point in Theorem \ref{thm:bound-utility}. 

Note that $\mathcal{T}u$ preserves monotonicity, concavity, and Lipschitz continuity of $u$, so $\mathcal{T}u \in \mathcal{U}_c$ and $\mathcal{U}_N\subseteq\mathcal{U}_c$. 
Moreover, we have 
\begin{align*}
{u}(y)
&\geq(1-\frac{y-y_{j}}{y_{j+1}-y_{j}}) {u}(y_{j}) + \frac{y-y_{j}}{y_{j+1}-y_{j}} {u}(y_{j+1})\nonumber\\
&=\frac{{u}(y_{j+1})-{u}(y_{j})}{y_{j+1}-y_{j}} (y-y_j)+{u}(y_{j}) = 
\mathcal{T} u(y), 
\ \forall y\in[y_j,y_{j+1}],\ j=1,\ldots,N-1.
\end{align*}
The inequality 
{\color{black}implies}
that $\mathcal{T} u$ provides a uniform lower bound 
{\color{black}for}
$u$ over $[0,\bar{b}]$. 
Furthermore, 
by the log-likelihood function in \eqref{eq:log-likelihood-re}, 
{\color{black}for any}
$\sigma\geq\frac{1}{\bar{c}}$, we have
\begin{equation*}
    \ell(u,\sigma)=\ell(\mathcal{T}u,\sigma),\ \forall u\in\mathcal{U}_c,
\end{equation*}
which motivates the following proposition.



\begin{proposition}\label{prop:MLE-piecewise}
Given the set of breakpoints $\mathbb{Y}$, the ambiguity set of VNM utility functions $\mathcal{U}_c$, and its subset $\mathcal{U}_N$ of piecewise-linear lower approximation utility functions, the optimal values of the following two problems are equal:
\begin{equation}\label{eq-ab}
(A)\quad
   \mathop{\max}\limits_{u\in\mathcal{U}_c,\ \sigma\geq \frac{1}{\bar{c}}}\ell(u,\sigma)
\quad\Longleftrightarrow\quad
(B)\quad
   \mathop{\max}\limits_{u_N\in\mathcal{U}_N,\ \sigma\geq \frac{1}{\bar{c}}}\ell(u_N,\sigma).
\end{equation}
\end{proposition}

\begin{proof}
Denote the optimal value of $(A)$ as $d_A$ and that of $(B)$ as $d_B$.
On the one hand, $\mathcal{U}_{{N}}\subseteq\mathcal{U}_c$, we have $d_A\geq d_B$. 
On the other hand, 
let $(\hat{u}, \hat{\sigma})$ be an optimal solution of $(A)$. 
Let $\mathcal{T}\hat{u}\in\mathcal{U}_N$ be its PLA, thus $(\mathcal{T}\hat{u},\hat{\sigma})$ is a feasible solution of $(B)$. 
Then we have 
$$
d_A=\ell(\hat{u},\hat{\sigma})=\ell(\mathcal{T}\hat{u},\hat{\sigma})\leq d_B.
$$
Combining the two inequalities gives the conclusion. 
\end{proof}

{\color{black}
Notice that $\mathcal{U}_{N}\subseteq \mathcal{U}_c$ and that 
problems $(A)$ and $(B)$ in \eqref{eq-ab} have
the same optimal value.  
Therefore, the optimal solution set of $(B)$ is contained in that of $(A)$. 
The optimizer of $(A)$ need not be unique, even 
when
the optimizer of $(B)$ is unique. 
Proposition~\ref{thm:optimal-set-MLE} shows that the optimal solution of $(B)$ can be extended to recover the full optimal solution set of $(A)$, by assigning admissible values at non-breakpoint locations. 
Theorem~\ref{thm:MLE-reformulation} then provides a tractable solution method for problem $(B)$. 

}


\begin{theorem}\label{thm:MLE-reformulation}


Let 
{\color{black}$(\bar{\alpha}^*,\bar{\beta}^*,\gamma^*)$}
be the optimal solution to the following convex program 
{\color{black}
\begin{subequations}\label{eq:MLE-re2}
\begin{align}
\label{const:MLE-obj}
\max\limits_{\bar{\alpha},\bar{\beta},\gamma}&\  
-\sum_{k=1}^K 
\log \left[1+\exp{\left(\sum\nolimits_{j=1}^N 
{\color{black}-Z_k p^k_j}
\bar{\alpha}_j\right)}\right]\\
\text{\rm s.t.}\ 
\label{const:MLE-beta-alpha}
& \bar{\beta}_j=(\bar{\alpha}_{j+1}-\bar{\alpha}_j)/(y_{j+1}-y_j),\ j=1,\ldots,N-1,\\
\label{const:MLE-beta}
& \bar{\beta}_{j+1}\leq\bar{\beta}_j,\ j=1,\ldots,N-2,\\
\label{const:MLE-beta-L}
& \bar{\beta}_{1}\leq L\gamma,\ 
\bar{\beta}_{N-1}\geq0,\\
\label{const:MLE-alpha}
& \bar{\alpha}_1=0,\ \bar{\alpha}_N=\gamma,\\
\label{const:MLE-domain}
& \bar{\alpha}\in\mathbb{R}^N,\ \bar{\beta}\in\mathbb{R}^{N-1},\ \gamma\in[0,\bar{c}].
\end{align}
\end{subequations}
If $\gamma^*=0$, then the implied optimal scale parameter $\hat{\sigma}_{\rm MLE}=+\infty$, 
and there exists no VNM utility function $u\in\mathcal{U}_c$ that can 
rationalize 
the DM's binary choice data.
If $\gamma^*>0$,}
define 
$$
\hat{\alpha}^{{\rm MLE}}:=\frac{\bar{\alpha}^*}{\gamma^*},\quad
\hat{\beta}^{{\rm MLE}}:=\frac{\bar{\beta}^*}{\gamma^*},\quad
\text{and}\quad
\hat{\sigma}_{\rm MLE}:=\frac{1}{\gamma^*}.
$$
Then, the piecewise linear function $\hat{u}_{\rm MLE}$ defined as 
\begin{equation}\label{eq:N-piecewise utility function}
    \hat{u}_{\rm MLE}(y)=\left\{ 
    \begin{aligned}
    & 0 &&{\rm for}\quad y={y}_1=0,\\
    &\hat{\beta}^{\text{\rm MLE}}_j(y-{y}_j)+\hat{\alpha}^{\text{\rm MLE}}_{j} &&{\rm for}\quad {y}_j\leq y <{y}_{j+1},\ j=1,\ldots,N-1,\\
    &1 && {\rm for}\quad y= y_{N}=\bar{b},
    \end{aligned}
    \right.
\end{equation}
which we refer to as the MLE-based VNM utility function, 
{\color{black}together with $\hat{\sigma}_{\rm MLE}$, constitutes an optimal solution to} 
\eqref{eq:MLE-origin}; that is, $(\hat{u}_{\rm MLE},\hat{\sigma}_{\rm MLE})\in \mathop{\arg\max}_{u\in\mathcal{U}_c,\ \sigma\geq \frac{1}{\bar{c}}} \ \ell(u,\sigma)$, and the optimal values of problems \eqref{eq:MLE-origin} and 
\eqref{eq:MLE-re2} {\color{black}coincide}. 
\end{theorem}

\begin{proof}
By Proposition \ref{prop:MLE-piecewise},  
we can solve the optimization problem $\max\limits_{u_N\in\mathcal{U}_N,\ \sigma\geq\frac{1}{\bar{c}}}\ell(u_N,\sigma)$ instead of \eqref{eq:MLE-origin}. 
For any $u_N\in\mathcal{U}_N$, let $\alpha:=\left[\alpha_1,\ldots,\alpha_N\right]\in\mathbb{R}^N$ denote the utility values at the breakpoints in ${\mathbb{Y}}$, i.e., $\alpha_j:=u_N({y}_j)=u(y_j),\ j=1,\ldots, N$,  and let $\beta:=\left[\beta_1,\ldots,\beta_{N-1}\right]\in\mathbb{R}^{N-1}_+$ denote the slopes of linear pieces between two adjacent breakpoints, i.e., $\beta_j:=(\alpha_{j+1}-\alpha_j)/({y}_{j+1}-{y}_j),\ j=1,\ldots,N-1$. 
Then any function $u_N$ in $\mathcal{U}_{N}$ can be written as 
\begin{equation*}
    u_N(y)=\left\{ 
    \begin{aligned}
    & 0 &&{\rm for}\quad y={y}_1=0, \\
    &\beta_j(y-{y}_j)+\alpha_j &&{\rm for}\quad {y}_j\leq y <{y}_{j+1},\ j=1,\ldots,N-1,\\
    &1 &&{\rm for}\quad y= {y}_N=\bar{b}. 
    \end{aligned}
    \right.
\end{equation*}

Since $\ell(u_N,\sigma)$ only relies on the utility values  $\left[\alpha_1,\ldots,\alpha_N\right]$ at the breakpoints in $\mathbb{Y}$, we can reformulate the problem of maximizing $\ell(u_N,\sigma)$ over $\mathcal{U}_N$ {\color{black}(i.e., problem $(B)$ in \eqref{eq-ab})} as the following finite-dimensional optimization problem: 
\begin{subequations}\label{eq:MLE-re1}
\begin{align}
\label{obj:MLE-re1}
\max\limits_{\alpha,\beta,\sigma}&\  
-\sum_{k=1}^K 
\log \left[1+\exp{\left(\frac{\sum\nolimits_{j=1}^N 
{\color{black}-Z_k p^k_j}
\alpha_j}{\sigma}\right)}\right]
\\
\text{s.t.}\ 
\label{eq:MLE-re1-constraints-alpha-beta}
& \beta_j=(\alpha_{j+1}-\alpha_j)/(y_{j+1}-y_j),\ j=1,\ldots,N-1,\\
\label{eq:MLE-re1-constraints-concave}
& \beta_{j+1}\leq\beta_j,\ j=1,\ldots,N-2,\\
\label{eq:MLE-re1-constraints-mono}
&\beta_1\leq L,\ \beta_{N-1}\geq 0, \\ 
\label{eq:MLE-re1-constraints-normalization}
& \alpha_1=0,\ \alpha_N=1,\ \sigma\geq \frac{1}{\bar{c}},\\
& \alpha\in\mathbb{R}^N,\ \beta\in\mathbb{R}^{N-1},
\end{align}
\end{subequations}
where 
the constraint \eqref{eq:MLE-re1-constraints-alpha-beta} 
characterizes the relation between utility values $\alpha$ and slopes $\beta$; 
\eqref{eq:MLE-re1-constraints-concave} guarantees the concavity of $u_N$; 
\eqref{eq:MLE-re1-constraints-mono} together with \eqref{eq:MLE-re1-constraints-concave} 
ensures the monotonicity and Lipschitz continuity of $u_N$;
\eqref{eq:MLE-re1-constraints-normalization} imposes the normalization constraint on $u_N$ and the lower bound on $\sigma$.


As $\sigma\geq \frac{1}{\bar{c}}>0$, we make the following variable transformations:  
$\bar{\alpha}_j=\frac{\alpha_j}{\sigma}$, $\bar{\beta}_j=\frac{\beta_j}{\sigma}$ and $\gamma=\frac{1}{\sigma}$. 
{\color{black}We also allow $\sigma=+\infty$ as a feasible value; under this convention, $\gamma=0$.  Hence $\gamma\in[0,\bar{c}]$.}
{\color{black}Substituting these transformations into \eqref{eq:MLE-re1} yields the equivalent reformulation:}
\begin{subequations}
\begin{align}
\max\limits_{\bar{\alpha},\bar{\beta},\gamma}&\  
-\sum_{k=1}^K 
\log \left[1+\exp{\left(\sum\nolimits_{j=1}^N 
{\color{black}-Z_k p^k_j}
\bar{\alpha}_j\right)}\right]\\
\text{s.t.}\ 
& \bar{\beta}_j=(\bar{\alpha}_{j+1}-\bar{\alpha}_j)/(y_{j+1}-y_j),\ j=1,\ldots,N-1,\\
& \bar{\beta}_{j+1}\leq\bar{\beta}_j,\ j=1,\ldots,N-2,\\
& \bar{\beta}_{1}\leq L\gamma,\ 
\bar{\beta}_{N-1}\geq0,\\
& \bar{\alpha}_1=0,\ \bar{\alpha}_N=\gamma,\\
& \bar{\alpha}\in\mathbb{R}^N,\ \bar{\beta}\in\mathbb{R}^{N-1},\ 
{\color{black}\gamma\in[0,\bar{c}]}.
\end{align}
\end{subequations}
\end{proof}

{\color{black}Although problem \eqref{eq:MLE-re2} in Theorem \ref{thm:MLE-reformulation} is already a convex program, for computational convenience,} we introduce auxiliary variables $c:=[c_1,c_2,\ldots,c_K]^{\top}$, 
{\color{black}which yield an equivalent and solver-friendly conic reformulation of \eqref{eq:MLE-re2}:} 
\begin{subequations}\label{eq:MLE-re-final}
\begin{align}
\label{eq:MLE-obj}
\max\limits_{c,\bar{\alpha},\bar{\beta},\gamma}\ & \sum_{k=1}^K c_k\\
\text{\rm s.t.}\
\label{const:MLE-BTL}
& 
e^{c_k}+e^{c_k+\sum_{j=1}^N {\color{black}-Z_k p^k_j}
\bar{\alpha}_j}\leq1,
\ k=1,\ldots,K,\\
\label{const:MLE-beta-alpha-}
& \bar{\beta}_j=(\bar{\alpha}_{j+1}-\bar{\alpha}_j)/(y_{j+1}-y_j),\ j=1,\ldots,N-1,\\
\label{const:MLE-concave}
& \bar{\beta}_{j+1}\leq\bar{\beta}_j,\ j=1,\ldots,N-2,\\
\label{const:MLE-Lip}
& \bar{\beta}_1\leq L\gamma,\\
\label{const:MLE-mono}
& \bar{\beta}_{N-1}\geq0,\\
\label{const:MLE-norm}
& \bar{\alpha}_1=0,\ \bar{\alpha}_N=\gamma,\\
\label{const:MLE-domain-}
& \bar{\alpha}\in\mathbb{R}^N,\ \bar{\beta}\in\mathbb{R}^{N-1},\ c\in\mathbb{R}^K,\ 
{\color{black}\gamma\in[0,\bar{c}]}.  
\end{align}
\end{subequations}
Problem \eqref{eq:MLE-re-final} is a convex optimization problem with exponential conic constraints and can therefore be efficiently solved using state-of-the-art solvers such as MOSEK and ECOS.

\begin{remark}\label{remark:prior-MLE}
In practical applications, we may have different 
{\color{black}structural preference}
information about the DM's 
{\color{black}true VNM utility function}, 
{\color{black}in addition to}
the observed pairwise comparison data. 
In this paper,  
{\color{black}we incorporate 
{\color{black}structural}
preference information by restricting the DM's utility function to}
the ambiguity set $\mathcal{U}_c$ 
in \eqref{eq:U_c}, {\color{black}which encodes} monotonicity, concavity, and Lipschitz continuity. 
{\color{black}We also assume that}
the response errors 
{\color{black}are}
i.i.d. and follow a ${\rm Gumbel}(0,\sigma)$ distribution. 
{\color{black}Accordingly,}
we derive the reformulation \eqref{eq:MLE-re-final} of the MLE problem. 
We remark that without 
{\color{black}part (or all) of}
the 
{\color{black}structural}
preference information encoded in $\mathcal{U}_c$, 
{\color{black}we}
can still derive tractable reformulations of the corresponding MLE problems.

\begin{itemize}
\item  For a rational and {\color{black}risk-averse} DM, let $\mathcal{U}_{1}$ be an ambiguity set that incorporates 
{\color{black}structural preference}
information of monotonicity and concavity, (compared with $\mathcal{U}_c$ in \eqref{eq:U_c}, $\mathcal{U}_{1}$ excludes the Lipschitz continuity constraint), i.e., 
\begin{align*}
\mathcal{U}_{1}:= \{ u\in \mathcal{L}^1([0,\bar{b}])\mid \
& u^{'}_{+}(y)\geq0,\ u^{'}_{-}(y)\geq0,\ u^{'}_{+}(y)\leq u^{'}_{-}(y),\ \forall y\in[0,\bar{b}],\\
&u^{'}_{-}(y_2)\leq u^{'}_{+}(y_1),\ \forall 0\leq y_1<y_2\leq\bar{b},\ u(0)=0,\ u(\bar{b})=1\}, 
\end{align*}
If $\mathcal{U}_c$ in \eqref{eq:MLE-origin} is replaced by $\mathcal{U}_1$, the corresponding reformulation
{\color{black}becomes maximizing}
\eqref{eq:MLE-obj} subject to constraints \eqref{const:MLE-BTL}--\eqref{const:MLE-concave} and \eqref{const:MLE-mono}--\eqref{const:MLE-domain-}. 
In this case, if the lotteries in the pairwise comparison queries follow Dirac 
distributions (i.e., deterministic constants), the corresponding MLE problem 
{\color{black}has been}
studied in \cite{matzkin1991semiparametric}.


\item 
For a rational DM, 
let $\mathcal{U}_{2}$ be an ambiguity set that incorporates 
{\color{black}structural preference}
information of monotonicity only, (compared with $\mathcal{U}_c$ in \eqref{eq:U_c}, $\mathcal{U}_{2}$ excludes the Lipschitz continuity and concavity {\color{black}constraints}), i.e., 
\begin{align*}
\mathcal{U}_{2}:= \{ u\in \mathcal{L}^1([0,\bar{b}])\mid 
u^{'}_{+}(y)\geq0,\ u^{'}_{-}(y)\geq0,\ \forall y\in[0,\bar{b}],\ 
u(0)=0,\ u(\bar{b})=1\}. 
\end{align*}
If $\mathcal{U}_c$ in \eqref{eq:MLE-origin} is replaced by $\mathcal{U}_2$, the corresponding reformulation 
{\color{black}becomes maximizing}
\eqref{eq:MLE-obj} subject to {\color{black}constraints \eqref{const:MLE-BTL}--\eqref{const:MLE-beta-alpha-}, \eqref{const:MLE-norm}--\eqref{const:MLE-domain-} and $\bar\beta_i\geq 0$, $i=1,\ldots,N-1$.}

\item For a DM without any 
{\color{black}structural}
preference information, let {\color{black}$\mathcal{U}_{3}$ be defined as}
\begin{align*}
\mathcal{U}_{3}:= \{ u\in \mathcal{L}^1([0,\bar{b}])\mid 
u(0)=0,\ u(\bar{b})=1\}. 
\end{align*}
If $\mathcal{U}_c$ in \eqref{eq:MLE-origin} is replaced by $\mathcal{U}_3$, the corresponding reformulation 
{\color{black}becomes maximizing}
\eqref{eq:MLE-obj} subject to constraints \eqref{const:MLE-BTL}--\eqref{const:MLE-beta-alpha-} and \eqref{const:MLE-norm}--\eqref{const:MLE-domain-}. 

\item If we consider the ambiguity set $\mathcal{U}_{3}$ and fix the scale parameter of the response error 
{\color{black}to}
a constant $\hat{\sigma}$,
the MLE problem reduces to an unconstrained convex program that maximizes {\color{black}the objective} \eqref{eq:MLE-obj}, 
with $\bar{\alpha}_1 = 0$ and $\bar{\alpha}_N = \frac{1}{\hat{\sigma}}$ {\color{black}fixed}. 
This corresponds to the 
formulation 
{\color{black}widely}
studied in the choice-behavior literature~\cite{fan2025ranking,walker2002generalized,fosgerau2020discrete}. 
\end{itemize}
The 
{\color{black}structural}
preference information 
{\color{black}reflects 
restrictions
on the shape
of the DM's}
true {\color{black}VNM} utility function, 
{\color{black}and}
is crucial for simplifying both the modeling and the associated elicitation procedure. 
The numerical results in Section \ref{sec:num-prior-info} show that incorporating more 
{\color{black}structural preference}
information about the true {\color{black}VNM} utility function improves elicitation efficiency.

\end{remark}



\subsection{\color{black}Existence and uniqueness of the MLE solution}

{\color{black}
In Theorem \ref{thm:MLE-reformulation}, we allow the variable $\gamma$ to take the value  0 to ensure the well-definedness of the optimization problem \eqref{eq:MLE-re2}. 
However, $\gamma^*=0$ cannot rationalize the DM’s binary choice data. 
Thus, in this subsection, 
we establish a necessary and sufficient condition for the existence (or nonexistence) of a VNM utility function that can rationalize the DM’s binary choice data (equivalently, for determining whether $\gamma^*=0$ or $\gamma^*>0$) in Proposition~\ref{prop:gamma=0}.
In Proposition~\ref{prop:unique}, we further analyze the uniqueness of the optimal solution $(\hat{u}_{\rm MLE}, \hat{\sigma}_{\rm MLE})$ of problem~$(B)$ in~\eqref{eq-ab}, which reduces to analyzing the uniqueness of the optimal 
$\bar\alpha$ in problem~\eqref{eq:MLE-re2}.
}

{\color{black}
\begin{proposition}\label{prop:gamma=0}
Let $\mathcal{A}$ be the set of all feasible utility values $\left\{\alpha_j=u(y_j)\right\}_{j=1}^N$ at the breakpoints in 
$\{y_j\in\mathbb{Y}\}_{j=1}^N$ 
for $u\in\mathcal{U}_c$, namely, 
\begin{equation}\label{eq:alpha-set}
\mathcal{A}:=\left\{\alpha\in\mathbb{R}^N:
 \alpha_1=0, \alpha_N=1, \alpha_1\le \cdots \le \alpha_N, 
0\le \tfrac{\alpha_{j+1}-\alpha_j}{y_{j+1}-y_j}\le \tfrac{\alpha_j-\alpha_{j-1}}{y_j-y_{j-1}}\le L,
j=2,\ldots,N-1
\right\}, 
\end{equation}
which is induced by the structural constraints 
in $\mathcal{U}_c$
(i.e., monotonicity, concavity, Lipschitz continuity, and normalization). 
Define $\bar p := 
\sum_{k=1}^K Z_k\,p^k\in\mathbb{R}^N$. 
Then $\gamma^*=0$ is optimal for problem \eqref{eq:MLE-re2}
(equivalently, there exists no VNM utility function in $\mathcal{U}_c$ that can 
rationalize the DM’s observed choices
in the aggregate sense)
if and only if
\begin{equation}\label{eq:gamma0-iff-cone}
\max_{\alpha\in\mathcal{A}}\ \bar p^\top \alpha \ \le\ 0.
\end{equation}
Equivalently, the optimal $\gamma^*>0$ holds if and only if
\begin{equation}\label{eq:gamma-positive}
\max_{\alpha\in\mathcal{A}} \ \bar p^\top \alpha \ >\ 0.
\end{equation}
In particular, a sufficient (stronger but easy-to-check) condition for $\gamma^*=0$ is  
\begin{equation}\label{eq:gamma0-suff-cone}
\bar p_j \le 0,\quad j=1,\ldots,N,
\end{equation}
since every $\alpha\in\mathcal{A}$ is componentwise nonnegative. 

\end{proposition}

\begin{proof}
Let 
\[
f(\bar\alpha):=-\sum_{k=1}^K 
\log \left[1+\exp{\left(
-Z_k  (p^k)^\top \bar{\alpha}
\right)}\right]
\]
denote the objective in \eqref{const:MLE-obj}, written as a function of $\bar\alpha$.  
Let $\bar{\mathcal{A}}$ be the feasible set of $\bar\alpha$ induced by constraints \eqref{const:MLE-beta-alpha}--\eqref{const:MLE-domain}.

We first note that, over the feasible set $\bar{\mathcal{A}}$, $\gamma=0$ is equivalent to $\bar\alpha=\mathbf0$, where $\mathbf0$ denotes the zero vector. 
Indeed, if $\gamma=0$, then $\bar\alpha_N=\gamma=0$. Together with $\bar\alpha_1=0$ and the monotonicity 
$\bar\alpha_1\le \cdots \le \bar\alpha_N$ implied by \eqref{const:MLE-beta-alpha}--\eqref{const:MLE-beta}, this implies $\bar\alpha=\mathbf0$. Conversely, $\bar\alpha=\mathbf0$ implies $\gamma=\bar\alpha_N=0$. 
Thus, it suffices to show that $\bar{\alpha}^*=\mathbf{0}$ is optimal for \eqref{eq:MLE-re2} if and only if condition \eqref{eq:gamma0-iff-cone} holds. 

Since $\gamma\in[0,\bar c]$ by \eqref{const:MLE-domain}, we define the truncated 
conic hull generated by $\mathcal{A}$ as
$$\mathcal{C}:=\{\gamma\alpha:\ \gamma\in[0,\bar c],\ \alpha\in\mathcal{A}\}.$$ On the one hand,  
for any $\alpha\in\mathcal{A}$ and any $\gamma\in[0,\bar c]$, define $\bar\alpha:=\gamma\alpha$, and let $\bar\beta$ be given by
\eqref{const:MLE-beta-alpha}. Then $(\bar\alpha,\bar\beta,\gamma)$ satisfies \eqref{const:MLE-beta-alpha}--\eqref{const:MLE-domain},
so $\bar\alpha\in\bar{\mathcal{A}}$, which indicates $
\mathcal{C}\subseteq \bar{\mathcal{A}}$. 
On the other hand, 
take any $\bar\alpha\in\bar{\mathcal{A}}$ and let $\gamma=\bar\alpha_N\in[0,\bar c]$. If $\gamma=0$, then $\bar\alpha=\mathbf0$
as shown above, and thus $\bar\alpha=0\cdot\alpha$ for any $\alpha\in\mathcal{A}$. If $\gamma>0$, define $\alpha:=\bar\alpha/\gamma$.
Then $\alpha_1=0$ and $\alpha_N=1$ by \eqref{const:MLE-alpha}. Dividing the slope constraints in
\eqref{const:MLE-beta-alpha}--\eqref{const:MLE-beta} by $\gamma$ implies that $\alpha$ satisfies the monotonicity, concavity,
and $L$-Lipschitz constraints. Hence $\alpha\in\mathcal{A}$. Therefore $\bar\alpha=\gamma\alpha$ with $\gamma\in(0,\bar c]$ and
$\alpha\in\mathcal{A}$, yielding $\bar{\alpha}\in\mathcal{C}$ and 
$\bar{\mathcal{A}}\subseteq \mathcal{C}
$. 
Consequently,
$\bar{\mathcal{A}}=\mathcal{C}$.

The function $f(\cdot)$ is concave in $\bar{\alpha}$ as it is the negative of a sum of convex $\log(1+\exp(\cdot))$ terms. 
Since $f$ is differentiable and concave, and $\bar{\mathcal{A}}$ is convex, $\bar\alpha^*=\mathbf0$ is optimal for maximizing $f$ over $\bar{\mathcal{A}}$
if and only if the first-order optimality condition holds:
\begin{equation*}
\nabla f(\mathbf0)^\top (\bar\alpha-\mathbf0)\le 0,\quad \forall \bar\alpha\in\bar{\mathcal{A}}.
\end{equation*}
A direct calculation gives
\[
\nabla f(\mathbf0)=-\sum_{k=1}^K \frac{e^0}{1+e^0} (-Z_k\,p^k)=\frac12\sum_{k=1}^K Z_k\,p^k
=\frac{1}{2} \bar p. 
\]
Therefore, the first-order condition is equivalent to $\bar p^\top\bar\alpha\le 0$ for all $\bar\alpha\in\bar{\mathcal{A}}$, i.e.,
$\bar p^\top(\gamma\alpha)\le 0$ for all $\gamma\in[0,\bar c]$ and all $\alpha\in\mathcal{A}$. 
Because this inequality must hold for all $\gamma\in[0,\bar c]$,
the optimal $\bar\alpha^*=\mathbf0$ holds if and only if
\[
\bar p^\top \alpha\le 0,\quad \forall \alpha\in\mathcal{A},
\]
which is equivalent to \eqref{eq:gamma0-iff-cone}.

The equivalence in \eqref{eq:gamma-positive} follows immediately from
\eqref{eq:gamma0-iff-cone} since the two conditions are complementary.

Finally, if $\bar p_j \le 0$ for all $j=1,\ldots,N$, then
$\bar p^\top \alpha \le 0$ for all $\alpha\in\mathcal{A}$ since every $\alpha\in\mathcal{A}$ is componentwise nonnegative.
Hence,
$\max_{\alpha\in\mathcal{A}} \bar p^\top \alpha \le 0$,
and \eqref{eq:gamma0-iff-cone} implies $\gamma^*=0$. 
\end{proof}
}

{\color{black}
\begin{remark}\label{remark:gamma=0}



We provide some 
comments and interpretations for the conditions in Proposition~\ref{prop:gamma=0} and discuss the case when $\gamma^*\to \infty$.

\begin{enumerate}[label=(\roman*)]
\item 
For each comparison $k$, recall that $p^k_j=\mathbb{P}[W_k=y_j]-\mathbb{P}[Y_k=y_j]$ and $Z_k\in\{1,-1\}$ encodes the observed choice.
For $k=1,\ldots,K$, define the \emph{chosen} and \emph{rejected} lotteries as
\[
C_k:=\begin{cases}
W_k, & \text{if }\ Z_k=1,\\
Y_k, & \text{if }\ Z_k=-1,
\end{cases}
\qquad
R_k:=\begin{cases}
Y_k, & \text{if }\ Z_k=1,\\
W_k, & \text{if }\ Z_k=-1.
\end{cases}
\]
Then $Z_k p^k$ is exactly the probability-mass difference between the chosen and rejected lotteries, 
$$Z_k p^k_j=\mathbb{P}[C_k=y_j]-\mathbb{P}[R_k=y_j],\ j=1,\ldots,N,$$ and
$\bar p=\sum_{k=1}^K Z_k p^k$ aggregates these choice-aligned probability-mass difference across all comparisons. 
For any $\alpha\in\mathcal{A}$ and any lottery $Y$, define the expected utility under $\alpha$ as $\mathbb{E}_{\alpha}[Y]
:=\sum_{j=1}^N \mathbb{P}[Y=y_j]\alpha_j$.
It follows that
\[
\bar p^\top \alpha
=\sum_{k=1}^K Z_k 
\left(p^k\right)^{\top}\alpha
=\sum_{k=1}^K \left(\mathbb{E}_{\alpha}[C_k]  - \mathbb{E}_{\alpha}[R_k]\right),
\]
which represents the aggregate expected-utility advantage of the chosen lotteries $C_k$ over the rejected lotteries $R_k$ across $k=1,\ldots,K$, under the utility function 
parameterized
by the utility values $\alpha$ at breakpoints. 
Importantly, the condition $\bar p^\top \alpha>0$ ensures a strictly positive
aggregate expected-utility margin across all comparisons, 
but it does not necessarily require that $\mathbb{E}_\alpha[C_k]> \mathbb{E}_\alpha[R_k]$ holds for each $k$ individually.


\item Proposition \ref{prop:gamma=0} shows that $\gamma^*=0$ (equivalently, 
$\hat\sigma_{\rm MLE}=+\infty$)
holds if and only if
\[
\bar p^\top \alpha \le 0,\quad \forall \alpha\in\mathcal{A}.
\]
This means that 
there exists no utility function 
parameterized
by $\alpha\in \mathcal{A}$
that creates a strictly positive
aggregate 
expected-utility margin
in favor of the chosen lotteries
$C_k$ 
over the rejected lotteries $R_k$ across $k=1,\ldots,K$.
Then, by Proposition \ref{prop:MLE-piecewise}, it follows that in this case, there exists 
no economically meaningful VNM utility function $u\in\mathcal{U}_c$ consistent with the 
structural preference information (monotonicity, concavity, and Lipschitz continuity) that 
rationalizes the DM's observed choices. Here, rationalization is understood in the \emph{aggregate sense} described in part (i) above, rather than requiring each individual comparison to be rationalized separately.
Alternatively, one may remove some of these structural restrictions, as 
discussed in Remark \ref{remark:prior-MLE}, and correspondingly weaken condition \eqref{eq:gamma0-iff-cone} by enlarging the set $\mathcal{A}$.



\item
The sufficient condition 
\eqref{eq:gamma0-suff-cone} implies that 
\begin{equation*}
\bar p^\top \alpha \ \le\ 0,\ \forall \alpha\in\mathbb{R}^N_+. 
\end{equation*}
In this case, it follows that
even 
a much broader class of nonnegative-valued utility functions
(without monotonicity, concavity, and Lipschitz continuity restrictions)
cannot rationalize the DM's 
observed
choices in the \emph{aggregate sense}.

\item 
We then consider the opposite case where $\gamma^*\to \infty$ (equivalently, $\hat\sigma_{\rm MLE}\to 0$). If there exists $\alpha\in\mathcal A$ such that 
\[
Z_k (p^k)^\top \alpha \geq 0 \text{ for all } k=1,\ldots,K,\ \text{and } 
Z_{k_0} (p^{k_0})^\top \alpha> 0 \text{ for some } k_0\in\{1,\ldots,K\},\
\]
then the supremum $\bar c$ of $\gamma$ in problem \eqref{eq:MLE-re2} is attained. 
This sufficient condition requires the existence of a feasible VNM utility function $u\in\mathcal{U}_c$ that rationalizes the DM’s binary choices \emph{comparison by comparison}, rather than in the \emph{aggregate sense}.
Notice that, if we remove the upper-bound constraint on $\gamma$, the above condition becomes sufficient and necessary for 
$\gamma^*\to\infty$. 
Moreover, if we further remove the structural constraints on the utility function and 
only require
$\alpha\in\mathbb{R}^N_+$, this condition is exactly the analogue of the quasi-complete separation condition in logistic regression (see, e.g.,~\cite{albert1984existence,silvapulle1981existence}), 
under which the maximum likelihood estimator diverges. 
Compared with the classical results in the literature, our formulation 
imposes
structural restrictions on the utility function,  
and we additionally discuss the case $\gamma=0$ (equivalently, $\sigma\to \infty$) in Proposition~\ref{prop:gamma=0}, 
which 
is 
rarely considered
in traditional results that solve the MLE with a fixed $\sigma$. 


\end{enumerate}

\end{remark}
}

{\color{black}
\begin{proposition}\label{prop:unique}
Let  
\begin{equation}\label{eq:SigmaD}
    \Sigma_D:=\frac{1}{K}P^\top P\in\mathbb{R}^{(N-1)\times (N-1)}
\end{equation}
be the empirical information matrix
associated with the dataset $D$, where $P\in\mathbb{R}^{K\times (N-1)}$ is constructed from the reduced probability-mass
difference vectors $p^k_{2:N}$ as defined in Section~\ref{sec:stat-infer}. 
If $\Sigma_D$ is full rank, i.e., ${\rm rank}(\Sigma_D)={\rm rank}(P)=N-1$, and the optimal $\gamma^*>0$ in problem \eqref{eq:MLE-re2}, then 
$(\hat{u}_{\rm MLE}, \hat{\sigma}_{\rm MLE})$
is the unique optimal solution to problem $(B)$ in \eqref{eq-ab}. 
\end{proposition}

\begin{proof}

According to Theorem \ref{thm:MLE-reformulation} and its proof, 
solving problem $(B)$ in \eqref{eq-ab} is equivalent to solving the convex program \eqref{eq:MLE-re2}. 
If the resulting optimal $\gamma^*>0$, we then recover $\hat{\alpha}^{\rm MLE}$, $\hat{\beta}^{\rm MLE}$, and $\hat{\sigma}_{\rm MLE}$, and subsequently construct $\hat{u}_{\rm MLE}$ via \eqref{eq:N-piecewise utility function}. 
Therefore, we proceed with the proof in three steps: (i) existence of an optimal solution to \eqref{eq:MLE-re2}, (ii) uniqueness of the optimal solution of \eqref{eq:MLE-re2} under ${\rm rank}(\Sigma_D)={\rm rank}(P)=N-1$, and (iii) uniqueness of $\hat{u}_{\rm MLE}$ together with $\hat{\sigma}_{\rm MLE}$.

\underline{Step 1.} 
Let $\mathcal{F}$ denote the feasible set of \eqref{eq:MLE-re2}. First, take $\tilde{\alpha}_j=0$ for all $j=1,\ldots,N$, $\tilde{\gamma}=0$, and $\tilde{\beta}_j=0$ for all $j=1,\ldots,N-1$. 
Then $(\tilde{\alpha},\tilde{\beta},\tilde{\gamma})\in\mathcal{F}$, so $\mathcal{F}\neq\emptyset$. 
Next, $\mathcal{F}$ is closed since it is defined by linear equalities/inequalities together with the box constraint $\gamma\in[0,\bar c]$. 
Moreover, for any $(\bar{\alpha},\bar{\beta},\gamma)\in\mathcal{F}$, we have $0\leq\bar{\alpha}_j\leq\gamma\leq\bar{c}$ for $j=1,\ldots,N$, $0\leq\bar{\beta}_j\leq L\gamma\leq L\bar{c}$ for $j=1,\ldots,N-1$, and $0\leq\gamma\leq\bar{c}$, which implies $\mathcal{F}$ is bounded. 
Hence $\mathcal{F}$ is compact. 
Denote the objective in \eqref{const:MLE-obj} by 
\[
f(\bar\alpha):=-\sum_{k=1}^K 
\log \left[1+\exp{\left(\sum_{j=1}^N -Z_k p^k_j
\bar{\alpha}_j\right)}\right]=-\sum_{k=1}^K 
\log \left[1+\exp{\left(\sum_{j=2}^N -Z_k p^k_j
\bar{\alpha}_j\right)}\right], 
\]
where the equality follows from $\bar{\alpha}_1=0$. 
The function $f(\cdot)$ is continuous 
as it is a finite sum of compositions of continuous functions (affine mapping, $\exp(\cdot)$, and $\log(\cdot)$). 
By the Weierstrass theorem, $f$ attains its maximum over the nonempty compact set $\mathcal{F}$. 
Therefore, problem \eqref{eq:MLE-re2} admits at least one optimal solution $(\bar\alpha^*,\bar\beta^*,\gamma^*)$.

\underline{Step 2.} 
As $\bar\alpha_1=0$, we define the reduced variable 
$\bar{\alpha}_{2:N}:=[\bar{\alpha}_2,\ldots,\bar{\alpha}_N]\in\mathbb{R}^{N-1}$, 
and restrict attention to 
the reduced objective 
$\bar f(\bar{\alpha}_{2:N}):=f((0;\bar{\alpha}_{2:N}))$. 
We compute the Hessian of $\bar{f}(\cdot)$ with respect to $\bar{\alpha}_{2:N}$ as: 
\begin{align*}
\nabla^2 \bar{f}(\bar{\alpha}_{2:N})
& = -\sum_{k=1}^K
\frac{\exp\!\left(\sum_{j=2}^N Z_k p_j^k \bar\alpha_j\right)}
{\left[1+\exp\!\left(\sum_{j=2}^N Z_k p_j^k \bar\alpha_j\right)\right]^2}\,
p^k_{2:N}\,(p^k_{2:N})^\top.
\end{align*}
Since $\bar\alpha_j\in[0,\bar c]$ and $|p_j^k|\le 1$, the quantity
$\sum_{j=2}^N -Z_k p_j^k \bar\alpha_j$ is uniformly bounded on $\mathcal{F}$. 
As the function 
$g(t):=\frac{e^{t}}{(1+e^{t})^{2}}$ is continuous and strictly positive on $t\in\mathbb{R}$, there exists $M>0$ such that 
$$\frac{\exp\left({\sum_{j=2}^N Z_k p_j^k \bar\alpha_j}\right)}{\left[1+\exp\left({\sum_{j=2}^N Z_k p_j^k \bar\alpha_j}\right)\right]^2}\geq M$$ for all feasible $\bar{\alpha}_{2:N}$.
Consequently, $$\nabla^2 \bar{f}(\bar{\alpha}_{2:N})\preceq -M \sum_{k=1}^K p^k_{2:N}(p^k_{2:N})^{\top}=-MP^{\top}P$$ for all feasible $\bar{\alpha}_{2:N}$. 
Thus, for any $v\in\mathbb{R}^{N-1}$, 
\begin{equation*}
v^{\top}\nabla^2 \bar{f}(\bar{\alpha}_{2:N})v\leq -M v^{\top}P^{\top}Pv=-M\|Pv\|^2_2\leq0. 
\end{equation*}
Moreover, $v^\top \nabla^2 \bar{f}(\bar{\alpha}_{2:N}) v=0$ implies $\|Pv\|_2^2=0$, i.e., $Pv=\mathbf{0}$, which further implies $v=\mathbf{0}$ since ${\rm rank}(P)=N-1$. 
Therefore, $\nabla^2 \bar{f}(\bar{\alpha}_{2:N})$ is strictly negative definite on $\mathbb{R}^{N-1}$,
and $\bar{f}(\cdot)$ is strongly concave in $\bar\alpha_{2:N}$. 
It follows that the convex program \eqref{eq:MLE-re2} admits a unique optimal solution $(\bar\alpha^*,\bar\beta^*,\gamma^*)$.

\underline{Step 3.} 
By Step 2, the optimizer $(\bar\alpha^*,\bar\beta^*,\gamma^*)$ of \eqref{eq:MLE-re2} is unique. 
When $\gamma^*>0$, the quantities
\[
\hat\alpha^{\mathrm{MLE}}=\frac{\bar\alpha^*}{\gamma^*},\quad
\hat\beta^{\mathrm{MLE}}=\frac{\bar\beta^*}{\gamma^*},\quad
\hat\sigma_{\mathrm{MLE}}=\frac{1}{\gamma^*}
\]
are uniquely determined as in Theorem~\ref{thm:MLE-reformulation}. 
Finally, $\hat u_{\mathrm{MLE}}$ in \eqref{eq:N-piecewise utility function} is uniquely constructed from $\hat\alpha^{\mathrm{MLE}}$ and $\hat\beta^{\mathrm{MLE}}$ as in \eqref{eq:N-piecewise utility function}. Therefore, $(\hat u_{\mathrm{MLE}},\hat\sigma_{\mathrm{MLE}})$ is unique. 
\end{proof}
}

{\color{black}Based on Theorem \ref{thm:MLE-reformulation} and Propositions \ref{prop:MLE-piecewise}--\ref{prop:unique}, we can now characterize 
the optimal solution set $\mathcal{U}^*_c$ of 
problem $(A)$ in \eqref{eq-ab} in 
Proposition \ref{thm:optimal-set-MLE} below. 
In effect, the MLE procedure identifies a family of VNM utility models 
of the form $U(\cdot) = \mathbb{E}[u(\cdot)] + \varepsilon$, 
where $u \in \mathcal{U}_c^*$ and $\varepsilon \sim \text{Gumbel}(0, \hat{\sigma}_{{\rm MLE}})$. 
}

\begin{proposition}\label{thm:optimal-set-MLE}
{\color{black}
Assume that the empirical information matrix $\Sigma_D$ is full rank and that the optimal $\gamma^*>0$ in problem \eqref{eq:MLE-re2}. 
With
the MLE-based VNM utility function $\hat{u}_{\rm MLE}$ 
and scale parameter $\hat{\sigma}_{\rm MLE}$
derived 
in Theorem~\ref{thm:MLE-reformulation}, define 
} 
\begin{equation}\label{eq:MLE-u-opt-set}
    \mathcal{U}_c^*:=\{u\in\mathcal{U}_c\mid \ u(y_j)=
    \hat{u}_{\rm MLE}(y_j),
    \ j=1,\ldots,N\}, 
\end{equation}
{\color{black}then $\left\{(u,\hat{\sigma}_{\rm MLE})\mid u\in\mathcal{U}^*_c\right\}=\mathop{\arg\max}\limits_{u\in\mathcal{U}_c,\ \sigma\geq\frac{1}{\bar{c}}}\ \ell(u,\sigma)$.
Equivalently, $\mathcal{U}_c^*=\mathop{\arg\max}\limits_{u\in\mathcal{U}_c
}\ \ell(u,\hat{\sigma}_{\rm MLE})$.}

\end{proposition}

{\color{black}
\begin{proof}

By 
\eqref{eq:log-likelihood-re}, 
the log-likelihood
$\ell(u,\sigma)$ depends on $u$ only through its values at the breakpoints $y_1,\ldots,y_N$.
Hence, if $u_1,u_2\in\mathcal{U}_c$ satisfy $u_1(y_j)=u_2(y_j)$ for all $j=1,\ldots,N$, 
then $\ell(u_1,\sigma)=\ell(u_2,\sigma)$. 

Since $\mathcal U_N\subseteq\mathcal U_c$ and problems $(A)$ and $(B)$ in \eqref{eq-ab} share the same optimal value (Proposition~\ref{prop:MLE-piecewise}), 
problem $(B)$ is a restriction of problem $(A)$. 
Moreover, under the assumptions in this proposition, 
Proposition~\ref{prop:unique} implies that $(\hat u_{\rm MLE},\hat\sigma_{\rm MLE})$ 
is the unique optimal solution to problem $(B)$.

Let $(u^*,\sigma^*)\in\mathop{\arg\max}_{u\in\mathcal U_c,\ \sigma\ge \frac{1}{\bar c}}\ \ell(u,\sigma)$ be any optimal solution of problem $(A)$.
Denote 
by $\mathcal{T} u^*\in\mathcal{U}_N$ its piecewise-linear lower approximation, which satisfies 
$\mathcal{T} u^*(y_j)=u^*(y_j)$ for $j=1,\ldots,N$. 
Then $\ell(u^*,\sigma^*)=\ell(\mathcal{T} u^*,\sigma^*)$. 
Since $\mathcal T u^*\in\mathcal U_N$ and $(B)$ has the same optimal value as $(A)$, it follows that
$(\mathcal T u^*,\sigma^*)$ is
optimal for problem $(B)$. 
By uniqueness, we obtain 
$\sigma^*=\hat{\sigma}_{\rm MLE}$ and $\mathcal{T} u^*=\hat{u}_{\rm MLE}$. 
Hence, $u^*(y_j)=\mathcal{T} u^*(y_j)=\hat{u}_{\rm MLE}(y_j)$ for $j=1,\ldots,N$, which implies
$u^*\in\mathcal{U}^*_c$. 

Conversely, for any $u
\in\mathcal U_c^\ast
$, 
we have 
$\ell(u
,\hat\sigma_{\mathrm{MLE}})=\ell(\hat u_{\mathrm{MLE}},\hat\sigma_{\mathrm{MLE}})$.
Since $(\hat u_{\mathrm{MLE}},\hat\sigma_{\mathrm{MLE}})$ attains the optimal value of problem $(B)$,
and problems $(A)$ and $(B)$ share the same optimal value,
$(u,\hat\sigma_{\rm MLE})$ attains the optimal value of problem $(A)$.



Combining the two inclusions yields the desired conclusions. 
\end{proof}}

\subsection{Statistical consistency of MLE}\label{sec:thm-MLE}

This subsection 
presents 
the statistical 
convergence analysis 
of the MLE. 
We first 
{\color{black}derive}
$\ell_2$ and $\ell_{\infty}$ statistical errors between the MLE-based parameters $(\hat{\alpha}^{\rm MLE}, \hat{\sigma}_{\rm MLE})$ and the true parameters $(\alpha^*, \sigma^*)$ 
in Theorem \ref{thm:alpha-sigma-matrix} and {\color{black}Proposition} \ref{thm:2-infty}, and 
establish the corresponding convergence rate in Corollary \ref{coll:bound-param}.  
We then establish the convergence results, 
{\color{black}in the sense of the Kolmogorov distance}, 
of the MLE-based adjusted VNM utility function
$\frac{\hat{u}_{\rm MLE}}{\hat{\sigma}_{\rm MLE}}$ to the true 
adjusted utility function
$\frac{u^*}{\sigma^*}$, 
as stated in Theorem \ref{thm:bound-utility} and 
Corollary \ref{coll:bound-utility}.

We begin by introducing the notation used in this subsection. 
Let  
{\color{black}$\mathcal{A}$ be defined as in \eqref{eq:alpha-set}}. 
Recall that $u^*$ denotes the DM's true VNM utility function and $\sigma^*$ denotes the true scale parameter {\color{black}as in Assumption \ref{assump:VNM}}.  
{\color{black}Let}
$\alpha^*:=[\alpha^*_1,\ldots,\alpha^*_N]$ 
{\color{black}denote}
the true VNM utility values at the breakpoints in $\mathbb{Y}$, i.e., $\alpha^*_j=u^*(y_j)$, $j=1,\ldots,N$. 
{\color{black}For $\alpha\in\mathcal{A}$ and 
$\sigma\geq\frac{1}{\bar{c}}$, define}
the adjusted utility value vector
$\theta:=[\theta_1,\ldots,\theta_N]^{\top}=[\frac{\alpha_1}{\sigma},\ldots,\frac{\alpha_N}{\sigma}]^{\top}$,  
and let 
$${\Theta}:=\left\{\theta \ \big| \ \theta=\frac{\alpha}{\sigma},\ \alpha\in\mathcal{A},\ \sigma\geq\frac{1}{\bar{c}}\right\}.$$
{\color{black}  
We introduce a fixed reduction matrix $Q:=[\mathbf{0}_{(N-1)\times1}, \mathbf{I}_{N-1}]\in\mathbb{R}^{(N-1)\times N}$, where $\mathbf{0}$ denotes the zero vector and $\mathbf{I}$ denotes the identity matrix, 
then $Qp^k=p^k_{2:N} \in\mathbb{R}^{N-1}$ for $k=1,\ldots,K$. 
Note that the normalization constraint $\alpha_1=0$ implies 
$\theta_1=0$ for all $\theta\in\Theta$, so the free variables are $Q\theta= \theta_{2:N}:=[\theta_2,\ldots,\theta_N]^\top\in\mathbb{R}^{N-1}$. 
} 
By 
{\color{black}the above}
transformation, {\color{black}we have $\theta^{\top}{p}^k=(Q\theta)^{\top}p^k_{2:N}$ for any $\theta\in\Theta$, $k=1,\ldots,K$},   
{\color{black}and}
the log-likelihood function $\ell(u,\sigma)$ in \eqref{eq:log-likelihood-re} 
{\color{black}can be rewritten as a function of $\theta$ as follows:}
\begin{align*}
\ell(\theta)=
{\color{black}-\sum_{k=1}^K\log\left[1+\exp\left(-Z_k\theta^{\top}p^k\right)\right]
=-\sum_{k=1}^K\log\left[1+\exp{\left(-Z_k(Q\theta)^{\top}p^k_{2:N}\right)}\right]}.
\end{align*}



\begin{lemma}\label{lemma:theta-bound}
{\color{black}Let $\theta^1,\theta^2\in\Theta$ be}
two adjusted utility value vectors.  
{\color{black}Then}
\begin{equation}
 \|{\theta}^1-\theta^2\|_2\leq \bar{c}\sqrt{N-1}. 
\end{equation}
For the MLE 
{\color{black}estimates}
$\hat{\alpha}^{\rm MLE}:=[\hat{\alpha}^{\rm MLE}_1,\ldots,\hat{\alpha}^{\rm MLE}_N]$ 
and $\hat{\sigma}_{\rm MLE}$ 
{\color{black}obtained}
in Theorem~\ref{thm:MLE-reformulation}, the MLE-based adjusted utility value vector $\hat{\theta}^{\rm MLE}:= \frac{\hat{\alpha}^{\rm MLE}}{\hat{\sigma}_{\rm MLE}}$ and 
the true adjusted utility value vector $\theta^*:=\frac{\alpha^*}{\sigma^*}$ both belong to 
$\Theta$, and 
{\color{black}hence}
satisfy $\|\hat{\theta}^{\rm MLE}-\theta^*\|_2\leq \bar{c}\sqrt{N-1}$. 
\end{lemma}

\begin{proof}
By the definition of $\Theta$ and $\mathcal{A}$, for any $\theta\in\Theta$, it is obvious that 
\begin{equation*}
\theta_1=0,\ 0\leq\theta_j=\frac{\alpha_j}{\sigma}\leq\frac{1}{\sigma}\leq \bar{c},\  j=2,\ldots,N.  
\end{equation*}
Thus $|\theta^1_1-\theta^2_1|=0$ and $|\theta^1_j-\theta^2_j|\leq\max\{\theta^1_j,\ \theta^2_j\}\leq\bar{c}$, $j=2,\ldots,N$. 
Then, we have 
\begin{equation*}
\left\|{\theta}^1-\theta^2\right\|_2 
=\sqrt{\sum_{j=2}^{N}\left(\theta^1_j-\theta^2_j\right)^2}\leq \bar{c}\sqrt{N-1}.  
\end{equation*} 
\end{proof}

\begin{lemma}\label{lemma:gradient-bound}
For any $\theta\in\Theta$, we have $\left|{\color{black}(Q\theta)^{\top}p^k_{2:N}}\right|
{\color{black}=\left|\theta_{2:N}^{\top}\ p^k_{2:N}\right|}
\leq2\bar{c}$, $k=1,\ldots,K$. 
\end{lemma}
\begin{proof} 
For a fixed $k$, 
since $\sum_{j=1}^N \mathbb{P}[W_k=y_j] =1$ and $\sum_{j=1}^N \mathbb{P}[Y_k=y_j] =1$, we have 
$$
\left| \sum_{j=1}^N \mathbb{P}[W_k=y_j] \theta_j\right|\leq\max\limits_{j=1,\ldots,N}\ \theta_j \leq \bar{c},
$$
and $\left| \sum_{j=1}^N \mathbb{P}[Y_k=y_j] \theta_j\right|\leq \bar{c}$. Then, 
\begin{equation*}
\left|{\color{black}(Q\theta)^{\top}p^k_{2:N}}\right|
=\left|\sum_{{\color{black}j=2}}^N \left(\mathbb{P}[W_k=y_j]-\mathbb{P}[Y_k=y_j]\right)\theta_j \right|
\leq \left| \sum_{j=1}^N \mathbb{P}[W_k=y_j]\theta_j\right|+\left| \sum_{j=1}^N \mathbb{P}[Y_k=y_j]\theta_j\right|\leq 2\bar{c}. 
\end{equation*}
\end{proof}

\begin{theorem}\label{thm:alpha-sigma-matrix}
For any $\lambda>0$ 
and 
{\color{black}any $\delta\in(0,1)$}, 
with probability at least $1-\delta$, we have 
\begin{equation*}
\left\| \frac{\hat{\alpha}^{\rm MLE}}{\hat{\sigma}_{\rm MLE}}-\frac{\alpha^*}{\sigma^*} \right\|_{{\color{black}Q^{\top}(\Sigma_D+\lambda \mathbf{I})Q}}
\leq \sqrt{\frac{R_D+2\sqrt{-R_D\log\delta}-2\log\delta}{\omega^2 K}}
+\sqrt{\frac{R_D+2\sqrt{-R_D\log\delta}-2\log\delta}{\omega^2 K}+\lambda\left(\bar{c}\sqrt{N-1}\right)^2}, 
\end{equation*}
where $\Sigma_D=\frac{1}{K}P^{\top}P{\color{black}\;\in\mathbb{R}^{(N-1)\times(N-1)}}$ {\color{black}is the empirical information matrix of the dataset $D$ defined in~\eqref{eq:SigmaD}}, 
$R_{D}:={\rm rank}(\Sigma_D)$, and $\omega:=\frac{1}{2+2\exp{(2\bar{c})}}$. 
\end{theorem}

\begin{proof}
We proceed with the proof in three steps: first, we establish a preliminary boundary as shown in \eqref{eq:Delta}; next, we further bound $\left\|\nabla\ell(\theta^*) \right\|_{(\Sigma_D+\lambda\mathbf{I})^{-1}}$ as in \eqref{eq:unknown-term}; and finally, we obtain the results.

\noindent
\underline{Step 1.}
For the log-likelihood function 
\begin{equation*}
\ell(\theta)=
-\sum_{k=1}^K 
{\color{black}
\log\left[1+\exp{\left(-Z_k(Q\theta)^{\top}p^k_{2:N}\right)}\right]},
\end{equation*}
the gradient of $\ell$ with respect to $\theta$ is
\begin{equation}\label{eq:grad-MLE}
\nabla\ell(\theta)=
-\sum_{k=1}^K 
{\color{black}
\left[
\frac{-Z_k}{1+\exp{\left(Z_k(Q\theta)^{\top}p^k_{2:N}\right)}}
\right]Q^{\top}p^k_{2:N}}, 
\end{equation}
and the Hessian matrix of $\ell$ with respect to $\theta$ is
\begin{equation*}
\nabla^2\ell(\theta)=
-\sum_{k=1}^K 
{\color{black}
\left[\frac{\exp{\left(Z_k(Q\theta)^{\top}p^k_{2:N}\right)}}{\left[1+\exp{\left(Z_k(Q\theta)^{\top}p^k_{2:N}\right)}\right]^2}\right]Q^{\top}p^k_{2:N}(p^k_{2:N})^{\top}Q
}.  
\end{equation*} 
Then, by Lemma \ref{lemma:gradient-bound}, we have 
$$
\frac{\exp{\color{black}\left(Z_k(Q\theta)^{\top}p^k_{2:N}\right)}}{\left[1+\exp{\color{black}\left(Z_k(Q\theta)^{\top}p^k_{2:N}\right)}\right]^2}=\frac{1}{2+\exp{\color{black}\left(Z_k(Q\theta)^{\top}p^k_{2:N}\right)}+\exp{\color{black}\left(-Z_k(Q\theta)^{\top}p^k_{2:N}\right)}}\geq\frac{1}{2+2\exp{(2\bar{c})}}:=\omega,
$$ 
thus 
$\nabla^2\ell(\theta) \preceq - \omega \sum_{k=1}^K{\color{black}Q^{\top}p^k_{2:N} ({p^k_{2:N}})^{\top}Q}$. 
Then, it follows that  
\begin{equation*}
    v^{\top} \nabla^2\ell(\theta) v \leq v^{\top}\left(- \omega \sum_{k=1}^K{\color{black}Q^{\top}p^k_{2:N} ({p^k_{2:N}})^{\top}Q}\right)v=-\omega \left\|P ({\color{black}Q} v) \right\|_2^2, \ \forall v\in\mathbb{R}^N.  
\end{equation*}
Let the estimation error be ${\color{black}\tilde{\Delta}}:=\hat{\theta}^{\rm MLE}-\theta^*=\frac{\hat{\alpha}^{\rm MLE}}{\hat{\sigma}_{\rm MLE}}-\frac{\alpha^*}{\sigma^*}\in\mathbb{R}^N$ {\color{black} with $\tilde{\Delta}_1\equiv0$ and let $\Delta:=Q\tilde{\Delta}\in\mathbb{R}^{N-1}$}. 
By the mean value theorem, there exists a $\tilde{\theta}$ between $\hat{\theta}^{\rm MLE}$ and $\theta^*$ such that 
\begin{equation*}
\ell(\theta^*+{\color{black}\tilde{\Delta}})-\ell(\theta^*)-\langle \nabla\ell(\theta^*), {\color{black}\tilde{\Delta}} \rangle=\frac{1}{2}{\color{black}\tilde{\Delta}}^{\top} \nabla^2\ell(\tilde{\theta}) {\color{black}\tilde{\Delta}} \leq -\frac{\omega}{2} \left\|P ({\color{black}Q\tilde{\Delta}}) \right\|_2^2=-\frac{\omega K}{2} \left\|\Delta \right\|_{\Sigma_D}^2.    
\end{equation*}

Since $\hat{\theta}^{\rm MLE}$ is an optimal solution that maximizes the log-likelihood function $\ell(\theta)$, we have $\ell(\theta^*)\leq\ell(\hat{\theta}^{\rm MLE})=\ell(\theta^*+{\color{black}\tilde{\Delta}})$, which implies that 
\begin{equation*}
-\langle \nabla\ell(\theta^*), {\color{black}\tilde{\Delta}} \rangle\leq \ell(\theta^*+{\color{black}\tilde{\Delta}})-\ell (\theta^*)-\langle \nabla\ell(\theta^*), {\color{black}\tilde{\Delta}} \rangle \leq-\frac{\omega K}{2} \left\|\Delta \right\|_{\Sigma_D}^2.   
\end{equation*}
Thus, {\color{black}by $\tilde{\Delta}_1=0$,}
\begin{equation*}
\frac{\omega K}{2} \left\|\Delta \right\|_{\Sigma_D}^2 \leq \left|\langle \nabla\ell(\theta^*), {\color{black}\tilde{\Delta}} \rangle \right|
{\color{black}=\left| \langle Q\nabla\ell(\theta^*), \Delta \rangle\right|}
\leq \left\|{\color{black}Q}\nabla\ell(\theta^*) \right\|_{(\Sigma_D+\lambda\mathbf{I})^{-1}} \left\|\Delta \right\|_{\Sigma_D+\lambda\mathbf{I}}, \quad \forall \lambda>0,   
\end{equation*}
where $\mathbf{I}$ denotes the identity matrix. 
By {\color{black} the fact $\tilde{\Delta}_1=0$ and} Lemma~\ref{lemma:theta-bound}, we have $\|\Delta\|_2
{\color{black}=\|\tilde{\Delta}\|_2}
=\|\hat{\theta}^{\rm MLE}-\theta^*\|_2\leq \bar{c}\sqrt{N-1}$, then 
\begin{align}\label{eq:Delta}
\left\|\Delta \right\|^2_{\Sigma_D+\lambda\mathbf{I}}
&=\left\|\Delta \right\|^2_{\Sigma_D}+\lambda\|\Delta\|_2^2\nonumber\\
&\leq\left\|\Delta \right\|^2_{\Sigma_D}+
\lambda\left(\bar{c}\sqrt{N-1}\right)^2
\nonumber\\
&\leq\frac{2}{\omega K}\left\|{\color{black}Q}\nabla\ell(\theta^*) \right\|_{(\Sigma_D+\lambda\mathbf{I})^{-1}}\left\|\Delta \right\|_{\Sigma_D+\lambda\mathbf{I}}+\lambda\left(\bar{c}\sqrt{N-1}\right)^2. 
\end{align}

\noindent
\underline{Step 2.}
Next, we further bound $\left\|{\color{black}Q}\nabla\ell(\theta^*) \right\|_{(\Sigma_D+\lambda\mathbf{I})^{-1}}$. 
Given the true adjusted utility values $\theta^{*}=\frac{\alpha^{*}}{\sigma^{*}}$, 
we define a random vector $V:=[V_1,\ldots,V_K]^{\top}$ with independent 
components 
\begin{equation*}
V_k=\left\{ 
\begin{aligned}
& \frac{1}{1+\exp{\color{black}(-(Q\theta^*)^{\top}p^k_{2:N})}}
&\quad{\rm w.p.}\quad & 
\mathbb{P}\left(Z_k=-1\right)=\frac{1}{1+\exp{\color{black}((Q\theta^*)^{\top}p^k_{2:N})}}
\\
& \frac{-1}{1+\exp{\color{black}((Q\theta^*)^{\top}p^k_{2:N})}}
&\quad{\rm w.p.}\quad & \mathbb{P}\left(Z_k=1\right)=\frac{1}{1+\exp {\color{black}(-(Q\theta^*)^{\top}p^k_{2:N})}}
\end{aligned}
\right., \quad k=1,\ldots,K. 
\end{equation*}
Notice that $\frac{1}{1+\exp {\color{black}((Q\theta^*)^{\top}p^k_{2:N})}} 
+\frac{1}{1+\exp{\color{black}(-(Q\theta^*)^{\top}p^k_{2:N})}}
=1$ always holds. 
Moreover, $\mathbb{E}[V_k]=0$ and $|V_k|\leq 1$ for all $k=1,\ldots,K$. 
Then, according to \eqref{eq:grad-MLE}, we have $\nabla\ell(\theta^*)=-{\color{black}Q^{\top}}P^{\top}V$ 
and 
\begin{equation*}
\left\|{\color{black}Q}\nabla\ell(\theta^*) \right\|_{(\Sigma_D+\lambda\mathbf{I})^{-1}}^2=({\color{black}Q}\nabla\ell(\theta^*))^{\top}(\Sigma_D+\lambda\mathbf{I})^{-1}({\color{black}Q}\nabla\ell(\theta^*))=V^{\top}P(\Sigma_D+\lambda\mathbf{I})^{-1}P^{\top}V=V^{\top}MV, 
\end{equation*}
where {\color{black} the second equality holds from the fact $QQ^{\top}=\mathbf{I}_{N-1}$}, and     $M:=P(\Sigma_D+\lambda\mathbf{I})^{-1}P^{\top}$.

Let the eigenvalue decomposition of $P^{\top}P$ be $P^{\top}P=U\Lambda U^{\top}$, where $U^{\top}U=\mathbf{I}$ and $\Lambda$ is a diagonal matrix whose diagonal elements are the eigenvalues of $P^{\top}P$. 
Then, $\Sigma_D=\frac{1}{K}P^{\top}P=\frac{1}{K}U\Lambda U^{\top}$. 
We first perform the following simplification: 
\begin{align*}
M &= P\Big(\frac{1}{K}U\Lambda U^{\top}+\lambda UU^{\top}\Big)^{-1}P^{\top}=PU\Big(\frac{\Lambda}{K}+\lambda\mathbf{I}\Big)^{-1}U^{\top}P^{\top},\\
M^2&=\left[PU\Big(\frac{\Lambda}{K}+\lambda\mathbf{I}\Big)^{-1}U^{\top}P^{\top}\right]^2=PU\Big(\frac{\Lambda}{K}+\lambda\mathbf{I}\Big)^{-1}\Lambda\Big(\frac{\Lambda}{K}+\lambda\mathbf{I}\Big)^{-1}U^{\top}P^{\top}. 
\end{align*}
We then bound the trace and the operator norm of the matrix $M$ by the cycle property as: 
\begin{equation*}
\begin{aligned}
{\rm Tr}(M)&={\rm Tr}\left(PU\Big(\frac{\Lambda}{K}+\lambda\mathbf{I}\Big)^{-1}U^{\top}P^{\top}\right)={\rm Tr}\left(\Big(\frac{\Lambda}{K}+\lambda\mathbf{I}\Big)^{-1}U^{\top}P^{\top}PU\right)={\rm Tr}\left(\Big(\frac{\Lambda}{K}+\lambda\mathbf{I}\Big)^{-1}\Lambda\right)\\
&=\sum_{i=1}^{{\color{black}N-1}}\frac{1}{\lambda_i(P^{\top}P)/K+\lambda}\lambda_i(P^{\top}P)
=K\sum_{i=1}^{{\color{black}N-1}}\frac{\lambda_i(P^{\top}P)}{\lambda_i(P^{\top}P)+K\lambda}\leq K\cdot{\rm rank}(P^{\top}P)=KR_D,
\end{aligned}
\end{equation*}
\begin{equation*}
\begin{aligned}
{\rm Tr}(M^2)
&={\rm Tr}\left(PU\Big(\frac{\Lambda}{K}+\lambda\mathbf{I}\Big)^{-1}\Lambda\Big(\frac{\Lambda}{K}+\lambda\mathbf{I}\Big)^{-1}U^{\top}P^{\top}\right)={\rm Tr}\left(\Big(\frac{\Lambda}{K}+\lambda\mathbf{I}\Big)^{-1}\Lambda\Big(\frac{\Lambda}{K}+\lambda\mathbf{I}\Big)^{-1}U^{\top}P^{\top}PU\right)\\
&={\rm Tr}\left(\Big(\frac{\Lambda}{K}+\lambda\mathbf{I}\Big)^{-1}\Lambda\Big(\frac{\Lambda}{K}+\lambda\mathbf{I}\Big)^{-1}\Lambda\right)
=\sum_{i=1}^{{\color{black}N-1}}\left[\frac{1}{\lambda_i(P^{\top}P)/K+\lambda}\lambda_i(P^{\top}P)\right]^2\\
&=K^2\sum_{i=1}^{{\color{black}N-1}}\left[\frac{\lambda_i(P^{\top}P)}{\lambda_i(P^{\top}P)+K\lambda}\right]^2\leq K^2\cdot{\rm rank}(P^{\top}P)=K^2R_D, 
\end{aligned}  
\end{equation*}
where $\lambda_i(P^{\top}P)$ represents the $i$-th eigenvalue of $P^{\top}P$, i.e., the $i$-th entry on the diagonal of $\Lambda$, and the rank 
satisfies $R_D:={\rm rank}(\Sigma_D)={\rm rank}(P^{\top}P)={\rm rank}(P)\leq\min\{K,{\color{black}N-1}\}$.  
We also bound the operator norm $\| \cdot \|_{\rm op}$ of matrix $M$ as 
\begin{align*}
\left\|M\right\|_{\rm op}=\max_{i=1,\ldots,{\color{black}N-1}} \ \lambda_i(M)=\max_{i=1,\ldots,{\color{black}N-1}} \ \frac{1}{\lambda_i(P^{\top}P)/K+\lambda}\lambda_i(P^{\top}P)\leq K.  
\end{align*}

Since $\mathbb{E}[V]=0$ and $|V_k|\leq 1$, Hoeffding’s Lemma implies that each $V_k$ is a sub-Gaussian random variable satisfying $\mathbb{E}[e^{tV_k}]\leq e^{t^2/2}$ for all $t\in\mathbb{R}$, $k=1,\ldots,K$. 
Given the independence of $V_k$ for $k=1,\ldots,K$, for any $t=[t_1,\ldots,t_K]^{\top}\in\mathbb{R}^K$, we have 
\begin{align*}
\mathbb{E}[e^{t^{\top}V}]=\mathbb{E}[e^{\sum_{k=1}^Kt_kV_k}]=\mathbb{E}[e^{t_1V_1}
e^{t_2V_2}\cdots e^{t_KV_K}]\leq e^{t_1^2/2}
e^{t_2^2/2}\cdots e^{t_K^2/2}=e^{\|t\|_2^2/2}. 
\end{align*}
Moreover, it can be easily verified that the matrix $M=P(\Sigma_D+\lambda\mathbf{I})^{-1}P^{\top}$ is positive semi-definite. 
Then, by applying the Hanson–Wright inequality \citep[Theorem 1]{daniel2012tail}, we obtain 
\begin{equation*}
    \mathbb{P}\left[V^{\top}MV > {\rm Tr}(M)+2\sqrt{{\rm Tr}(M^2)t}+2t\|M\|_{\rm op}\right]\leq e^{-t}, \ \forall t>0. 
\end{equation*}
For any $0<\delta<1$, by setting $t=-\log\delta>0$, it follows that, with probability at least $1-\delta$, 
\begin{align}\label{eq:unknown-term}
\left\|{\color{black}Q}\nabla\ell(\theta^*) \right\|_{(\Sigma_D+\lambda\mathbf{I})^{-1}}^2=V^{\top}MV
&\leq {\rm Tr}(M)+2\sqrt{-{\rm Tr}(M^2)\log\delta}-2\log\delta\|M\|_{\rm op}\nonumber\\
&\leq KR_D+2K\sqrt{-R_D\log\delta}-2K\log\delta. 
\end{align}

\noindent
\underline{Step 3.} 
By \eqref{eq:Delta} and \eqref{eq:unknown-term}, we have
\begin{align*}
\left\|\Delta \right\|^2_{\Sigma_D+\lambda\mathbf{I}}
&\leq\frac{2}{\omega K}\left\|{\color{black}Q}\nabla\ell(\theta^*) \right\|_{(\Sigma_D+\lambda\mathbf{I})^{-1}}\left\|\Delta \right\|_{\Sigma_D+\lambda\mathbf{I}}+\lambda\left(\bar{c}\sqrt{N-1}\right)^2\\
&\leq\frac{2}{\omega K}\sqrt{KR_D+2K\sqrt{-R_D\log\delta}-2K\log\delta}\  \left\|\Delta \right\|_{\Sigma_D+\lambda\mathbf{I}}\ +\lambda\left(\bar{c}\sqrt{N-1}\right)^2. 
\end{align*}
Solving the quadratic inequality above yields the following conclusion
\begin{equation*}
\left\|\Delta \right\|_{\Sigma_D+\lambda\mathbf{I}} 
\leq \sqrt{\frac{R_D+2\sqrt{-R_D\log\delta}-2\log\delta}{\omega^2 K}}
+\sqrt{\frac{R_D+2\sqrt{-R_D\log\delta}-2\log\delta}{\omega^2 K}+\lambda\left(\bar{c}\sqrt{N-1}\right)^2}. 
\end{equation*}
{\color{black}Finally, the fact  $\|\tilde{\Delta}\|_{Q^{\top}(\Sigma_D+\lambda \mathbf{I})Q}=
\|Q\tilde{\Delta}\|_{\Sigma_D+\lambda \mathbf{I}}=
\|\Delta\|_{\Sigma_D+\lambda \mathbf{I}}$ implies the desired results. 
}
\end{proof}

\begin{proposition}\label{thm:2-infty}
For any $\lambda>0$ and 
{\color{black}any $\delta\in(0,1)$}, 
with probability at least $1-\delta$, we have 
\begin{align}\label{eq:2-infty}
\left\| \frac{\hat{\alpha}^{\rm MLE}}{\hat{\sigma}_{\rm MLE}}-\frac{\alpha^*}{\sigma^*} \right\|_{\infty}
&\leq
\left\| \frac{\hat{\alpha}^{\rm MLE}}{\hat{\sigma}_{\rm MLE}}-\frac{\alpha^*}{\sigma^*} \right\|_{2} \nonumber\\
&\leq
\frac
{\sqrt{\frac{R_D+2\sqrt{-R_D\log\delta}-2\log\delta}{\omega^2 K}}
+\sqrt{\frac{R_D+2\sqrt{-R_D\log\delta}-2\log\delta}{\omega^2 K}+\lambda\left(\bar{c}\sqrt{N-1}\right)^2}}
{\sqrt{\lambda_{\rm min}\left(\Sigma_D+\lambda\mathbf{I}\right)}},
\end{align}
where $\lambda_{\rm min}\left(\Sigma_D+\lambda\mathbf{I}\right)$ denotes the minimum eigenvalue of the matrix $\Sigma_D+\lambda\mathbf{I}$. 
We refer to $\left\| \frac{\hat{\alpha}^{\rm MLE}}{\hat{\sigma}_{\rm MLE}}-\frac{\alpha^*}{\sigma^*} \right\|_{2}$ as the $\ell_{2}$ statistical error and 
$\left\| \frac{\hat{\alpha}^{\rm MLE}}{\hat{\sigma}_{\rm MLE}}-\frac{\alpha^*}{\sigma^*} \right\|_{\infty}$ as the $\ell_{\infty}$ statistical error. 
\end{proposition}

\begin{proof}
Firstly, by the Rayleigh quotient inequality, we have for any $x\in\mathbb{R}^{{\color{black}N-1}}$ that 
\begin{equation*}
\lambda_{\min}\left(\Sigma_D+\lambda \mathbf{I}\right)\|x\|_2^2\leq x^{\top}\left(\Sigma_D+\lambda \mathbf{I}\right)x. 
\end{equation*} 
As $\Sigma_D=\frac{1}{K}P^{\top}P$, 
$\Sigma_D+\lambda\mathbf{I}$ is a symmetric positive definite matrix, with  $\lambda_{\rm min}\left(\Sigma_D+\lambda\mathbf{I}\right)\geq\lambda>0$. 
By Theorem~\ref{thm:alpha-sigma-matrix}, we have
\begin{equation*}
\begin{aligned}
\left\| \frac{\hat{\alpha}^{\rm MLE}}{\hat{\sigma}_{\rm MLE}}-\frac{\alpha^*}{\sigma^*} \right\|_{2}
&{\color{black}=\left\|Q \left(\frac{\hat{\alpha}^{\rm MLE}}{\hat{\sigma}_{\rm MLE}}-\frac{\alpha^*}{\sigma^*}\right) \right\|_{2}}\\
&\leq \frac{1}{\sqrt{\lambda_{\rm min}\left(\Sigma_D+\lambda\mathbf{I}\right)}}
\left\| {\color{black}Q}\left(\frac{\hat{\alpha}^{\rm MLE}}{\hat{\sigma}_{\rm MLE}}-\frac{\alpha^*}{\sigma^*} \right)\right\|_{\Sigma_D+\lambda\mathbf{I}}\\
&{\color{black}=\frac{1}{\sqrt{\lambda_{\rm min}\left(\Sigma_D+\lambda\mathbf{I}\right)}}\left\| \frac{\hat{\alpha}^{\rm MLE}}{\hat{\sigma}_{\rm MLE}}-\frac{\alpha^*}{\sigma^*} \right\|_{Q^{\top}(\Sigma_D+\lambda\mathbf{I})Q}}\\
&\leq\frac{\sqrt{\frac{R_D+2\sqrt{-R_D\log\delta}-2\log\delta}{\omega^2 K}}
+\sqrt{\frac{R_D+2\sqrt{-R_D\log\delta}-2\log\delta}{\omega^2 K}+\lambda\left(\bar{c}\sqrt{N-1}\right)^2}}
{\sqrt{\lambda_{\rm min}\left(\Sigma_D+\lambda\mathbf{I}\right)}},  
\end{aligned} 
\end{equation*}
{\color{black}where the first equality holds 
as 
$\hat{\alpha}^{\rm MLE}_1=\alpha^*_1=0$. } 

Secondly, since $\left\| \frac{\hat{\alpha}^{\rm MLE}}{\hat{\sigma}_{\rm MLE}}-\frac{\alpha^*}{\sigma^*} \right\|_{\infty}\leq\left\| \frac{\hat{\alpha}^{\rm MLE}}{\hat{\sigma}_{\rm MLE}}-\frac{\alpha^*}{\sigma^*} \right\|_{2}$, we obtain the desired result. 
\end{proof}

Theorem \ref{thm:alpha-sigma-matrix} and Proposition \ref{thm:2-infty} not only provide convergence guarantees for the MLE 
{\color{black}estimator of the}
{\color{black}VNM utility model with response error,}
but also provide quantitative error bounds that can guide the design of 
{\color{black}the lottery pairs}
in the pairwise comparison dataset. 
We 
{\color{black}next}
consider several specific cases 
{\color{black}corresponding to different dataset structures}.

\begin{assumption}\label{assump:i-ii}
We consider 
{\color{black}the following two
cases  
regarding 
the informativeness (or redundancy) of the pairwise comparison lotteries
in the dataset $D$:}

\begin{itemize}

\item[(i)] 
The pairs of lotteries in the dataset $D$ are well-designed such that $\lambda_{\rm min}(\Sigma_D) = \bar{\lambda} > 0$. 
In this case, $\Sigma_D$ is invertible 
{\color{black}and}
$R_D={\rm rank}
(\Sigma_D) = {\color{black}N-1}$.
{\color{black}Consequently,}
the 
{\color{black}regularization}
parameter $\lambda$ in \eqref{eq:2-infty} can be set to $0$ 
{\color{black}and}
$\lambda_{\rm min}(\Sigma_D+\lambda\mathbf{I})=\bar{\lambda}$. 

\item[(ii)]  
The pairs of lotteries in the dataset $D$ are not well-designed 
{\color{black}so}
that $\lambda_{\rm min}(\Sigma_D) = 0$. 
In this case, $\Sigma_D$ is not invertible 
{\color{black}and}
$R_D={\rm rank}(\Sigma_D)< {\color{black}N-1}$.
{\color{black}Consequently,}
the 
{\color{black}regularization}
parameter $\lambda$ in \eqref{eq:2-infty} cannot be set to $0$. 
{\color{black}For any}
$\lambda>0$, we have $\lambda_{\rm min}(\Sigma_D+\lambda\mathbf{I})=\lambda$. 

\end{itemize}
\end{assumption}

{\color{black}
Assumption~\ref{assump:i-ii} concerns the rank of the empirical information matrix
$\Sigma_D$, which reflects the informativeness of the observed pairwise comparison dataset $D$.
In the following Corollaries~\ref{coll:bound-param} and~\ref{coll:bound-utility}, we show that
Assumption~\ref{assump:i-ii} provides sufficient conditions for establishing
finite-sample convergence rates of the MLE. }

\begin{corollary}\label{coll:bound-param}

\begin{itemize}

\item[(a)]
{\color{black}Under 
Assumption~\ref{assump:i-ii}(i)}, for any 
{\color{black}$\delta\in(0,1)$},
with probability at least $1-\delta$, we have 
\begin{equation*}
\begin{aligned}
\left\| \frac{\hat{\alpha}^{\rm MLE}}{\hat{\sigma}_{\rm MLE}}-\frac{\alpha^*}{\sigma^*} \right\|_{\infty}
\leq
\left\| \frac{\hat{\alpha}^{\rm MLE}}{\hat{\sigma}_{\rm MLE}}-\frac{\alpha^*}{\sigma^*} \right\|_{2}
\leq
\frac{2}{\sqrt{\bar{\lambda}}}\sqrt{\frac{{\color{black}N-1}+2\sqrt{-{\color{black}(N-1)}\log\delta}-2\log\delta}{\omega^2 K}}
={\color{black}O\left(\sqrt{\frac{N}{K}}\right)}.
\end{aligned}
\end{equation*}

{\color{black}Furthermore, 
(a.1)
if the cardinality $N$ of the overall support set induced by all lotteries
is 
fixed 
(with $K$ increasing),
then the estimation error 
with respect to the utility values at the 
breakpoints 
converges to 0 at the rate $O\left(\frac{1}{\sqrt{K}}\right)$; 
(a.2) if 
$N$ grows
asymptotically slower than $K$, i.e., $N=o(K)$, then the estimation error converges to 0 as $K\rightarrow\infty$.
}

\item[(b)] 
{\color{black}Under 
Assumption~\ref{assump:i-ii}(ii),} 
for any 
{\color{black}$\lambda>0$ and any $\delta\in(0,1)$},
with probability at least $1-\delta$, we have 
\begin{equation*}
\begin{aligned}
\left\| \frac{\hat{\alpha}^{\rm MLE}}{\hat{\sigma}_{\rm MLE}}-\frac{\alpha^*}{\sigma^*} \right\|_{\infty}
\leq
\left\| \frac{\hat{\alpha}^{\rm MLE}}{\hat{\sigma}_{\rm MLE}}-\frac{\alpha^*}{\sigma^*} \right\|_{2}
& 
\leq 
\frac{2}{\sqrt{\lambda}}
\sqrt{\frac{{\color{black}N-1}+2\sqrt{-{\color{black}(N-1)}\log\delta}-2\log\delta}{\omega^2 K}
+\lambda\left(\bar{c}\sqrt{N-1}\right)^2}\\
&=
{\color{black}O\left(\sqrt{\frac{N}{K}}\right)}
+O\left(\bar{c}\sqrt{N}\right).
\end{aligned}
\end{equation*}
{\color{black}In this case, the estimation error may fail to converge to zero as $K\rightarrow \infty$, since the pairwise comparison queries may be repeated or linearly dependent. 
}


\end{itemize}

\end{corollary}

Finally, with the monotonicity, concavity, and Lipschitz continuity properties of the VNM utility function in $\mathcal{U}_c$, 
we quantify the error between the MLE-based adjusted VNM utility function and the true adjusted VNM utility function, 
{\color{black}measured in the}
Kolmogorov distance.

\begin{theorem}\label{thm:bound-utility}
For any $\lambda>0$ and 
{\color{black}any $\delta\in(0,1)$}, 
with probability at least $1-\delta$, we have 
\begin{equation*}
\dd_K\left(\frac{\hat{u}_{\rm MLE}}{\hat{\sigma}_{\rm MLE}},\frac{u^*}{\sigma^*}\right)\leq  
L \bar{c} \mu_N+
\frac
{\sqrt{\frac{R_D+2\sqrt{-R_D\log\delta}-2\log\delta}{\omega^2 K}}
+\sqrt{\frac{R_D+2\sqrt{-R_D\log\delta}-2\log\delta}{\omega^2 K}+\lambda\left(\bar{c}\sqrt{N-1}\right)^2}}
{\sqrt{\lambda_{\rm min}\left(\Sigma_D+\lambda\mathbf{I}\right)}},
\end{equation*}
where 
$\dd_K(u,\tilde{u}):=\sup\limits_{y\in[y_1,y_N]}\left|u(y)-\tilde{u}(y) \right|$ is the Kolmogorov distance between any two utility functions $u,\tilde{u}\in\mathcal{U}_c$, $L$ is the Lipschitz modulus defined in \eqref{eq:U_c}, and $\mu_N:=\max_{j=2,\ldots,N} (y_{j}-y_{j-1})$. 

\end{theorem}


\begin{proof}
{
Denote $C_0:=\frac
{\sqrt{\frac{R_D+2\sqrt{-R_D\log\delta}-2\log\delta}{\omega^2 K}}
+\sqrt{\frac{R_D+2\sqrt{-R_D\log\delta}-2\log\delta}{\omega^2 K}+\lambda\left(\bar{c}\sqrt{N-1}\right)^2}}
{\sqrt{\lambda_{\rm min}\left(\Sigma_D+\lambda\mathbf{I}\right)}}$.  
By Proposition~\ref{thm:2-infty}, we have $|\frac{\hat{\alpha}^{\rm MLE}_j}{\hat{\sigma}_{\rm MLE}}-\frac{\alpha^*_j}{\sigma^*} |\leq C_0$, $j=1,\ldots,N$.
, i.e., 
\begin{equation*}
\frac{\hat{\alpha}^{\rm MLE}_j}{\hat{\sigma}_{\rm MLE}}-C_0 \leq \frac{u^*(y_j)}{\sigma^*}\leq\frac{\hat{\alpha}^{\rm MLE}_j}{\hat{\sigma}_{\rm MLE}}+C_0,\ j=1,\ldots,N. 
\end{equation*} 
By the monotonicity and the Lipschitz continuity of both $\hat{u}_{\rm MLE}$ and $u^*$, 
we have 
\begin{equation*}
\begin{aligned}
\dd_K\left(\frac{\hat{u}_{\rm MLE}}{\hat{\sigma}_{\rm MLE}},\frac{u^*}{\sigma^*}\right)
=&\sup_{j=2,\ldots,N}\ \sup_{y\in[y_{j-1},y_j]} \left|\frac{\hat{u}_{\rm MLE}(y)}{\hat{\sigma}_{\rm MLE}}-\frac{u^*(y)}{\sigma^*} \right|\\
\leq& \sup_{j=2,\ldots,N}\ \max\left\{\left|\frac{\hat{u}_{\rm MLE}(y_{j-1})}{\hat{\sigma}_{\rm MLE}}-\frac{u^*(y_{j})}{\sigma^*} \right|, \left|\frac{\hat{u}_{\rm MLE}(y_j)}{\hat{\sigma}_{\rm MLE}}-\frac{u^*(y_{j-1})}{\sigma^*} \right|\right\}\\
\leq& \sup_{j=2,\ldots,N}\ \max\left\{ 
\left|\frac{\hat{\alpha}^{\rm MLE}_{j-1}}{\hat{\sigma}_{\rm MLE}}-\left(\frac{\hat{\alpha}^{\rm MLE}_j}{\hat{\sigma}_{\rm MLE}}+C_0\right)\right|, 
\left|\frac{\hat{\alpha}^{\rm MLE}_{j-1}}{\hat{\sigma}_{\rm MLE}}-\left(\frac{\hat{\alpha}^{\rm MLE}_j}{\hat{\sigma}_{\rm MLE}}-C_0\right)\right|
\right.,\\
&\left.
\left|\frac{\hat{\alpha}^{\rm MLE}_{j}}{\hat{\sigma}_{\rm MLE}}-\left(\frac{\hat{\alpha}^{\rm MLE}_{j-1}}{\hat{\sigma}_{\rm MLE}}+C_0 \right)\right|,
\left|\frac{\hat{\alpha}^{\rm MLE}_{j}}{\hat{\sigma}_{\rm MLE}}-\left(\frac{\hat{\alpha}^{\rm MLE}_{j-1}}{\hat{\sigma}_{\rm MLE}}-C_0 \right)\right|
\right\}\\
=&\max_{j=2,\ldots,N} \left(\frac{\hat{\alpha}^{\rm MLE}_{j}}{\hat{\sigma}_{\rm MLE}}-\frac{\hat{\alpha}^{\rm MLE}_{j-1}}{\hat{\sigma}_{\rm MLE}}\right)+C_0\\ 
\leq& \max_{j=2,\ldots,N} \frac{L(y_{j}-y_{j-1})}{\hat{\sigma}_{\rm MLE}}+C_0\leq  L\bar{c}\left[\max_{j=2,\ldots,N}(y_{j}-y_{j-1})\right]+C_0, 
\end{aligned}
\end{equation*}
which gives the desired conclusion. 
}
\end{proof}

Note that the error bound in Theorem~\ref{thm:bound-utility} 
{\color{black}consists of}
two terms. The first term bounds the difference between a general VNM utility function and its piecewise linear lower 
approximation 
under the metric $\dd_K$.
The second term quantifies the statistical  
{\color{black}estimation}
error incurred by the MLE procedure in obtaining $(\hat{u}_{\rm MLE},\hat{\sigma}_{\rm MLE})$, as analyzed in Theorem~\ref{thm:alpha-sigma-matrix} 
{\color{black}and Proposition~\ref{thm:2-infty}}. 
The first term 
{\color{black}vanishes}
when the DM's true VNM utility function is piecewise linear 
and its breakpoints align with the support set $\mathbb{Y}$ 
{\color{black}induced by}
the paired lotteries. 
The following corollary then establishes  
convergence rates under 
{\color{black}different conditions on 
the informativeness (or redundancy)
of the paired lotteries.
}



\begin{corollary}\label{coll:bound-utility}

\begin{itemize}

\item [(a)]
{\color{black}Under 
Assumption~\ref{assump:i-ii}(i)}, for any 
{\color{black}$\delta\in(0,1)$},
with probability at least $1-\delta$, we have 
\begin{equation*}
\dd_K\left(\frac{\hat{u}_{\rm MLE}}{\hat{\sigma}_{\rm MLE}},\frac{u^*}{\sigma^*}\right)\leq 
L \bar{c} \mu_N+
\frac{2}{\sqrt{\bar{\lambda}}} 
\sqrt{\frac{{\color{black}N-1}+2\sqrt{-{\color{black}(N-1)}\log\delta}-2\log\delta}{\omega^2 K}}.  
\end{equation*}
Moreover, 
if the 
{\color{black}support set $\mathbb{Y}$ of the lotteries is sufficiently dense such that}
$\mu_N=\max\limits_{j=2,\ldots,N} \ (y_{j}-y_{j-1})= O\left(\frac{1}{N}\right)
$, then for any 
{\color{black}$\delta\in(0,1)$},
with probability at least $1-\delta$, 
{\color{black}it follows that}
\begin{equation*}
\dd_K\left(\frac{\hat{u}_{\rm MLE}}{\hat{\sigma}_{\rm MLE}},\frac{u^*}{\sigma^*}\right) =
{\color{black}O\left(\frac{1}{N}\right)+O\left(\sqrt{\frac{N}{K}}\right).}
\end{equation*}
{\color{black}
Furthermore, if $N=o(K)$, then the 
estimation error of the utility function converges to 0 as both $K,N\rightarrow\infty$.
}

\item[(b)]  
{\color{black}Under 
Assumption~\ref{assump:i-ii}(ii) and $\mu_N= O\left(\frac{1}{N}\right)
$,}
for any 
{\color{black}$\delta\in(0,1)$},
with probability at least $1-\delta$, we have  
\begin{equation*}
\dd_K\left(\frac{\hat{u}_{\rm MLE}}{\hat{\sigma}_{\rm MLE}},\frac{u^*}{\sigma^*}\right) 
=
{\color{black}O\left(
{\color{black}\frac{1}{N}}
\right)+O\left(\sqrt{\frac{N}{K}}\right)}
+ O\left(\bar{c}\sqrt{N}\right). 
\end{equation*}

\end{itemize}

\end{corollary}

{\color{black}
\subsection{Interpretation and implications of the statistical results}

This subsection interprets the statistical error bounds established above.
We explain the roles of the key assumptions (the informativeness of paired lotteries and the balance between the number of breakpoints $N$ and the number of queries $K$) and highlight their implications for practical questionnaire design.

\medskip
\noindent
\textbf{Roles of $N$ and $K$ (model complexity and sample size).}
In our setting, $K$ is the number of 
observed pairwise comparisons, while $N:=|\mathbb{Y}|$ is the number of distinct payoff support points induced by these comparisons, which also determines the number of breakpoints in the piecewise linear approximation $\hat u_{\rm MLE}$.
As discussed before Corollary~\ref{coll:bound-utility}, the error bound in Theorem~\ref{thm:bound-utility} consists of an approximation term and an estimation term.
Under the dense support condition $\mu_N=O(1/N)$ in Corollary~\ref{coll:bound-utility}(a), the total error is of order $O(1/N)$ (approximation) plus $O(\sqrt{N/K})$ (estimation).
This expression makes it clear that $K$ plays the role of a sample size, while $N$ plays the role of an effective model complexity (the dimension $N-1$ of adjusted utility values at breakpoints). 
Such an interplay between model complexity and sample size is standard in statistical learning and non-parametric estimation, which
reflects the classical approximation–estimation tradeoff~\cite{vapnik2013nature,hastie2009elements}.

\medskip
\noindent
\textbf{Geometric role of $\Sigma_D$ and why directional diversity matters.}
Recall that the empirical information matrix $\Sigma_D=\frac1K P^{\top}P$ in \eqref{eq:SigmaD} is the Gram matrix of the reduced probability-mass difference vectors $p^k_{2:N}$, $k=1,\ldots,K$.
Hence, ${\rm rank}(\Sigma_D)$ and $\lambda_{\rm min}(\Sigma_D)$ quantify how informative the pairwise comparison queries
are by measuring whether $\{p^k_{2:N}\}_{k=1}^K$ span $\mathbb{R}^{N-1}$ and how well-conditioned this span is.
This is exactly the content of Assumption~\ref{assump:i-ii}: the well-designed case requires $\lambda_{\rm min}(\Sigma_D)=\bar\lambda>0$, whereas the not well-designed case corresponds to $\lambda_{\rm min}(\Sigma_D)=0$.

Consider the case where the number of breakpoints $N$ is fixed. Adding a new 
query
contributes an additional positive semidefinite term to the unscaled Gram matrix $P^\top P$.
Let $P_K\in\mathbb{R}^{K\times (N-1)}$ be the matrix formed by the first $K$ 
queries
and let $p^{K+1}_{2:N}$ denote the new vector.
Then $P_{K+1}$ is obtained by appending $p^{K+1}_{2:N}$ as an additional row, and
\[
P_{K+1}^{\top}P_{K+1}
=
P_{K}^{\top}P_{K}+p^{K+1}_{2:N}\big(p^{K+1}_{2:N}\big)^{\top}\succeq P_{K}^{\top}P_{K}.
\]
Consequently, $\lambda_{\rm min}(P_{K+1}^{\top}P_{K+1})\ge \lambda_{\rm min}(P_{K}^{\top}P_{K})$.
Equivalently, writing the empirical
information matrix $\Sigma_{D,K}:=\frac{1}{K}P_K^{\top}P_K$, we obtain
\[
\Sigma_{D,K+1}
=
\frac{K}{K+1}\Sigma_{D,K}
+
\frac{1}{K+1}
p^{K+1}_{2:N}\big(p^{K+1}_{2:N}\big)^{\top}.
\]
This rank-one update reveals how the $(K+1)$-th query affects the empirical information matrix. 
For any unit vector $v\in\mathbb{R}^{N-1}$, the Rayleigh quotient satisfies
\begin{equation}\label{eq-rayleigh}
    v^{\top}\Sigma_{D,K+1}v
=
\frac{K}{K+1}\,v^{\top}\Sigma_{D,K}v
+\frac{1}{K+1}\big(p^{K+1}_{2:N}{}^{\top}v\big)^2,
\end{equation}
so the gain along direction $v$ is governed by the squared projection $\big(p^{K+1}_{2:N}{}^{\top}v\big)^2$.
This suggests a simple design principle for the $(K\!+\!1)$-th paired-lottery query.
If $\lambda_{\min}(\Sigma_{D,K})=0$ (rank deficiency), then adding a comparison with
$p^{K+1}_{2:N}\in{\rm span}\{p^k_{2:N}\}_{k=1}^K$ is redundant and cannot 
change
the zero eigenvalue.

To see this, we consider a case 
where 
$\text{rank}(\Sigma_{D,K})=N-2$ (i.e., $\Sigma_{D,K}$ has
only one zero eigenvalue). 
The condition
$\lambda_{\min}(\Sigma_{D,K})=0$ 
implies that 
there exists a nonzero vector 
$v\in\mathbb{R}^{N-1}$ lies in the null space of $\Sigma_{D,K}$, i.e.,
$v\in\ker(\Sigma_{D,K})$ and $(p^k_{2:N})^\top v = 0$, $k=1,\ldots,K$.
If the new comparison vector satisfies $p^{K+1}_{2:N}\in\mathrm{span}\{p^k_{2:N}\}_{k=1}^K$,
then $(p^{K+1}_{2:N})^\top v = 0$ as well. 
By
the update formula \eqref{eq-rayleigh}, we obtain 
$v^{\top}\Sigma_{D,K+1}v = 0$ and $\lambda_{\min}(\Sigma_{D,K+1})=0$.
To increase the rank, $p^{K+1}_{2:N}$ 
must
have a nonzero component outside the current span (equivalently, a nonzero projection onto $\ker(\Sigma_{D,K})$), so that the 
span of identifying directions 
expands.
For instance, we can take 
$p^{K+1}_{2:N}=\alpha u$ where $u$ is the unit basis vector of the null space of $\Sigma_{D,K}$,
i.e., a direction that is orthogonal to all 
$\{p^k_{2:N}\}_{k=1}^K$.
Therefore,
\[
\lambda_{\min}(\Sigma_{D,K+1})
=
\min\left\{\frac{\alpha^2}{K+1},\ \frac{K}{K+1}\lambda_{+}(\Sigma_{D,K})\right\},
\]
where $\lambda_{+}(\Sigma_{D,K})$ denotes the smallest positive eigenvalue of $\Sigma_{D,K}$.
In particular, $\lambda_{\min}(\Sigma_{D,K+1})=\frac{K}{K+1}\lambda_{+}(\Sigma_{D,K})$ when $\alpha^2=K\lambda_{+}(\Sigma_{D,K})$.

If instead $\lambda_{\min}(\Sigma_{D,K})>0$ (i.e., $\Sigma_{D,K}$ is full rank), a natural greedy strategy for designing a new query is to 
strengthen
the least-informed direction by choosing $p^{K+1}_{2:N}$ to have a large projection on an eigenvector associated with $\lambda_{\min}(\Sigma_{D,K})$, thereby improving conditioning.
As a special but instructive benchmark, one may consider 
an orthogonal (``vertical-cutting'') design, where each new $p^k_{2:N}$ is chosen approximately orthogonal to the previous ones; once $K=N-1$ linearly independent directions are collected, $\Sigma_{D,K}$ becomes full rank with a well-conditioned spectrum, achieving identification with the minimum number of informative queries.

\medskip
\noindent
\textbf{Relation to the Fisher information matrix and D-/E-optimal design.}
Notice that the normalization constraint implies $\theta_1=0$, so the free variables are only $\theta_{2:N}=[\theta_2,\ldots,\theta_N]^\top$.
We can then compute the generalized Fisher information matrix in the reduced parameter space as follows~\cite{YU2012789,saure2019ellipsoidal}: 
\[
\mathcal{I}_K(\theta)\;=-\mathbb{E}[\nabla^2_{\theta_{2:N}} \ell(\theta)]=\;\sum_{k=1}^K w_k(\theta_{2:N})\,  p^k_{2:N}(p^k_{2:N})^\top ,
\]
where $w_k(\theta_{2:N})\;=\;\frac{\exp(\theta_{2:N}^\top p^k_{2:N})}{\bigl(1+\exp(\theta_{2:N}^\top p^k_{2:N})\bigr)^2}$ and the Hessian is taken with respect to the free variables $\theta_{2:N}$. 
By Lemma \ref{lemma:gradient-bound},
$|\theta_{2:N}^\top p^k_{2:N}|\le 2\bar{c}$, which implies $w_k(\theta_{2:N})\geq \frac{1}{2+2\exp(2\bar{c})}>0$.
Consequently, 
$\lambda_{\min}(\Sigma_D)>0$ immediately implies $\lambda_{\min}(\mathcal{I}_K(\theta))>0$
for all such $\theta$, and lower bounds on $\lambda_{\min}(\Sigma_D)$ translate into
curvature/variance control through $\mathcal{I}_K(\theta)^{-1}$.

This also clarifies the connection to classical optimal experimental design.
D-optimal design maximizes $\det(\mathcal{I}_K(\theta))=\lambda_{\rm min}(\mathcal{I}_K(\theta))\prod_{i=2}^{N-1} \lambda_i(\mathcal{I}_K(\theta))$ (equivalently, it minimizes the D-error $\det(\mathcal{I}^{-1}_K(\theta))^{1/(N-1)}$), which
controls the volume of the asymptotic confidence ellipsoid.
In contrast, 
E-optimal design maximizes
$\lambda_{\min}(\mathcal{I}_K(\theta))$, focusing on the worst-informed direction~\cite{kiefer1959optimum,pukelsheim2006optimal}.
Our condition $\lambda_{\min}(\Sigma_D)>0$ 
can be viewed as 
an unweighted, parameter-robust analogue of the
E-optimal requirement, and it is necessary for any D-optimal criterion (since $\det(\cdot)>0$ implies full rank).
\citet{saure2019ellipsoidal} relate D-error criteria to
posterior covariance and Fisher-information-based approximations,
and use such (approximate) D-error measures to guide adaptive question selection.
By contrast, our matrix $\Sigma_D$ is defined purely from the lottery differences and
is used primarily to establish identifiability and finite-sample guarantees.
While Fisher-based criteria depend on the unknown parameter (typically through plug-in or posterior estimates),
$\Sigma_D$ is parameter-robust in its definition and can therefore be used
to analyze and guide query design without relying on estimated model parameters.

\medskip
\noindent
\textbf{Balancing the number of queries $K$ and the number of breakpoints $N$.}
To approximate a nonlinear non-parametric utility function, the 
set of breakpoints 
$\mathbb{Y}$ must be refined adaptively, which means that new breakpoints
are gradually introduced and the dimension of the adjusted utility 
values at breakpoints
increases (e.g., from $N$ to $N+1$). 
However, such refinement immediately enlarges the parameter space.
Even when the existing $K$ queries
fully span the previous reduced space $\mathbb{R}^{N-1}$,
they generally fail to span the expanded space $\mathbb{R}^{N}$.
Consequently, the empirical information matrix $\Sigma_D$
is typically 
rank-deficient in the 
enlarged
dimension,
and additional informative queries are required to regain identifiability. 

Moreover, merely ensuring $\lambda_{\min}(\Sigma_{D})>0$
is not sufficient to guarantee statistical convergence of the MLE. 
From the 
bounds in 
Theorem~\ref{thm:bound-utility}
and Corollary~\ref{coll:bound-utility},
the total error 
scales as $O({1}/{N})+O(\sqrt{N/K})$. 
For this total error bound to vanish asymptotically, 
it is necessary that $N\to\infty$ and $N/K \to 0$,
that is, $N=o(K)$.
This rate 
highlights 
that each additional degree of freedom
must be supported by a 
sufficiently faster growth 
in informative 
queries. 
This requirement reflects the classical principle that
model complexity must be controlled relative to the sample size.
Interpreting $K$ as the sample size and $N$ as a measure of 
effective model complexity, 
the condition $N=o(K)$ ensures that the approximation error
and the estimation error decrease simultaneously.
This 
scaling is consistent with the principle of sublinear growth of
complexity measures in learning theory~\cite{vapnik2013nature}, 
and with the classical approximation–estimation tradeoff~\cite{hastie2009elements},
which requires model complexity to increase slowly enough to preserve consistency in non-parametric estimation. 




\medskip
\noindent
\textbf{Implications for questionnaire design under $N=o(K)$.}
The preceding discussion shows that adaptive refinement of $\mathbb{Y}$
must be coordinated with a sufficient increase in informative queries.
This naturally suggests a structured multi-round design.

Suppose we start at round 1 with a piecewise linear approximation 
utility function with $N_1$
breakpoints.
At round $r$, suppose the current 
piecewise linear approximation has $N_r$ breakpoints in $\mathbb{Y}^{(r)}$.
We first generate $N_r-1$ queries supported on  $\mathbb{Y}^{(r)}$ whose reduced probability-mass difference vectors are as diverse as possible, ideally mutually orthogonal, so that the empirical information matrix $\Sigma_D$ remains well-conditioned
in the current $(N_r-1)$-dimensional parameter space.
Then, we introduce a single exploration query that adds one new breakpoint~\footnote{\color{black}
For example, it can be achieved via RUS, RRUS methods~\cite{armbruster2015decision}, or a modified RUS method by setting the new breakpoint as the midpoint of the widest sub-intervals between the current breakpoints.}.
Thus, at the end of round $r$, we generate $N_r$ queries in total.
In the next round $r+1$, the cardinality of 
the breakpoints
increases to $N_{r+1}=N_r+1$
due to the new breakpoint introduced in the last exploration query of the previous round.
At 
round $r+1$, the first $N_r$ mutually orthogonal queries are newly generated, rather than reusing the $N_r-1$ queries 
in
the previous round. 
Thus, after $R$ rounds, we have 
\[
N_R = N_1 + R -1 =O(R),\ \ K_R = \sum_{r=1}^{R} N_r = O(R^2),
\]
which implies $N_R = o(K_R)$.
This hybrid strategy simultaneously enforces directional diversity
for identification (Assumption~\ref{assump:i-ii}(i))
and controls complexity growth for statistical consistency
(Corollary~\ref{coll:bound-utility}(a)).

In the noiseless setting, the above multi-round design can be viewed as
an extension of \cite{WLX2026}.
In their framework, the number of breakpoints $N$ is fixed and each round
generates $N$ mutually orthogonal cuts in the parameter space to achieve exponential contraction of the ambiguity set.
In contrast, our procedure allows $N$ to increase adaptively across rounds
while preserving directional orthogonality within each round.


When response errors are present, similar exact geometric contraction
may no longer be achieved;
nevertheless, the statistical error bounds established in Theorem~\ref{thm:bound-utility} and Corollary~\ref{coll:bound-utility}  ensure consistency,
and the proposed 
questionnaire design strategy
maintains the identification condition
required for convergence.
}

\section{MLE-based VNM utility maximization with application to portfolio optimization}\label{sec-decision-making}


{\color{black}
In Section
\ref{sec-MLE}, we develop an MLE-based preference elicitation procedure to estimate a VNM utility model from pairwise comparison data in the presence of response error, discuss the existence and uniqueness of the 
MLE estimates
and then 
derive statistical error bounds and convergence rates for the estimated utility. 
In this section, we apply the elicited utility function
to a concrete decision-making problem to illustrate its practical implications.
Specifically, in the context of portfolio optimization, we consider maximizing the expected utility of random wealth using the MLE-based VNM utility function $\hat{u}_{\rm MLE}$. 
By introducing two new risk measures, Preference-at-Risk (PaR) and Preference-Robust-at-Risk (PRaR), we show that the optimization problem that maximizes the expected MLE-based VNM utility 
is robust to both the DM's response errors and the estimation errors from MLE in a probabilistic sense.
}


Suppose there are $S$ risky assets with random return rates $\xi_1,\ldots, \xi_S\in\mathcal{L}^p(\Omega, \mathcal{F}, \mathbb{P};\mathbb{R})$, 
{\color{black}and}
a risk-free asset with return rate $\xi_0=0$. 
{\color{black}Let}
$\xi:=[\xi_0,\xi_1,\ldots,\xi_S]^{\top}$,  
and let $x:=[x_0,x_1,\ldots,x_S]^\top\in\mathcal{X}$ denote the portfolio, where the 
{\color{black}feasible set}
$\mathcal{X}$ is a closed convex set defined as 
\begin{equation*}
    \mathcal{X}:=\left\{x\in\mathbb{R}^{S+1}\ \bigg|\ 
    \sum_{s=0}^{S}x_s=W_0,\ 0\leq x_0\leq W_0, \  0\leq x_s\leq c_s W_0,\ s=1,\ldots,S \right\}. 
\end{equation*}
Here, $W_0\in[0,\bar{b}]$ denotes the DM's total investment budget, and 
each $c_s\in[0,1]$ denotes the maximum proportion 
{\color{black}of the budget that is allowed to be allocated to}
the $s$-th risky asset. 
Let $h(x,\xi): \mathcal{X}\times\mathcal{L}^p(\Omega, \mathcal{F}, \mathbb{P};\mathbb{R}) \to \mathcal{L}^p(\Omega, \mathcal{F}, \mathbb{P};[0,\bar{b}])$ 
{\color{black}denote}
the random wealth 
{\color{black}generated}
from portfolio $x \in \mathcal{X}$ 
under the random return 
{\color{black}rate}
$\xi$. 
{\color{black}In what follows},
we 
{\color{black}set}
$h(x, \xi) := x^{\top} (\mathbf{1}+\xi)$, where $\mathbf{1}$ denotes the vector of ones. 
We first introduce a preference-based portfolio optimization problem that maximizes the 
{\color{black}expected MLE-based VNM utility}:
\begin{equation}\label{eq:DM-MLE-utility}
\max_{x\in \mathcal{X}} \ 
\mathbb{E}_{\mathbb{P}_{\xi}}\left[\hat{u}_{\rm MLE}(h(x,\xi))\right],
\end{equation}
where $\hat{u}_{\rm MLE}$ is the 
MLE-based VNM utility function 
in \eqref{eq:N-piecewise utility function}.
Notice that 
\begin{equation}\label{eq-joint-exp}
\mathbb{E}_{\mathbb{P}_{\xi}}\left[\hat{u}_{\rm MLE}(h(x,\xi))\right]=\mathbb{E}_{\mathbb{P}_{\varepsilon}}\left[\mathbb{E}_{\mathbb{P}_{\xi}}\left[\hat{u}_{\rm MLE}(h(x,\xi))\right] +\varepsilon\right] -\mathbb{E}[\varepsilon],
\end{equation}
{\color{black}when}
$\varepsilon$ 
{\color{black}follows a Gumbel distribution independent of $\xi$}.
{\color{black}Here, the expectation of the random utility $\hat{u}_{\rm MLE}(h(x,\xi))+\varepsilon$  jointly with respect to the random return $\xi$ and the response error $\varepsilon$ is equivalent to
}
the expectation of the MLE-based VNM utility $\hat{u}_{\rm MLE}(h(x,\xi))$. 
{\color{black}The joint expectation in \eqref{eq-joint-exp}} is also referred to as the \emph{expected expected utility} \citep{boutilier2003foundations}. 
This is closely related to Bayesian risk optimization, which imposes an additional risk functional to the expected cost objective function~\cite{wu2018bayesian,lin2022bayesian}, and related to Bayesian distributionally robust optimization~\cite{shapiro2023bayesian}. 
It is also connected to the decision rule approach for contextual decision-making, which introduces an outer expectation over contextual variables 
{\color{black}in addition to}
the conditional expected cost \citep{ban2019big}, as well as to the contextual optimization from a risk management perspective \citep{tao2025risk}.

Let $\mathscr{U}:=\{u\in\mathcal{L}^1([0,\bar{b}])\mid u: [0,\bar{b}]\to [0,1] \}$
{\color{black}denote a set of admissible utility functions}. 
We consider the 
{\color{black}VNM utility model with response error}
defined in \eqref{eq:DM-random-model}. To measure the randomness 
{\color{black}induced by}
response error, we propose a new risk measure 
called Preference-at-Risk (PaR) in Section \ref{sec:PaR}. 
We then examine the relationship between PaR and the MLE-based VNM utility function $\hat{u}_{\rm MLE}$. 
In Section \ref{sec:PRaR}, 
{\color{black}motivated by the}
ambiguity of the 
{\color{black}MLE-based optimal solutions}
discussed in Proposition \ref{thm:optimal-set-MLE}, we define a Preference-Robust-at-Risk (PRaR) measure for a given ambiguity set $\mathcal{U}\subseteq\mathscr{U}$. 
We also explore its 
{\color{black}connection}
with 
{\color{black}$\hat{u}_{\rm MLE}$}.
As shown in Proposition~\ref{prop:par} and Proposition~\ref{prop:prar}, the three portfolio optimization problems, maximizing the MLE-based VNM utility function, maximizing the PaR measure, and maximizing the PRaR measure, respectively, admit the same optimal solutions.

\subsection{Preference-at-Risk against response error}\label{sec:PaR}

Motivated by the well-known Value-at-Risk {\color{black}(VaR)} measure, we introduce 
the following 
Preference-at-Risk (PaR) to 
{\color{black}quantify}
the impact of response error in a probabilistic sense.

\begin{definition}[Preference-at-Risk (PaR)]\label{def:PaR}
Consider the 
{\color{black}VNM utility model with response error} defined in \eqref{eq:DM-random-model}. 
For any confidence level 
{\color{black}$\delta\in(0,1)$}
and any $u \in \mathscr{U}$,  
the Preference-at-Risk measure $\rho_{u,\delta}: \mathcal{L}^p(\Omega, \mathcal{F}, \mathbb{P};[0,\bar{b}])\to\mathbb{R}$ 
is defined as  
\begin{equation}\label{eq:PaR}
\rho_{u,\delta}(X):=\max \left\{ \eta \mid \mathbb{P}_{\varepsilon}\left(\mathbb{E}_{\mathbb{P}_X}[u(X)]+\varepsilon\geq \eta\right)\geq1-\delta\right\},\; \forall X\in\mathcal{L}^p(\Omega, \mathcal{F}, \mathbb{P};[0,\bar{b}]), 
\end{equation}
where $\varepsilon$ denotes 
the 
response error and is 
assumed to follow a ${\rm Gumbel}(0, \sigma)$ distribution, with CDF given by $F_{\varepsilon}(\cdot)$ 
in \eqref{eq:gumbel-cdf}. 
\end{definition}

Since the 
{\color{black}VNM utility model with response error in}
\eqref{eq:DM-random-model} 
includes a random 
term, we naturally introduce a chance constraint in \eqref{eq:PaR} to account for this randomness, in contrast to deterministic utility 
{\color{black}models}.
This chance constraint ensures that, for a random outcome $X$, the DM's 
{\color{black}VNM expected utility with response error},
exceeds a threshold $\eta$ with probability at least $1 - \delta$. 
The PaR measure is then defined as the maximal 
threshold 
{\color{black}that satisfies the chance constraint.}
Notably, this is 
{\color{black}analogous}
to the VaR measure, as both quantify the worst acceptable outcome at a given confidence level. 
It is also related to the utility-based shortfall risk measure, which evaluates risk by requiring that the expected loss under an increasing convex loss function does not exceed a given threshold~\cite{guo2019distributionally,delage2022shortfall}. 
{\color{black}
It is straightforward to verify that the PaR measure is law-invariant, as it depends on $X$ only through $\mathbb{E}[u(X)]$. 
If $u$ is non-decreasing, concave, and translation-invariant, then $\rho_{u,\delta}$ satisfies monotonicity, concavity, and translation invariance, respectively. 
Nevertheless, positive homogeneity fails in general due to both the non-homogeneous nature of $u$ and the additive constant term arising from the response error (see the equivalent formulation in Proposition \ref{prop:par}). 
Subadditivity also does not hold in general, since it would require $u$ to be subadditive, which essentially reduces to the linear case. 
Consequently, PaR is not a coherent risk measure in general.
} 
Given the additive structure of the 
{\color{black}utility model}
with respect to the response error, we have the following proposition.

\begin{proposition}\label{prop:par}
For any 
{\color{black}$\delta\in(0,1)$}
and any $u \in \mathscr{U}$, we have 
\begin{align*}\label{eq:rho-u-closed-form}
\rho_{u,\delta}(X)=\mathbb{E}_{\mathbb{P}_X}[u(X)]+F_{\varepsilon}^{-1}(\delta), \; \forall X\in\mathcal{L}^p(\Omega, \mathcal{F}, \mathbb{P};[0,\bar{b}]). 
\end{align*}
Moreover, consider the MLE-based VNM utility function $\hat{u}_{\rm MLE}$ in \eqref{eq:N-piecewise utility function}. For any confidence level 
{\color{black}$\delta\in(0,1)$},
the following two problems share the same set of optimal solutions: 
\begin{equation}\label{eq:par-port}
\max\limits_{x\in \mathcal{X}}\ \rho_{\hat{u}_{\rm MLE},\delta}(x^{\top}(\mathbf{1}+\xi))
\quad\Longleftrightarrow\quad
\max\limits_{x\in \mathcal{X}}\ \mathbb{E}_{\mathbb{P}_{\xi}}\left[\hat{u}_{\rm MLE}(x^{\top}(\mathbf{1}+\xi))\right].
\end{equation} 
\end{proposition}

\begin{proof}
By the definition in \eqref{eq:PaR}, for any $u\in\mathscr{U}$, $\rho_{u,\delta}(X)$ is the optimal value of the following maximization problem: 
\begin{subequations}\label{eq:Pra-initial}
\begin{align}
\max_{\eta\in\mathbb{R}}\ & \eta\\
\label{const:par-chance}
\text{s.t.}\ 
& \mathbb{P}_{\varepsilon}\left[\mathbb{E}_{\mathbb{P}_X}[u(X)]+\varepsilon\geq \eta\right]\geq1-\delta.  
\end{align}
\end{subequations} 
Considering the CDF of $\varepsilon$ given by $F_{\varepsilon}(\cdot)$ in \eqref{eq:gumbel-cdf}, the constraint \eqref{const:par-chance} is equivalent to 
\begin{equation*}
    \eta\leq \mathbb{E}_{\mathbb{P}_X}[u(X)]+F_{\varepsilon}^{-1}(\delta), 
\end{equation*}
which obviously implies that the optimal value of \eqref{eq:Pra-initial} is $\mathbb{E}_{\mathbb{P}_X}[u(X)]+F_{\varepsilon}^{-1}(\delta)$, and thus $\rho_{u,\delta}(X)=\mathbb{E}_{\mathbb{P}_X}[u(X)]+F_{\varepsilon}^{-1}(\delta)$. 
\end{proof}

From \eqref{eq:par-port}, we observe that maximizing the 
PaR 
measure 
{\color{black}of}
random wealth, which incorporates 
response error 
{\color{black}through}
a quantile 
{\color{black}criterion},
is equivalent to directly maximizing the corresponding expected MLE-based VNM utility function 
in \eqref{eq:DM-MLE-utility}. 
This equivalence shows that, although optimization problem \eqref{eq:DM-MLE-utility} 
{\color{black}uses}
a deterministic estimate of the DM's utility, it remains robust to the 
response error in a probabilistic sense. 
{\color{black}Consequently}, 
$\hat{u}_{\rm MLE}$ provides a reliable estimate 
of the DM's preferences
{\color{black}given}
their potentially inconsistent 
{\color{black}choices.}

\subsection{Preference-Robust-at-Risk against response error and estimation error}\label{sec:PRaR}

Beyond the 
PaR 
measure in Section~\ref{sec:PaR}, 
{\color{black}which captures the}
randomness in response error, we define a robust counterpart 
{\color{black}risk measure}
that further accounts for the estimation error 
{\color{black}arising from}
the MLE procedure.   

\begin{definition}[Preference-Robust-at-Risk, PRaR]\label{def:PraR}
Consider the 
{\color{black}VNM utility model with response error}
defined in \eqref{eq:DM-random-model} and an ambiguity set $\mathcal{U}\subseteq\mathscr{U}$. 
For any confidence level 
{\color{black}$\delta\in(0,1)$},
the 
Preference-Robust-at-Risk (PRaR) 
measure $\rho_{\mathcal{U},\delta}: \mathcal{L}^p(\Omega, \mathcal{F}, \mathbb{P};[0,\bar{b}])\to\mathbb{R}$ 
is defined as 
\begin{equation}\label{eq:PraR}
\rho_{\mathcal{U},\delta}(X):=\max \left\{ \eta \mid \mathbb{P}_{\varepsilon}\left[\mathbb{E}_{\mathbb{P}_X}[u(X)]+\varepsilon\geq \eta\right]\geq1-\delta,\ \forall u\in \mathcal{U}\right\}. 
\end{equation}
\end{definition}

The idea of PRaR is closely related to preference robust optimization (PRO)~\cite{armbruster2015decision,guo2024utility}.
{\color{black}Specifically}, 
at a 
{\color{black}given}
confidence level, 
we require that for a random outcome $X$, the 
{\color{black}expected VNM utility with response error
exceeds} the threshold $\eta$ for all 
utility functions $u\in\mathcal{U}$, 
{\color{black}thus ensuring}
preference robustness. 
PRaR is also closely related to stochastic dominance~\cite{dentcheva2007stability,rudolf2008optimization}, which requires that payoff comparison
{\color{black}relations}
hold for all monotone (and/or concave) utility functions. 
The following proposition shows that 
PRaR 
$\rho_{\mathcal{U},\delta}(\cdot)$ is the worst-case 
PaR 
$\rho_{u,\delta}(\cdot)$ over the ambiguity set $\mathcal{U}$. This observation allows us to 
{\color{black}compute}
PRaR 
{\color{black}by evaluating $\rho_{u,\delta}$}
for each $u\in\mathcal{U}$. 


\begin{proposition}\label{prop:PaR-PraR}
Given an ambiguity set $\mathcal{U}\subseteq\mathscr{U}$, for any 
{\color{black}$\delta\in(0,1)$},
we have 
\begin{equation*}
\rho_{\mathcal{U},\delta}(X) =  \inf_{ u\in \mathcal{U}}  \rho_{u,\delta}(X),\ \forall X\in\mathcal{L}^p(\Omega, \mathcal{F}, \mathbb{P};[0,\bar{b}]). 
\end{equation*}  
\end{proposition}

\begin{proof}
On one hand, for 
any $u\in\mathcal{U}$, we have the relation between the two sets that 
\begin{equation*}
    \{ \eta \mid \mathbb{P}_{\varepsilon}\left[\mathbb{E}_{\mathbb{P}_X}[u(X)]+\varepsilon\geq \eta\right]\geq1-\delta,\ \forall u\in \mathcal{U}\}\subseteq\{ \eta\mid \mathbb{P}_{\varepsilon}\left[\mathbb{E}_{\mathbb{P}_X}[u(X)]+\varepsilon\geq \eta\right]\geq1-\delta\}.  
\end{equation*}
Thus, we have
\begin{align*}
\rho_{u,\delta}(X)
&=\max_\eta \{ \eta \mid \mathbb{P}_{\varepsilon}\left[\mathbb{E}_{\mathbb{P}_X}[u(X)]+\varepsilon\geq \eta\right]\geq1-\delta\}\\
&\geq\max_\eta \{ \eta \mid \mathbb{P}_{\varepsilon}\left[\mathbb{E}_{\mathbb{P}_X}[u(X)]+\varepsilon\geq \eta\right]\geq1-\delta,\ \forall u\in \mathcal{U}\}=\rho_{\mathcal{U},\delta}(X),  
\end{align*}
which implies that $\rho_{\mathcal{U},\delta}(X)\leq \inf\limits_{ u\in \mathcal{U}}\rho_{u,\delta}(X)$. 
On the other hand, note that $\inf\limits_{ u\in \mathcal{U}}\rho_{u,\delta}(X)\leq\rho_{u,\delta}(X)$, $\forall u\in\mathcal{U}$. 
Then we have  
\begin{align*}
\mathbb{P}_{\varepsilon}\left[\mathbb{E}_{\mathbb{P}_X}[u(X)]+\varepsilon\geq\inf_{ u\in \mathcal{U}}\rho_{u,\delta}(X)\right]\geq \mathbb{P}_{\varepsilon}\left[\mathbb{E}_{\mathbb{P}_X}[u(X)]+\varepsilon\geq\rho_{u,\delta}(X)\right] \geq 1-\delta, \ \forall u\in\mathcal{U},  
\end{align*}
which implies that $\inf\limits_{ u\in \mathcal{U}}\rho_{u,\delta}(X)$ is a feasible solution of \eqref{eq:Pra-initial}, and hence $\inf\limits_{ u\in \mathcal{U}}\rho_{u,\delta}(X)\leq\rho_{\mathcal{U},\delta}(X)$. 
\end{proof}

We then study properties of the PRaR measure 
{\color{black}over the set $\mathcal{U}^*_c$}
in \eqref{eq:MLE-u-opt-set},
{\color{black}consisting of}
all possible maximum likelihood estimates of the VNM utility function. 

\begin{lemma}\label{prop:lower-bound}
Consider the MLE-based VNM utility function $\hat{u}_{\rm MLE}$ in \eqref{eq:N-piecewise utility function} and the 
{\color{black}associated}
optimal solution set of VNM utility functions $\mathcal{U}^{*}_c$ 
{\color{black}defined}
in \eqref{eq:MLE-u-opt-set}.
{\color{black}Then,}
$$\min\limits_{u\in\mathcal{U}_c^*} u(y) = \hat{u}_{\rm MLE}(y),\ \ \forall y\in[0,\bar{b}].$$
\end{lemma}

\begin{proof}
On the one hand, since $\hat{u}_{\rm MLE}(\cdot)\in\mathcal{U}^*_c$, $\min\limits_{u\in\mathcal{U}^*_c} u(y)\leq \hat{u}_{\rm MLE}(y)$, $\forall y\in[0,\bar{b}]$. 
On the other hand, for any ${u}\in\mathcal{U}^*_c$, we have ${u}(y_j)=\hat{\alpha}^{{\rm MLE}}_{j}=\hat{u}_{\rm MLE}(y_j)$, $j=1,\ldots,N$. 
Then, by the concavity of $u$, we have 
\begin{eqnarray*}
{u}(y) &\geq & (1-\frac{y-y_{j}}{y_{j+1}-y_{j}}){u}(y_{j}) + \frac{y-y_{j}}{y_{j+1}-y_{j}}{u}(y_{j+1})  \\
&=& \frac{{u}(y_{j+1})-{u}(y_{j})}{y_{j+1}-y_{j}} (y-y_j)+{u}(y_{j}) = \frac{\hat{\alpha}^{{\rm MLE}}_{j+1}-\hat{\alpha}^{{\rm MLE}}_{j}}{y_{j+1}-y_j}(y-{y}_j)+\hat{\alpha}^{{\rm MLE}}_{j}= \hat{u}_{\rm MLE}(y),
\end{eqnarray*}
$\forall y\in [y_j,y_{j+1}]$, $j=1,\ldots,N-1$. 
Thus, $0\leq \hat{u}_{\rm MLE}(y)\leq {u}(y)$, $\forall y\in[y_1,y_N]$, $\forall u\in\mathcal{U}^{*}_c$, which implies that $\hat{u}_{\rm MLE}(y)\leq \min\limits_{u\in\mathcal{U}^*_c} u(y)$, $\forall y\in[0,\bar{b}]$.  
\end{proof}

{\color{black}Lemma~\ref{prop:lower-bound} essentially reflects the property that, within a class of monotonically increasing and concave functions with prescribed values on a finite set of breakpoints, the pointwise minimum function dominating these values is given by the piecewise-linear interpolation 
through the breakpoints~\cite{armbruster2015decision,hu2015robust,hu2019data,
haskell2018preference}.
}
{\color{black}By}
Lemma~\ref{prop:lower-bound}, we conclude that $\hat{u}_{\rm MLE}$ is the worst-case utility in $\mathcal{U}^{*}_c$, 
{\color{black}and it provides}
a uniform pointwise lower bound on all utility functions in $\mathcal{U}^{*}_c$. 
See Figure \ref{fig:lower-bound} for an illustrative example. 
{\color{black}Considering}
the two 
{\color{black}sources}
of randomness in \eqref{eq:PraR}, 
{\color{black}stemming from $\varepsilon$ and $X$,}
we 
{\color{black}obtain the following result on}
the worst-case utility in $\mathcal{U}_c^*$.

\begin{figure}[htbp]
    \centering
    \includegraphics[width=0.8\linewidth]{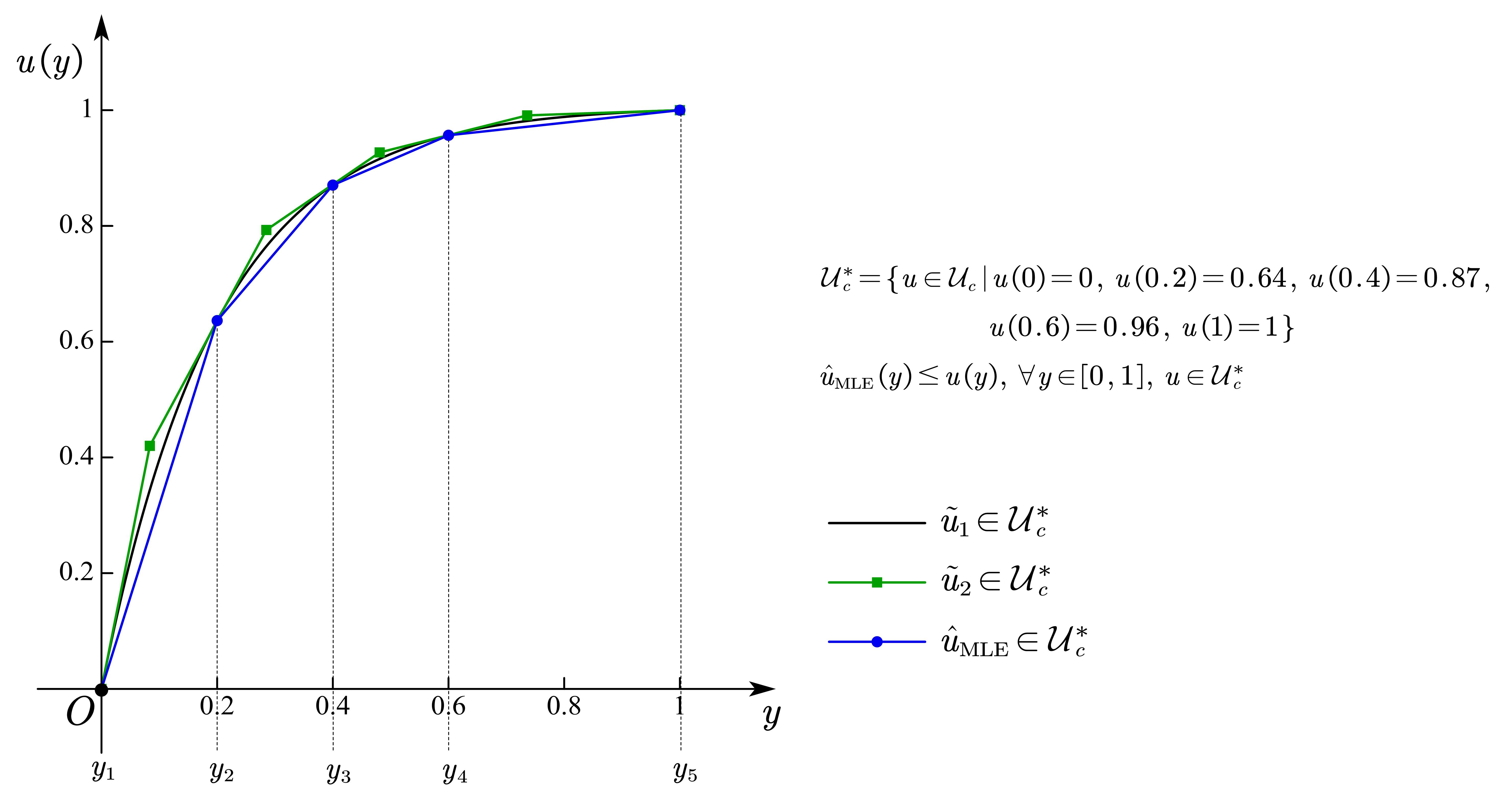}
    \caption{An illustrative example of $\hat{u}_{\rm MLE}$ as the pointwise lower bound of $\mathcal{U}^{*}_c$ when $N=5$ and $\bar{b}=1$.  
    }
    \label{fig:lower-bound}
\end{figure}


\begin{lemma}\label{lemma:inf} 
For any $\eta\in\mathbb{R}$, we have 
\begin{equation*}
\begin{aligned}
\inf\limits_{u\in\mathcal{U}_c^*} \mathbb{P}_{\varepsilon} \left[\mathbb{E}_{\mathbb{P}_{X}}[u(X)]+\varepsilon\geq\eta\right]
=\mathbb{P}_{\varepsilon} \left[\inf\limits_{u\in\mathcal{U}_c^*}\mathbb{E}_{\mathbb{P}_{X}}[u(X)]+\varepsilon\geq\eta\right]
=\mathbb{P}_{\varepsilon} \left[\mathbb{E}_{\mathbb{P}_{X}}[\hat{u}_{\rm MLE}(X)]+\varepsilon\geq\eta\right]. 
\end{aligned}
\end{equation*}
\end{lemma}

\begin{proof}

Since $\hat{u}_{\rm MLE}\in\mathcal{U}^{*}_c$, it is obvious that $\mathbb{E}_{\mathbb{P}_X}\left[\hat{u}_{\rm MLE}(X)\right]=\inf_{u\in\mathcal{U}^*_{c}}\mathbb{E}_{\mathbb{P}_X}\left[u(X)\right]$ by Lemma \ref{prop:lower-bound}, thus, for any $\eta\in\mathbb{R}$, 
\begin{equation*}
\begin{aligned}
\mathbb{P}_{\varepsilon} \left[\inf\limits_{u\in\mathcal{U}_c^*}\mathbb{E}_{\mathbb{P}_{X}}[u(X)]+\varepsilon\geq\eta\right]
=\mathbb{P}_{\varepsilon} \left[\mathbb{E}_{\mathbb{P}_{X}}[\hat{u}_{\rm MLE}(X)]+\varepsilon\geq\eta\right]
\geq\inf\limits_{u\in\mathcal{U}_c^*} \mathbb{P}_{\varepsilon} \left[\mathbb{E}_{\mathbb{P}_{X}}[u(X)]+\varepsilon\geq\eta\right]. 
\end{aligned}
\end{equation*}

On the other hand, for any $u\in\mathcal{U}^*_{c}$, by $\inf\limits_{u\in\mathcal{U}_c^*}\mathbb{E}_{\mathbb{P}_{X}}[u(X)]\leq\mathbb{E}_{\mathbb{P}_{X}}[u(X)]$, we have $\eta-\inf\limits_{u\in\mathcal{U}_c^*}\mathbb{E}_{\mathbb{P}_{X}}[u(X)]\geq\eta-\mathbb{E}_{\mathbb{P}_{X}}[u(X)]$, $\forall \eta\in\mathbb{R}$, and then  
\begin{equation*}
    \mathbb{P}_{\varepsilon} \left[\varepsilon\geq\eta-\inf\limits_{u\in\mathcal{U}_c^*}\mathbb{E}_{\mathbb{P}_{X}}[u(X)]\right]\leq\mathbb{P}_{\varepsilon} \left[\varepsilon\geq\eta-\mathbb{E}_{\mathbb{P}_{X}}[u(X)]\right],\; \forall \eta\in\mathbb{R},\; \forall u\in\mathcal{U}^*_{c}. 
\end{equation*}
This implies 
$$
\mathbb{P}_{\varepsilon} \left[\varepsilon\geq\eta-\inf\limits_{u\in\mathcal{U}_c^*}\mathbb{E}_{\mathbb{P}_{X}}[u(X)]\right]\leq\inf\limits_{u\in\mathcal{U}_c^*} \mathbb{P}_{\varepsilon} \left[\varepsilon\geq\eta-\mathbb{E}_{\mathbb{P}_{X}}[u(X)]\right],\; \forall \eta\in\mathbb{R}. 
$$  
Combining the two aspects gives the desired result. 
\end{proof}

Based on Lemma \ref{lemma:inf} and the assumption that $\varepsilon$ follows a Gumbel distribution, we 
{\color{black}derive the following}
result.

\begin{proposition}\label{prop:prar}

For any confidence level 
{\color{black}$\delta\in(0,1)$}, 
we have 
\begin{equation*}
    \rho_{\mathcal{U}^{*}_{c},\delta}(X)=\rho_{{\hat{u}_{\rm MLE}},\delta}(X)=\mathbb{E}_{\mathbb{P}_X}\left[\hat{u}_{\rm MLE}(X)\right]+F^{-1}_{\varepsilon}(\delta),\; \forall X\in\mathcal{L}^p(\Omega, \mathcal{F}, \mathbb{P};[0,\bar{b}]). 
\end{equation*}
Moreover, for any confidence level 
{\color{black}$\delta\in(0,1)$}, 
the following two problems share the same set of optimal solutions: 
\begin{equation}\label{eq:prar-port}
\max\limits_{x\in \mathcal{X}}\ \rho_{\mathcal{U}^{*}_{c},\delta}(x^{\top}(\mathbf{1}+\xi))
\quad\Longleftrightarrow\quad
\max\limits_{x\in \mathcal{X}}\ \mathbb{E}_{\mathbb{P}_{\xi}}\left[\hat{u}_{\rm MLE}(x^{\top}(\mathbf{1}+\xi))\right].
\end{equation} 
\end{proposition}

\begin{proof}
According to the definition in \eqref{eq:PraR}, $\rho_{{\color{black}\mathcal{U}^*_c},\delta}(X)$ is the optimal value of the following maximization problem: 
\begin{subequations}
\begin{align}
\max_{\eta\in\mathbb{R}}\ & \eta\\
\label{const:par-chance-1}
\text{s.t.}\ 
& \mathbb{P}_{\varepsilon}\left[\mathbb{E}_{\mathbb{P}_X}[u(X)]+\varepsilon\geq \eta\right]\geq1-\delta,\; \forall u\in{\color{black}\mathcal{U}^*_c}.   
\end{align}
\end{subequations} 
By Lemma \ref{lemma:inf} and the CDF of $\varepsilon$ given by $F_{\varepsilon}(\cdot)$ in \eqref{eq:gumbel-cdf}, constraint \eqref{const:par-chance-1} is equivalent to 
\begin{equation*}
    \eta-\mathbb{E}_{\mathbb{P}_{X}}[\hat{u}_{\rm MLE}(X)]\leq F_{\varepsilon}^{-1}(\delta). 
\end{equation*}
Maximizing $\eta$ gives $\rho_{\mathcal{U}^{*}_{c},\delta}(X)=\mathbb{E}_{\mathbb{P}_X}\left[\hat{u}_{\rm MLE}(X)\right]+F^{-1}_{\varepsilon}(\delta)=\rho_{{\hat{u}_{\rm MLE}},\delta}(X)$, where the second equality follows from Proposition \ref{prop:par}. 
\end{proof}

From \eqref{eq:prar-port}, 
{\color{black}it follows that}
maximizing the 
PRaR 
measure of random wealth, which 
{\color{black}accounts for}
both the DM's 
response error and the 
estimation error 
of the VNM utility function 
in $\mathcal{U}^*_c$, is equivalent to 
maximizing the expected MLE-based VNM utility function 
in \eqref{eq:DM-MLE-utility}. This equivalence further 
{\color{black}demonstrates}
the robustness of 
$\hat{u}_{\rm MLE}$ against both response errors and estimation errors.

\subsection{Sample average approximation}

We finally derive a tractable reformulation of \eqref{eq:DM-MLE-utility} using the sample average approximation (SAA). 
Assume we collect $T$ i.i.d. samples of the return 
{\color{black}rate} $\xi$, denoted 
{\color{black}by}
$\xi^{(1)},\xi^{(2)},\ldots,\xi^{(T)}$. 
The empirical distribution of $\xi$ is given by $\mathbb{P}[\xi=\xi^{(t)}]=\frac{1}{T}$, $t=1,\ldots,T$. 
Then, \eqref{eq:DM-MLE-utility} can be approximated by 
$$
\max\limits_{x\in\mathcal{X}}\frac{1}{T}\sum_{t=1}^T \hat{u}_{\rm MLE}(x^{\top}(\mathbf{1}+\xi^{(t)})).
$$

Furthermore, due to the piecewise linear structure of $\hat{u}_{\rm MLE}$ in \eqref{eq:N-piecewise utility function}, 
$\hat{u}_{\rm MLE}$ 
{\color{black}can be represented}
as the minimum of ${N}-1$ linear pieces:
\begin{equation*}
\hat{u}_{\rm MLE}(y)=\min\limits_{j=1,\ldots,N-1}\left[\hat{\beta}^{{\rm MLE}}_j(y-{y}_j)+\hat{\alpha}^{{\rm MLE}}_{j}\right].  
\end{equation*} 
By introducing auxiliary variables $\zeta_1,\ldots,\zeta_T$, we can reformulate 
{\color{black}the SAA of}
\eqref{eq:DM-MLE-utility} as the following linear programming problem: 
\begin{align*}
\max\limits_{x\in\mathbb{R}^{S+1}} \; 
& \frac{1}{T}\sum_{t=1}^T \zeta_t\\ 
\text{s.t.}\; & \zeta_t\leq \hat{\beta}^{{\rm MLE}}_j(x^{\top}(\mathbf{1}+\xi^{(t)})-{y}_j)+\hat{\alpha}^{{\rm MLE}}_{j}, \ j=1,\ldots,N-1,\ t=1,\ldots,T,\\
& \sum_{s=0}^{S} x_s=W_0,\\
&0\leq x_0\leq W_0,\ 0\leq x_s\leq c_s W_0 ,\ s=1, \ldots, S.
\end{align*}




\section{Numerical test}\label{sec-num-test}

We 
{\color{black}conduct}
three groups of numerical 
{\color{black}experiments}
to evaluate the effectiveness of the proposed preference elicitation method 
{\color{black}based on}
discrete choices with response error, on a simulated 
DM. 
Each group of 
{\color{black}experiments}
validates a 
{\color{black}specific}
consideration in our method: 
1) the modeling of the scale parameter $\sigma$ as a decision variable rather than a fixed constant;  
2) the incorporation of 
{\color{black}structural preference}
information (monotonicity, concavity, and Lipschitz continuity) of the DM's VNM utility function 
{\color{black}into the preference elicitation process}; 
and 3) the design of pairwise comparison lotteries in dataset $D$. 

To evaluate the 
{\color{black}accuracy}
of the elicited utility preference, we 
{\color{black}use}
the $\ell_2$ and $\ell_\infty$ statistical errors introduced in Proposition \ref{thm:2-infty}, namely $\left\| \frac{\hat{\alpha}^{\rm MLE}}{\hat{\sigma}_{\rm MLE}}-\frac{\alpha^*}{\sigma^*} \right\|_{2}$ and $\left\| \frac{\hat{\alpha}^{\rm MLE}}{\hat{\sigma}_{\rm MLE}}-\frac{\alpha^*}{\sigma^*} \right\|_{\infty}$, which quantify the deviation from the DM’s true 
{\color{black} utility with response error.}
Here, $(\hat{\alpha}^{\rm MLE}, \hat{\sigma}_{\rm MLE})$ is obtained by solving the MLE problem \eqref{eq:MLE-re-final}, while the 
$(\alpha^*, \sigma^*)$ serves as the 
{\color{black}true}
benchmark for comparison.



\subsection{Test setting}\label{sec:num-setting} 

We begin by 
{\color{black}specifying}
the parameter settings used in the numerical tests: 
1) the Lipschitz modulus in \eqref{eq:U_c} is set to $L=10$; 2) the upper bound of $1/\sigma$ is set to $\bar{c}=100$; 
{\color{black}and}
3) the maximal outcome that the DM cares about is set to 100,000 (i.e., $\bar{b}=100000$). 
In the simulation, we assume 
a virtual DM who acts according to a pre-defined true 
{\color{black}VNM utility model with response error},
$U^*(\cdot)=\mathbb{E}[u^*(\cdot)]+\varepsilon$, 
where the response error  
$\varepsilon\sim \text{Gumbel}(0,\sigma^*)$
{\color{black}with}
$\sigma^*=10$, and the true VNM utility function $u^*$ is set as: 
\begin{equation}\label{eq:num-true-VNM-u}
    u^*(y)=\frac{1}{1-e^{-6}}(1-e^{-6\times10^{-5}y}),\; \forall y\in[0,100000]. 
\end{equation}
{\color{black}The function in \eqref{eq:num-true-VNM-u}}
is monotonically increasing, concave, Lipschitz continuous, and normalized to the range $[0, 1]$ over the interval $[0,100000]$. 
{\color{black} The exponential function has been widely used in parametric utility modeling because it exhibits constant absolute risk aversion~\cite{choi2011multi}. Here, we normalize it by rescaling with a constant so that it fits the range $[0,1]$. 
It is also commonly adopted as a ground-truth utility function in numerical studies of preference elicitation~\cite{liu2025preference}, preference robust optimization~\cite{guo2024utility},
and likelihood-based robust parameter estimation~\cite{nguyen2020distributionally}. 
}

\textbf{Dataset.}
We then 
{\color{black}construct}
the dataset $D=\{(W_k, Y_k, Z_k)\mid k=1,\ldots,K\}$, which consists of $K$ pairwise lottery comparisons. 
We first randomly draw {\color{black}$N-2$} points from $[0,100000]$ 
{\color{black}restricted to}
multiples of 100; 
together with $0$ and $100000$,  
{\color{black}these points}
constitute the 
{\color{black}support set}
$\mathbb{Y}$. 
We then generate 
risky lotteries $W_k$ and $Y_k$ for $k=1,\ldots,K$, by randomly selecting up to three points from 
$\mathbb{Y}$ and assigning random probabilities so that 
{\color{black}they sum to}
$1$. 
For each such pair 
$(W_k,Y_k)$, we compute $U^*(W_k)=\mathbb{E}[u^*(W_k)]+\varepsilon_{1k}$ and $U^*(Y_k)=\mathbb{E}[u^*(Y_k)]+\varepsilon_{2k}$, respectively, where $u^*$ is defined in \eqref{eq:num-true-VNM-u} and $\varepsilon$ is 
{\color{black}sampled}
from 
${\rm Gumbel}(0,\sigma^*)$, $\sigma^*=10$. If $U^*(W_k)$ 
{\color{black}exceeds}
$U^*(Y_k)$, 
we set $Z_k=1$ to simulate the DM's response; otherwise, 
{\color{black}we set}
$Z_k=-1$. 
An illustrative example of a dataset with $K=5$ is provided in Table \ref{tab:dataset}. 


\begin{table}[htbp]
\small
\centering
\caption{An example of a randomly generated 
{\color{black}pairwise comparison}
dataset $D$}
\label{tab:dataset}
\renewcommand{\arraystretch}{1.3} 
\begin{tabular}{c | p{0.4\textwidth} | p{0.4\textwidth} | c}
\toprule
\multicolumn{1}{c|}{\centering $k$} & 
\multicolumn{1}{c|}{\centering $W_k$} & 
\multicolumn{1}{c|}{\centering $Y_k$} & 
\multicolumn{1}{c}{\centering $Z_k$} \\
\midrule
1 & A 30\% chance at winning \$42,500 and a 40\% chance at winning \$50,200; otherwise \$0.
  & A 45\% chance at winning \$3,800;  otherwise \$0. 
  & -1 \\
2 & A 40\% chance at winning \$88,200 and a 35\% chance at winning \$36,000; otherwise \$0. 
  & A 5\% chance at winning \$52,800; otherwise \$0.
  & 1 \\
3 & A 55\% chance at winning \$11,400 and a 10\% chance at winning \$88,200; otherwise \$0.  
  & A 95\% chance at winning \$4,000; otherwise \$0.
  & -1 \\
4 & A 35\% chance at winning \$32,800; otherwise \$0. 
  & A 40\% chance at winning \$7,800; otherwise \$0.  
  & -1 \\
5 & A 40\% chance at winning \$47,700 and a 20\% chance at winning \$7,800; otherwise \$0.
  & A 70\% chance at winning \$26,900 and a 5\% chance at winning \$95,000; otherwise \$0. 
  & 1 \\
\bottomrule
\end{tabular}
\end{table}

\subsection{Modeling scale parameter $\sigma$ as a variable versus as a constant}


In this test, we compare the elicitation errors of the MLE-based utility function under two approaches: 
modeling the scale parameter $\sigma$ of the response error 
as 
{\color{black}
a free}
variable 
as in \eqref{eq:DM-random-model}, 
versus 
{\color{black}fixing $\sigma$ to a misspecified constant value.}


Given a randomly generated dataset $D$ consisting of $K$ pairwise comparison queries, we 
{\color{black}jointly}
estimate the VNM utility values $\hat{\alpha}^{\rm MLE}$ and the scale parameter $\hat{\sigma}_{\rm MLE}$ by solving \eqref{eq:MLE-re-final}. Meanwhile, we consider two benchmark models by solving \eqref{eq:MLE-re-final} with $\sigma$ fixed {\color{black}to a misspecified value of} 
either $1$ or $100$. We set the number of 
breakpoints in $\mathbb{Y}$ to $N=50$ and consider $K\in\{50,100,200,400,600,800,1000,2000,3000,4000,5000\}$. 
As $K$ increases, pairwise lotteries are generated sequentially and the dataset 
is 
constructed 
{\color{black}by progressively adding new pairwise comparison queries}
without replacement. 
For each $K$, we compute the
$\ell_2$ {\color{black}statistical} error $\| \frac{\hat{\alpha}^{\rm MLE}}{\sigma}-\frac{\alpha^*}{\sigma^*} \|_{2}$ and the $\ell_{\infty}$ {\color{black}statistical} error $\| \frac{\hat{\alpha}^{\rm MLE}}{\sigma}-\frac{\alpha^*}{\sigma^*} \|_{\infty}$ between the elicited  $\hat{\alpha}^{\rm MLE}$ and the true $\alpha^*$. 
In our model, we set $\sigma=\hat{\sigma}^{\rm MLE}$,
{\color{black}whereas}
in the two benchmark models 
{\color{black}we fix $\sigma$}
to $1$ or $100$. 
We 
{\color{black}repeat}
the test 30 times independently and 
average the resulting
$\ell_2$  and $\ell_{\infty}$ {\color{black}statistical} errors to mitigate the impact of randomness in dataset generation. 
The results are reported in Figure \ref{fig:num-const}.

From Figure \ref{fig:num-const}, we observe the following: 
\begin{itemize}
\item When we set $\sigma$ as a free variable, the $\ell_2$ and $\ell_{\infty}$ {\color{black}statistical} errors of the elicited VNM utility values and scale parameter $(\hat{\alpha}^{\rm MLE}, \hat{\sigma}_{\rm MLE})$ 
{\color{black}decrease and}
converge to $0$ at a 
{\color{black}noticeable}
rate as the number of 
queries increases, 
{\color{black}consistent with the theoretical guarantees in}
Proposition~\ref{thm:2-infty} and Corollary~\ref{coll:bound-param}. 
{\color{black}In our experiments, $N$ is fixed while $K$ increases, 
which corresponds to the setting of Corollary~\ref{coll:bound-param}(a.1).
}


\item 
The $\ell_2$ and $\ell_{\infty}$ {\color{black}statistical} errors of the two benchmark models with a fixed {\color{black}but misspecified} scale parameter 
{\color{black}are highly sensitive to}
the chosen value. 
Moreover, they 
{\color{black}vary only slightly}
as the number of queries increases, showing no clear convergence trend. 

\end{itemize}
{\color{black}In conclusion, these results suggest that it is preferable to treat $\sigma$ as a free variable and estimate it jointly with $\alpha$, rather than fixing $\sigma$ to a misspecified value, when the modeler lacks sufficiently accurate information about the response error $\varepsilon$.}

\begin{figure}[htbp]
    \centering
    \subfigure
    {
    \includegraphics[width=0.48\linewidth]{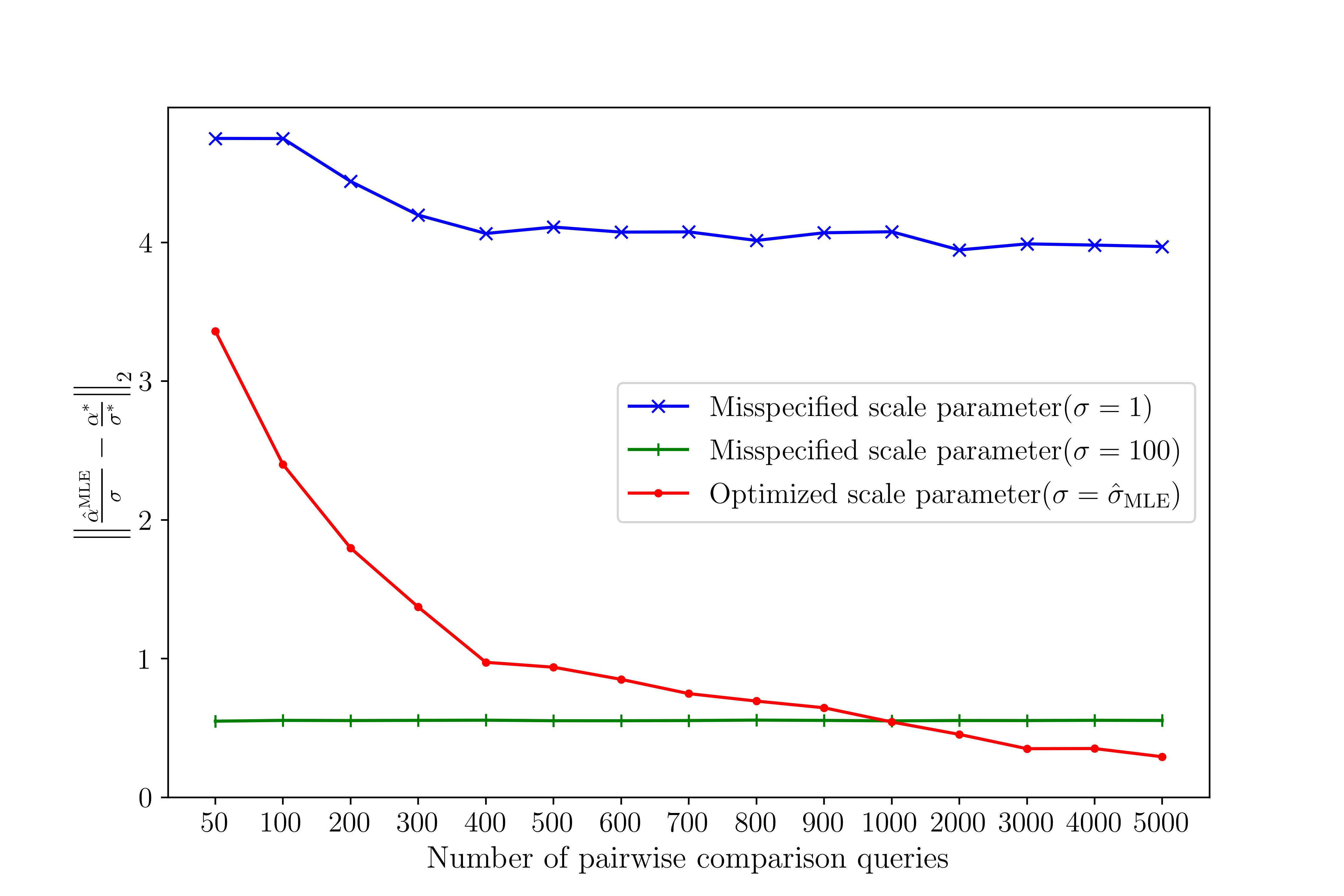}
    }
    \subfigure
    {\includegraphics[width=0.48\linewidth]{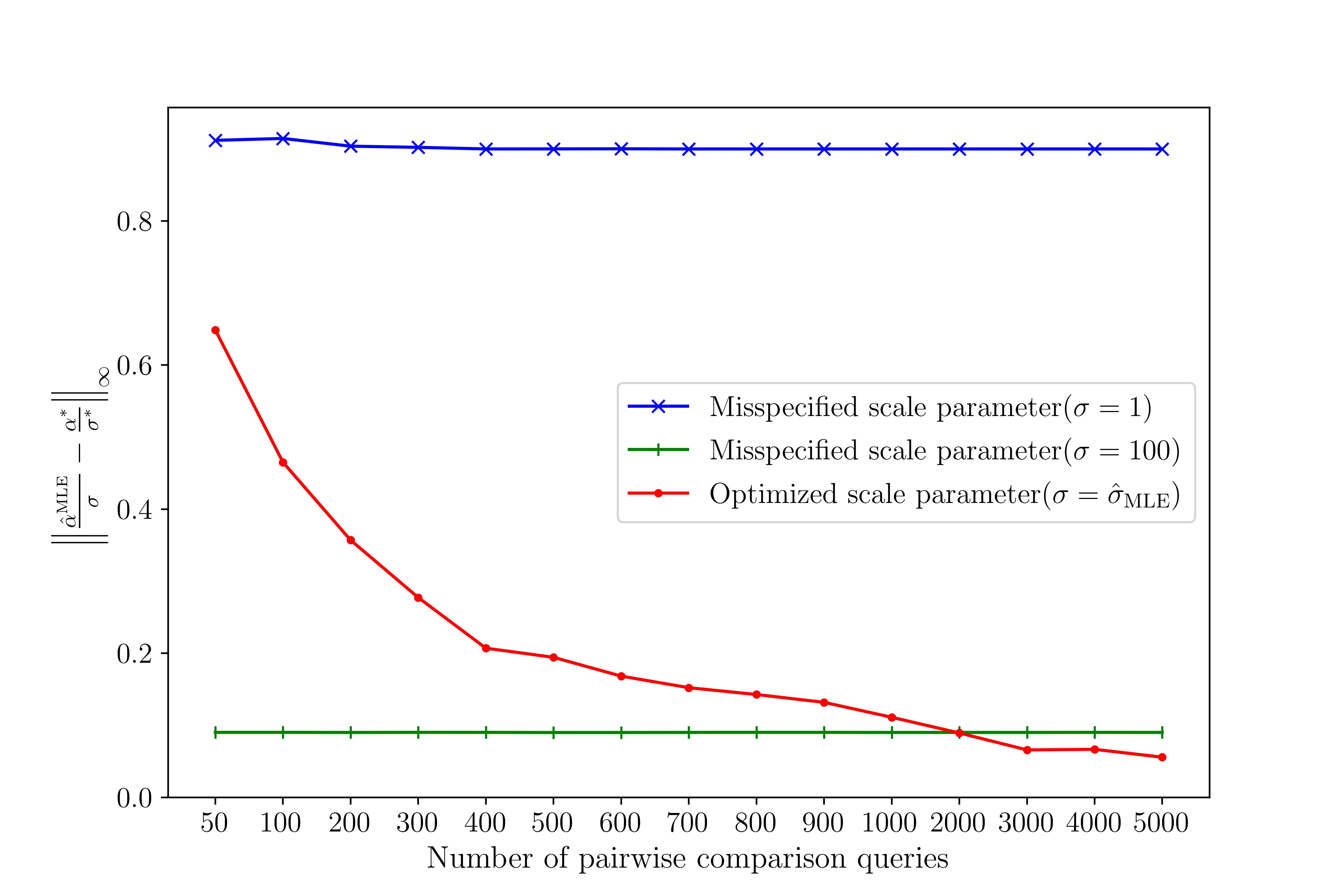}
    }
\caption{$\ell_2$ and $\ell_{\infty}$ {\color{black}statistical} errors between $(\hat{\alpha}^{\rm MLE}, \sigma)$ and $(\alpha^*,\sigma^*)$ with optimized $\hat{\sigma}_{\rm MLE}$, or with fixed
{\color{black}but misspecified}
$\sigma = 1$ and $\sigma = 100$, as the number of pairwise comparison queries increases.} 
\label{fig:num-const}
\end{figure}


\subsection{Effect of incorporating 
{\color{black}structural preference}
information about the VNM utility function}\label{sec:num-prior-info}


In the 
{\color{black}proposed VNM utility model with response error in}
\eqref{eq:DM-random-model}, the true VNM utility function $u^*$ is assumed to be monotonically increasing, concave, and Lipschitz continuous, i.e., $u^*\in\mathcal{U}_c$ as defined in \eqref{eq:U_c}. 
{\color{black}This assumption can be viewed as}
{\color{black}structural preference}
information 
{\color{black}about}
the DM's 
{\color{black}underlying true VNM utility function 
known in advance
}
to the modeler, as discussed in Remark~\ref{remark:prior info}.  
In this test, we examine whether incorporating such 
{\color{black}structural preference}
information 
{\color{black}can}
improve the efficiency of utility elicitation.

We consider four preference elicitation models with different levels of 
{\color{black}structural preference}
information as follows: 1) with all 
{\color{black}structural preference}
information (monotonicity, concavity, and Lipschitz continuity); 2) without Lipschitz continuity; 3) without both Lipschitz continuity and concavity; and 4) without any 
{\color{black}structural preference}
information. 
The last three settings correspond to the cases discussed in Remark~\ref{remark:prior-MLE}.

We generate a series of datasets $D$ with $K$ pairwise comparison queries. 
{\color{black}We set}
$N=50$ and 
{\color{black}consider}
$K\in\{400,500,600,700,800,900, 1000,2000,3000,4000,5000\}$. 
For each $K$, 
{\color{black}we solve the corresponding MLE problems under each of the four model settings with different levels of 
{\color{black}structural preference}
information, and thus obtain}
$(\hat{\alpha}^{\rm MLE}, \hat{\sigma}_{\rm MLE})$. 
{\color{black}We then}
compute the $\ell_2$ and $\ell_{\infty}$ {\color{black}statistical} errors (i.e., $\| \frac{\hat{\alpha}^{\rm MLE}}{\hat{\sigma}_{\rm MLE}}-\frac{\alpha^*}{\sigma^*} \|_{2}$ and $\| \frac{\hat{\alpha}^{\rm MLE}}{\hat{\sigma}_{\rm MLE}}-\frac{\alpha^*}{\sigma^*} \|_{\infty}$).
{\color{black}We average the results}
over 30 independent runs to reduce 
{\color{black}the impact of}
randomness in dataset generation. The results are presented in Figure~\ref{fig:num-func}.

From Figure \ref{fig:num-func}, we can observe that: 
\begin{itemize}
\item 
The $\ell_2$ and $\ell_{\infty}$ {\color{black}statistical} errors of the four models, each with a different level of 
{\color{black}structural preference}
information, all converge as the number of queries increases. 

\item 
{\color{black} The MLE models incorporating more 
{\color{black}structural preference}
information exhibit faster convergence of}
the $\ell_2$ and $\ell_{\infty}$
{\color{black}statistical} errors 
than those {\color{black}incorporating}
less 
{\color{black}structural preference}
information. 

\item 
{\color{black}Removing the Lipschitz continuity constraint has little effect on the convergence rate, as we set 
a large Lipschitz constant $L = 10$, which makes the constraint effectively inactive in this case study}. 

\end{itemize}
These results suggest that incorporating more 
{\color{black}structural preference}
information about the true VNM utility function, such as monotonicity and concavity 
({\color{black}corresponding to}
the DM’s risk-averse attitude) can significantly improve elicitation efficiency. 
{\color{black}In particular, it can}
achieve satisfactory accuracy with fewer pairwise comparison queries and fewer interactions with the DM. 
Since interactions with a real DM in practice are costly, incorporating 
{\color{black}structural preference} 
information is essential {\color{black}for improving the efficiency of preference elicitation}. 

\begin{figure}[htp]
    \centering
    {    \includegraphics[width=0.47\linewidth]{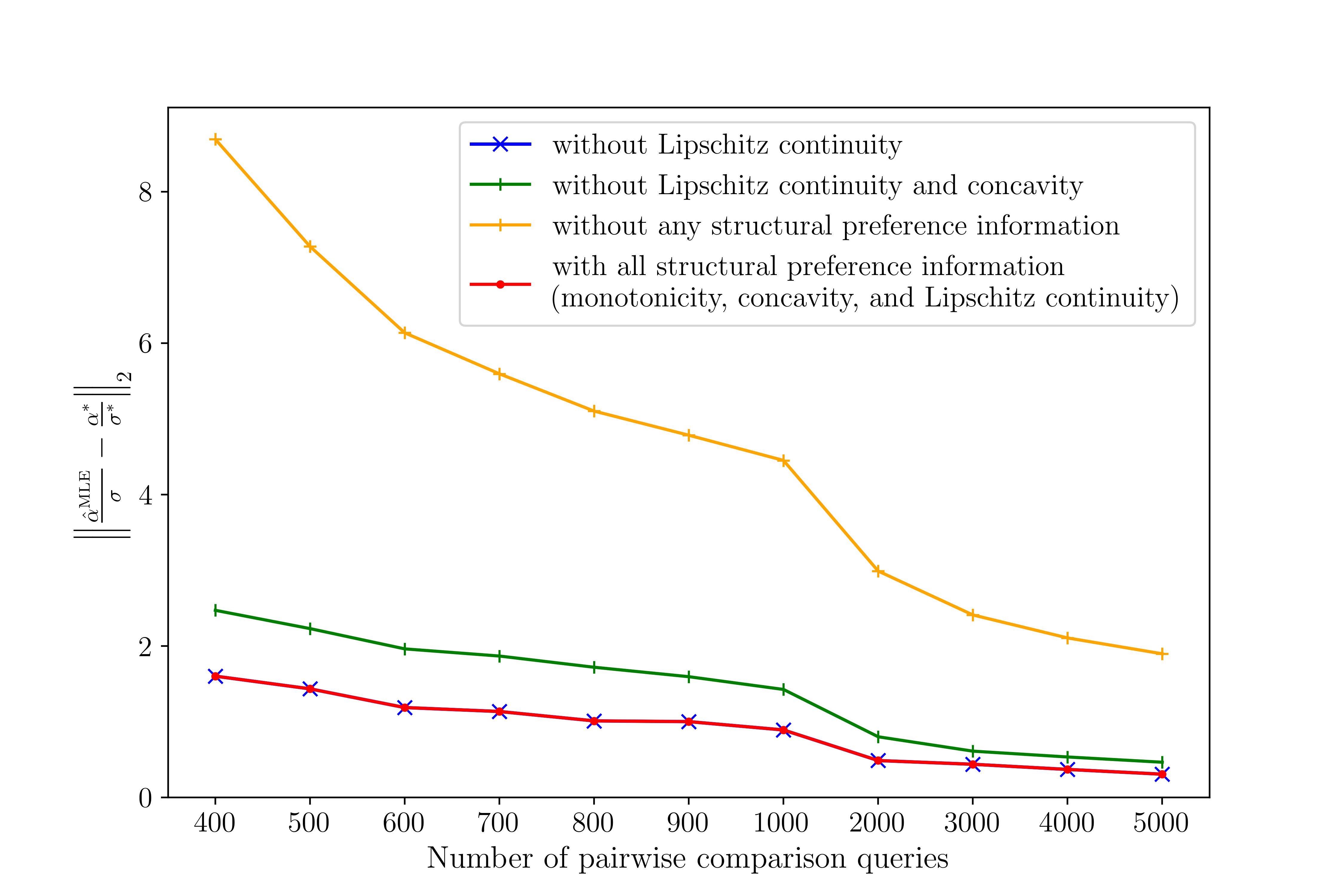}
    }
    { \includegraphics[width=0.47\linewidth]{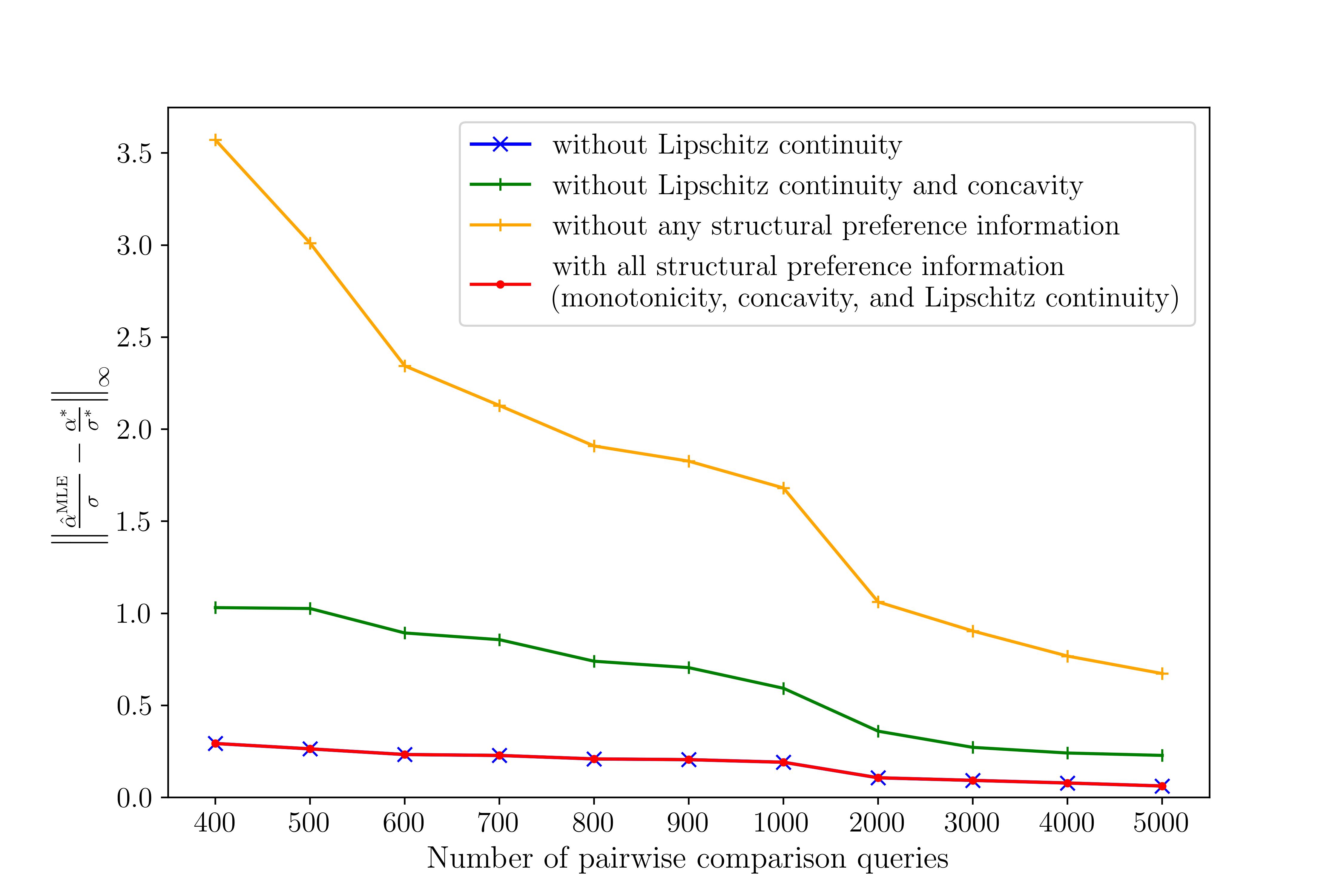}
    }
\caption{$\ell_2$ and $\ell_{\infty}$ {\color{black}statistical} errors of 
{\color{black}MLE-based utility elicitation}
for the four models with different levels of 
{\color{black}structural preference}
information about the DM's 
{\color{black}true VNM utility function},
as the number of pairwise comparison queries increases. }
\label{fig:num-func}
\end{figure}




\subsection{Impact of the structure of the pairwise comparison dataset $D$}\label{sec:num-rank}


{\color{black}As shown in}
Corollary \ref{coll:bound-param}, the minimum eigenvalue $\lambda_{\text{min}}(\Sigma_D)$  of the 
{\color{black}empirical information matrix $\Sigma_D$}
plays an important role in the convergence 
{\color{black}rate}
of the $\ell_2$ 
and $\ell_{\infty}$ {\color{black}statistical} errors (i.e., $\| \frac{\hat{\alpha}^{\rm MLE}}{\hat{\sigma}_{\rm MLE}}-\frac{\alpha^*}{\sigma^*} \|_{2}$ and $\| \frac{\hat{\alpha}^{\rm MLE}}{\hat{\sigma}_{\rm MLE}}-\frac{\alpha^*}{\sigma^*}\|_{\infty}$).  
In this section, we 
{\color{black}investigate this effect}
by considering two groups of datasets: one with full-rank empirical 
{\color{black}information matrices}
($R_D={\color{black} N-1}$ and $\lambda_{\rm min}(\Sigma_D)>0$) and the other with rank-deficient empirical 
{\color{black}information matrices}
($R_D<{\color{black} N-1}$ and $\lambda_{\rm min}(\Sigma_D)=0$).  
{\color{black}We adopt the same true utility function and experimental settings as in Section~\ref{sec:num-setting}.}


{\color{black}For both groups},
we set $N=200$ and $K \in \{50,\allowbreak 100,\allowbreak 150,\allowbreak 200,\allowbreak 250,\allowbreak 300,\allowbreak 400,\allowbreak 500,\allowbreak 600,\allowbreak 700,\allowbreak 800,\allowbreak 900,\allowbreak 1000\}$. 
We then {\color{black}elicit a series of 
MLE-based}
utility functions 
{\color{black}as $K$ increases}
in both dataset
{\color{black}groups}
and compute the corresponding $\ell_2$ and $\ell_{\infty}$ {\color{black}statistical} errors.
{\color{black}All results are}
averaged over 30 independent runs. 
Note that when $K<N$, it is impossible to generate a dataset whose 
{\color{black}information matrix $\Sigma_D$
is full-rank, so the $\lambda_{\rm min}(\Sigma_D)>0$ case starts from $K=200$.}
The results are summarized in Table~\ref{tab:rank_comparison}
{\color{black}and are illustrated in Figure~\ref{fig:matrix-rank}}.


From Figure~\ref{fig:matrix-rank} and Table \ref{tab:rank_comparison}, we find that
for datasets with $\lambda_{\text{min}}(\Sigma_D)=0$
the elicitation errors decrease significantly when $K<200$, but decrease much more slowly 
{\color{black}for larger $K$},
especially when $K>800$. 
{\color{black}For} datasets with $\lambda_{\text{min}}(\Sigma_D)>0$, 
the errors continue to decrease 
{\color{black}as $K$ increases}
and consistently outperform those 
{\color{black}from}
datasets with $\lambda_{\text{min}}(\Sigma_D)=0$.
{\color{black}To achieve the same estimation accuracy, fewer queries are required when the information matrix is full-rank.
With $N$ fixed, the minimum eigenvalue of the information matrix increases as the number of queries $K$ increases.}
This observation is consistent with the theoretical error bounds in Corollary \ref{coll:bound-param}. These findings 
{\color{black}highlight}
the practical value of the theoretical error bounds 
{\color{black}in guiding the design of}
lottery pairs, to improve the efficiency of preference elicitation.

\begin{sidewaystable}[htbp]
\centering
\caption{$\ell_2$ and $\ell_\infty$ {\color{black}statistical} errors 
{\color{black}of MLE-based}
utility elicitation under 
{\color{black}datasets}
with $\lambda_{\text{min}}(\Sigma_D)=0$ and $\lambda_{\text{min}}(\Sigma_D)>0$,
{\color{black}as the}
number of pairwise comparison queries
{\color{black}increases}.
}
\vspace{2mm}
\renewcommand{\arraystretch}{1.7}
\begin{tabular}{c|c|*{13}{>{\color{black}}c}}
\hline
\multicolumn{1}{c|}{} & \multicolumn{1}{c|}{} & \multicolumn{13}{c}{Number of pairwise comparison queries} \\
\hline
Error & Dataset $D$ & \multicolumn{1}{c}{50} & \multicolumn{1}{c}{100} & \multicolumn{1}{c}{150} & \multicolumn{1}{c}{200} & \multicolumn{1}{c}{250} & \multicolumn{1}{c}{300} & \multicolumn{1}{c}{400} & \multicolumn{1}{c}{500} & \multicolumn{1}{c}{600} & \multicolumn{1}{c}{700} & \multicolumn{1}{c}{800} & \multicolumn{1}{c}{900} & \multicolumn{1}{c}{1000} \\
\hline
\multirow{2}{*}{$\left\| \frac{\hat{\alpha}^{\rm MLE}}{\hat{\sigma}_{\rm MLE}}-\frac{\alpha^*}{\sigma^*} \right\|_{2}$} 
& $\lambda_{\text{min}}(\Sigma_D)=0$   & 9.247451 & 6.5284  & 6.2977  & 5.4790  & 4.7928  & 4.7817  & 4.2628  & 4.0856  & 3.5506  & 3.2035  & 2.8541  & 2.6336  & 2.5676 \\
& $\lambda_{\text{min}}(\Sigma_D)>0$ & $\diagdown$ & $\diagdown$ & $\diagdown$ & 4.5739  & 4.0667  & 3.7883  & 3.1078  & 2.7167  & 2.3391  & 2.1182  & 1.8575  & 1.7454  & 1.6218  \\
\hline
\multirow{2}{*}{$\left\| \frac{\hat{\alpha}^{\rm MLE}}{\hat{\sigma}_{\rm MLE}}-\frac{\alpha^*}{\sigma^*} \right\|_{\infty}$} 
& $\lambda_{\text{min}}(\Sigma_D)=0$   & 0.767391 & 0.5517  & 0.5319  & 0.4806 & 0.4297  & 0.4206  & 0.3836  & 0.3605  & 0.3275  & 0.2898  & 0.2544  & 0.2403  & 0.2356 \\
& $\lambda_{\text{min}}(\Sigma_D)>0$ & $\diagdown$ & $\diagdown$ & $\diagdown$ & 0.4066  & 0.3618  & 0.3376  & 0.2879  & 0.2536  & 0.2169  & 0.1954  & 0.1745  & 0.1651  & 0.1541  \\
\hline
\multicolumn{2}{c|}{Minimum eigenvalue ($\lambda_{\text{min}}(\Sigma_D)>0$)} 
& $\diagdown$ & $\diagdown$ & $\diagdown$  & 1.00E\text{-}06 & 7.54E\text{-}04 & 4.41E\text{-}03 & 0.0192  & 0.0598  & 0.1365  & 0.2105  & 0.2985  & 0.3691 & 0.4399  \\
\hline
\end{tabular}
\label{tab:rank_comparison}
\end{sidewaystable}

\begin{figure}[htbp]
\color{black}
    \centering
    \subfigure
    {    \includegraphics[width=0.47\linewidth]{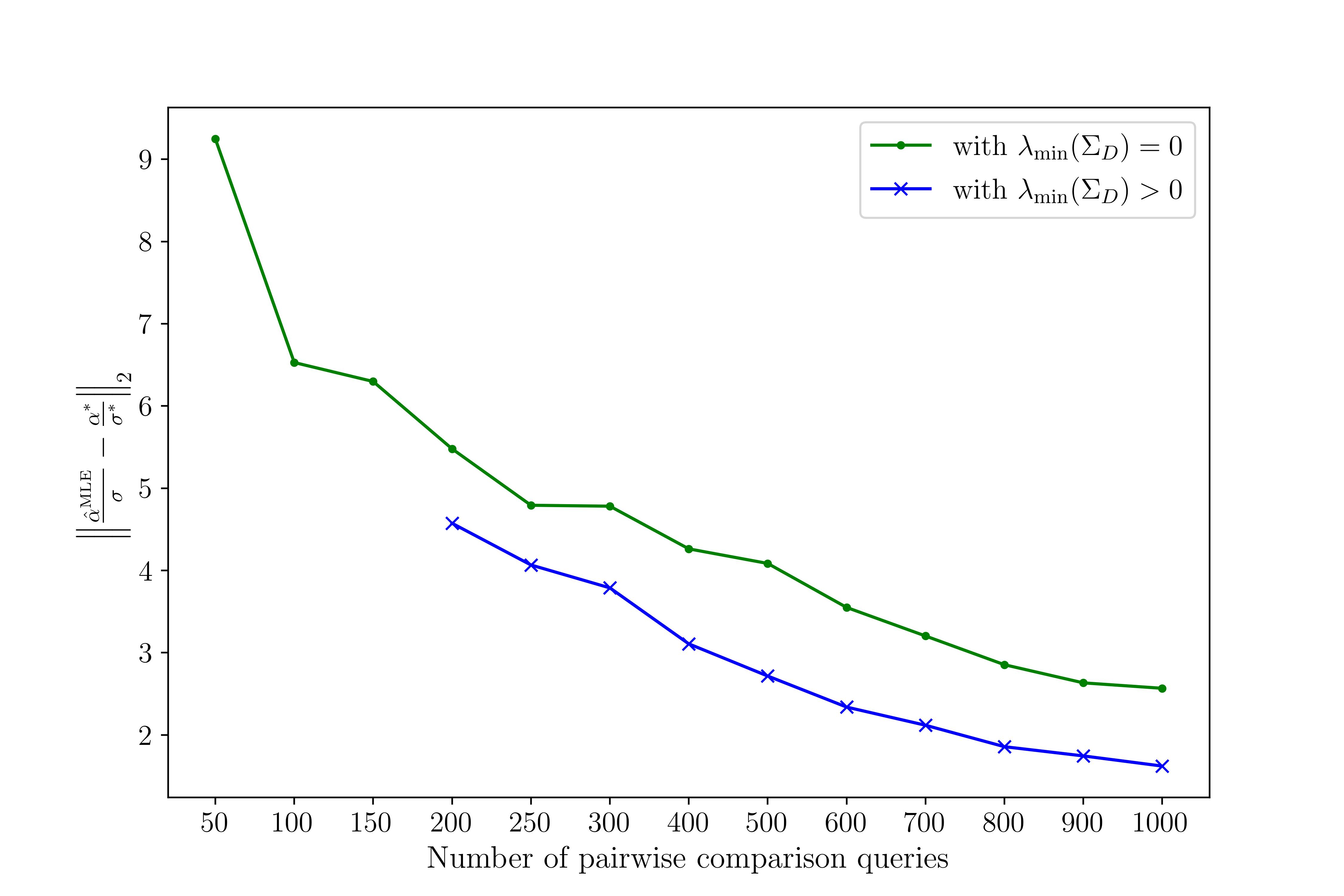}
    }
    \subfigure
    { \includegraphics[width=0.47\linewidth]{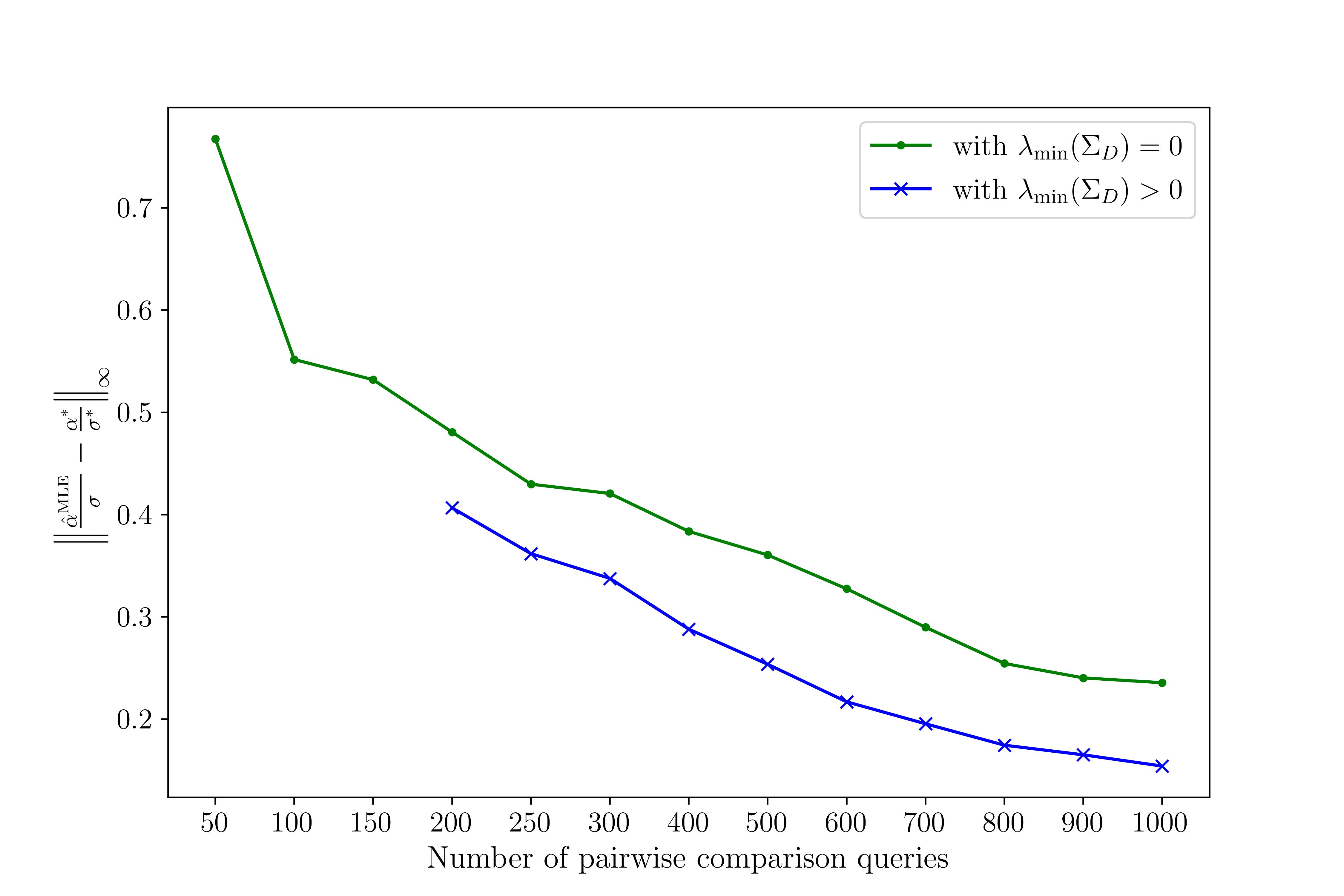}
    }
\caption{$\ell_2$ and $\ell_\infty$ statistical errors of MLE-based utility elicitation under datasets with $\lambda_{\text{min}}(\Sigma_D)=0$ and $\lambda_{\text{min}}(\Sigma_D)>0$, as the number of pairwise comparison queries increases.}
\label{fig:matrix-rank}
\end{figure}

\section{Conclusion}\label{sec-conclude}

In this paper, we propose a preference elicitation approach for the 
non-parametric 
{\color{black}VNM utility model with response error}.
The elicitation is driven {\color{black}by the DM's observed pairwise comparisons and structural preference information known in advance}.
The 
{\color{black}resulting}
MLE problem is reformulated as a convex program. We establish finite-sample error bounds characterizing the convergence of the MLE estimates to the true utility function.
{\color{black}We further apply the estimated utility function to a portfolio optimization problem, 
where the resulting decisions are shown to be robust to both response error and estimation error.}

To 
{\color{black}examine}
the performance of the proposed preference elicitation approach,
we conduct 
{\color{black}some simulation} 
experiments 
with a simulated 
DM, 
using synthetically generated pairwise lottery comparisons. 
The 
results 
{\color{black}indicate}
that the estimated utility function
converges to the 
{\color{black}true utility function}
as the number of pairwise comparison queries increases.
Compared to two benchmark methods,
{\color{black}one of which fixes the}
scale parameter 
and the other 
{\color{black}of which
ignores}
{\color{black}the structural preference}
information, 
our method achieves superior elicitation efficiency. 
{\color{black}We also study how}
the rank of the 
{\color{black}empirical information matrix associated with}
the pairwise comparison dataset 
{\color{black}affects the}
estimation errors, 
{\color{black}providing numerical support for the theoretical error bounds in guiding the design of lottery pairs.}

In this paper, the response error is assumed to be additive to 
the systematic VNM utility function {\color{black}and to follow a parametric Gumbel distribution independent of 
the exogenous uncertainty in the lotteries}.
{\color{black}If these modeling assumptions are mis-specified, the resulting MLE estimates may lose efficiency.}
{\color{black}Exploring alternative specifications of the response error distributions, or modeling its potential dependence on lotteries, 
constitutes a natural direction for future research.
For example, Gaussian additive errors are widely used to model symmetric noise~\cite{train2009discrete,mcculloch1994exact}, leading to the probit model. 
However, unlike the Gumbel case, the corresponding MLE problem under Gaussian distribution generally does not admit a convex reformulation, which may significantly increase computational complexity (see, e.g.,~\cite{kamakura1989estimation,liesenfeld2010efficient,xu2010likelihood}).
It is also of interest to consider non-parametric response error distributions 
(e.g., general marginal distributions studied in~\cite{mishra2014theoretical}) 
or distributional ambiguity within a distributionally robust optimization framework~\cite{nguyen2020distributionally,shafieezadeh2015distributionally}.
However, such extensions would introduce substantial computational and statistical challenges. 
On the computational side, 
losing 
the
parametric structure may preclude a finite-dimensional convex reformulation of the MLE problem. 
On the statistical side, modeling non-parametric or ambiguous response error distributions would lead to infinite-dimensional parameter spaces, 
likely requiring mathematical tools from empirical process theory or functional central limit theory to establish consistency and finite-sample convergence guarantees.

}

{\color{black}In practice, many human DMs may exhibit state-dependent preferences,
whose utility representation varies across states.}
{\color{black}Our current framework assumes only a single invariant VNM utility function $u$ in the systematic component. A natural extension is to consider $U(Y;s)=\mathbb{E}[u(Y;s)] + \varepsilon$, where $s$ denotes an exogenous state or a contextual variable,
and 
$u(\cdot;s)$ captures how the underlying preferences vary across states.
Axiomatic work~\cite{karni1983state} provides theoretical foundations for expected utility with state-dependent preferences, 
and future research could investigate the elicitation of state-dependent, non-parametric VNM utility functions~\cite{liu2025preference} in the presence of response error.
}

{\color{black}
Finally, preference robust optimization (PRO) without response error has been extensively studied in the literature. 
Incorporating the proposed elicitation framework with response error into PRO constitutes a promising research direction. 
In particular, the MLE-based VNM utility function can serve as a nominal center, 
while the finite-sample error bounds may provide guidance for determining the radius of a ball-type ambiguity set. 
Alternatively, the feasible parameter set or an $\epsilon$-optimal solution set of the MLE problem could serve directly as an ambiguity set in future extensions.
}

\backmatter





\section*{Acknowledgements}
{\color{black}The authors would like to thank the associate editor and 
the two anonymous reviewers for their insightful comments and 
constructive suggestions, which help us improve the paper significantly in both contents and style. } 
This work was funded by the National Key R\&D Program of China (No. 2022YFA1004000), National Natural Science Foundation of China (No. 12371324) and Fundamental Research Funds for the Central Universities (No. xzy012024071).



\section*{Declarations}
The authors declare that they have no known competing financial interests or personal relationships that could have appeared to influence the work reported in this paper.	

\bibliography{ref}

\end{document}